\documentclass[11pt]{amsart}
\usepackage{hyperref}
\usepackage{amssymb,amsfonts}
\usepackage{hyperref}
\usepackage{amssymb,textcomp}

\usepackage{enumerate}

\usepackage[utf8]{inputenc}
\usepackage[T1]{fontenc}

\usepackage[mathscr]{eucal}

\usepackage[a4paper, twoside=false, vmargin={2cm,3cm}, includehead]{geometry}

\newtheorem{lemma}{Lemma}[section]
\newtheorem{theorem}[lemma]{Theorem}
\newtheorem*{theorem*}{Theorem}
\newtheorem{corollary}[lemma]{Corollary}

\newtheorem{proposition}[lemma]{Proposition}
\newtheorem*{conjecture*}{Conjecture}
\newtheorem*{proposition*}{Proposition}

\newtheorem{conjecture}{Conjecture}
\newtheorem*{problem}{Problem}

\theoremstyle{definition}
\newtheorem*{claim*}{Claim}
\newtheorem*{convention}{Convention}
\newtheorem*{notation}{Notation}
\newtheorem{definition}[lemma]{Definition}

\newtheorem*{remark}{Remark}
\newtheorem*{remarks}{Remarks}

\newcommand{\lE}{\mathbb{E}^{\log}}

\newcommand{\C}{{\mathbb C}}

\newcommand{\E}{{\mathbb E}}

\newcommand{\N}{{\mathbb N}}
\renewcommand{\P}{{\mathbb P}}
\newcommand{\Q}{{\mathbb Q}}
\newcommand{\R}{{\mathbb R}}
\newcommand{\T}{{\mathbb T}}
\newcommand{\Z}{{\mathbb Z}}

\newcommand{\CB}{{\mathcal B}}

\newcommand{\CF}{{\mathcal F}}

\newcommand{\CO}{{\mathcal O}}
\newcommand{\CP}{{\mathcal P}}

\newcommand{\CX}{{\mathcal X}}
\newcommand{\CY}{{\mathcal Y}}
\newcommand{\CZ}{{\mathcal Z}}

\newcommand{\CK}{{\mathcal K}_{\text{rat}}}

\newcommand{\ug}{\underline g}

\newcommand{\uX}{{\underline X}}
\newcommand{\uG}{\underline G}
\newcommand{\ue}{\underline e}
\newcommand{\uz}{\underline z}
\newcommand{\uGamma}{\underline \Gamma}
\newcommand{\uh}{\underline h}
\newcommand{\umu}{\underline \mu}

\newcommand{\ux}{\underline x}

\newcommand{\bN}{{\mathbf N}}

\newcommand{\bI}{{\mathbf{I}}}

\newcommand{\bmu}{{\boldsymbol{\mu}}}
\newcommand{\blambda}{{\boldsymbol{\lambda}}}

\newcommand{\wt}{\widetilde}
\newcommand{\e}{\mathrm{e}}
\newcommand{\one}{\mathbf{1}}

\newcommand{\Aut}{\mathrm{Aut}}

\newcommand{\nnorm}[1]{\lvert\!|\!| #1|\!|\!\rvert}

\newcommand{\inv}{^{-1}}

\DeclareMathOperator{\id}{id}

\begin{document}

		\title{The  logarithmic Sarnak conjecture for ergodic weights}
		
		\author{Nikos Frantzikinakis}
		\address[Nikos Frantzikinakis]{University of Crete, Department of mathematics, Voutes University Campus, Heraklion 71003, Greece} \email{frantzikinakis@gmail.com}
		\author{Bernard Host}
		\address[Bernard Host]{
			Universit\'e Paris-Est Marne-la-Vall\'ee, Laboratoire d'analyse et
			de math\'ematiques appliqu\'ees, UMR CNRS 8050, 5 Bd Descartes,
			77454 Marne la Vall\'ee Cedex, France }
		\email{bernard.host@u-pem.fr}

		\begin{abstract} The M\"obius disjointness conjecture of Sarnak states that
			the M\"obius function does not correlate  with any bounded sequence of complex numbers arising from a  topological dynamical system with zero topological  entropy. We verify the logarithmically averaged variant of this  conjecture for a large class of systems, which includes  all  uniquely ergodic  systems with zero entropy. One consequence of our results is that the Liouville function has super-linear block growth. 	Our proof uses  a  disjointness argument and the key ingredient  is a structural result for measure preserving  systems naturally associated with the  M\"obius and the Liouville function. We prove that such systems  have no irrational spectrum
			and their building blocks are infinite-step nilsystems and Bernoulli systems.
			  To establish this structural result we
			  make a connection
			  with a problem of purely  ergodic nature via some identities recently obtained by Tao. In addition to an ergodic  structural result of Host and Kra,
			   our analysis is guided by the notion of strong stationarity which was introduced by Furstenberg and Katznelson in the early 90's and naturally plays a central  role in  the structural analysis of measure preserving systems associated with  multiplicative functions. 		
\end{abstract}

		
		
		\maketitle

\section{Introduction and main results }

\subsection{Main results related to the Sarnak conjecture}

Let $\blambda\colon \N\to \{-1,1\}$ be  the Liouville function which  is defined to be $1$ on positive integers with an even number of prime factors, counted with  multiplicity, and $-1$ elsewhere. We extend $\blambda$ to the integers in an arbitrary way, for example by letting $\blambda(-n)=\blambda(n)$ for negative $n\in \Z$ and $\blambda(0)=0$. The M\"obius function $\bmu$
is equal to $\blambda$ on integers  which are not divisible by any square number and is $0$ otherwise.

It is widely believed that
 the values of  the Liouville function and the non-zero values of the M\"obius function fluctuate between
$-1$ and $1$ in such a random way that forces non-correlation with any ``reasonable''  sequence
of complex numbers. This rather vague principle is referred  to  as the ``M\"obius randomness law'' (see \cite[Section~13.1]{IK04}) and   is often used to give   heuristic 
 asymptotics for  various
sums over primes (for examples see \cite{T15b}). The class of ``reasonable'' sequences is expected to include  all  bounded  ``low complexity'' sequences, and in this direction
a precise conjecture that uses the language of dynamical systems
was formulated by  Sarnak in~\cite{S, Sa12}:
\begin{conjecture*}[Sarnak]	Let $(Y,R)$ be  a topological dynamical system\footnote{Meaning that $Y$ is a compact metric space and $R\colon Y\to Y$ is a homeomorphism.} with zero  topological entropy. Then  for every $g\in C(Y)$ and $y\in Y$ we have
	$$
	\lim_{N\to\infty} \frac{1}{N}\sum_{n=1}^N g(R^ny)\, \bmu(n)=0.
	$$
	\end{conjecture*}
This is a fundamental and difficult problem and there is a long  list of partial results that cover a variety of dynamical systems (see Section~\ref{subsec:particular}).
The goal of this article is to verify the conjecture of Sarnak for a  large class of dynamical systems $(Y,R)$, by exploiting mostly the structure of measure preserving dynamical systems generated by  the M\"obius  and the Liouville function
rather than the structure of the topological dynamical  system  $(Y,R)$  for which we have limited information.
The price to pay is that we have to restrict to logarithmic averages rather than the more standard Ces\`aro averages.

We give two variants of our main result, the first
 imposes a global condition on the topological dynamical system:
\begin{theorem}
\label{th:main-new}
	Let $(Y,R)$ be a topological  dynamical system with zero topological entropy and at most countably many  ergodic invariant measures. Then for every  $y\in Y$  and  every $g\in C(Y)$  we have
	\begin{equation}
	\label{eq:main-new}
	\lim_{N\to\infty} \frac 1{\log N}\sum_{n=1}^N\frac{g(R^ny)\, \bmu(n)}{n} =0.
	\end{equation}
		Moreover, a similar statement holds with the Liouville function $\blambda$ in place of $\bmu$.
\end{theorem}
\begin{remark}
			In particular,  our result applies if the system $(Y,R)$ has zero topological entropy
	and  is uniquely ergodic.
	\end{remark}
A rather surprising consequence of the previous result is a seemingly unrelated statement
about the block complexity $P_\blambda(n)$ of the Liouville function which is defined to be  the number of  sign patterns of size $n$  that  are taken by consecutive values of the Liouville function (see Section~\ref{th:Liou-non-linear} for a more formal definition).
 Since the Liouville function is not periodic (because $\lambda(2n)=-\lambda(n)$), it follows from \cite{MH40} that $P_\blambda(n)\geq n+1$ for every $n\in\N$. Moreover, in
  \cite[Proposition~2.9]{MRT15b} it was shown that $P_\blambda(n)\geq n+5$ for every $n\geq 3$ and that these $n+5$ sign patterns are taken on a set of positive upper density of starting points. The Chowla conjecture predicts that  $P_\blambda(n)=2^n$ for every $n\in\N$, equivalently, all possible sign patterns of size $n$ are taken by the Liouville function.  But we are  far from being able to verify this. In fact, it was not even known  that  $P_\blambda(n)$ has super-linear growth, meaning, $\lim_{n\to\infty}P_\blambda(n)/n=\infty$. We verify this property:
\begin{theorem}
	\label{th:Liou-non-linear}
	The  Liouville function has super-linear block growth.
\end{theorem}
\begin{remark}
	In fact, we prove something stronger. If $a\colon \N\to \C$ takes finitely many values and  has linear block growth, then the logarithmic averages of $a(n)\, \blambda(n)$ are $0$. It follows that even if we modify the values of $\blambda$ on a set of logarithmic density $0$, using values taken from a finite set of real numbers, then the new sequence  still has super-linear block growth.
	\end{remark}
Theorem~\ref{th:Liou-non-linear} is deduced from Theorem~\ref{th:main-new} in   Section~\ref{S:linear-growth}.

Another variant of our main result
assumes genericity of  the point defining the weight
sequence for  a zero entropy system that has at most countably many ergodic components:
\begin{theorem}
	\label{th:main-1}
	Let $(Y,R)$ be a topological  dynamical system and $y\in Y$ be   generic for a  measure with  zero entropy and  at most countably many ergodic components.
	Then for every $g\in C(Y)$  we have
	\begin{equation}
	\label{eq:main-1}
	\lim_{N\to\infty} \frac 1{\log N}\sum_{n=1}^N\frac{g(R^ny)\, \bmu(n)}{n} =0.
	\end{equation}
	Moreover, a similar statement holds with the Liouville function $\blambda$ in place of $\bmu$.
\end{theorem}
Genericity of  $y\in Y$ for a Borel probability measure $\nu$ on $Y$  means  that for every $f\in C(Y)$ we have
$\lim_{N\to\infty}	\frac{1}{N}\sum_{n=1}^N  f(R^ny)=\int f\, d\nu$. Our assumption is that the induced system
$(Y,\nu,R)$ has zero entropy and at most  countably many ergodic components.
\begin{remarks}
$\bullet$	A straightforward adaptation of our argument shows that the conclusion of Theorem~\ref{th:main-1}   holds
	for those
	$y\in Y$ that satisfy the following property:  for any sequence  $(N_k)_{k\in\N}$ with $N_k\to \infty$ along which $y$ is quasi-generic for logarithmic averages for some measure $\nu$ (meaning,  $\lim_{k\to\infty} \frac{1}{\log{N_k}} \sum_{n=1}^{N_k}\frac{f(R^ny)}{n}=\int f\, d\nu$ for every $f\in C(Y)$),  the system  $(Y,\nu,R)$  has zero entropy  and  countably many ergodic
	components.
	
	$\bullet$ See Section~\ref{SS:problems} for an example of a topological system and a point which is generic for a zero entropy system with uncountably many ergodic components; in this case our result does not apply.
\end{remarks}

If the ergodic components of the measure in the statement of Theorem~\ref{th:main-1} are assumed to be  totally ergodic, then
we get a much stronger conclusion:
\begin{theorem}
\label{th:main-2}
Let $(Y,R)$ be a topological  dynamical system and
$y\in Y$ be   generic for a  measure $\nu$ with  zero entropy and  at most countably many ergodic components all of which  are totally ergodic.
Then  for every $g\in C(Y)$  that is orthogonal in $L^2(\nu)$ to all $R$-invariant functions
 we have
	\begin{equation}
		\label{eq:main-2}
\lim_{N\to\infty}		\frac 1{\log N}\sum_{n=1}^N\frac{g(R^ny)\, \prod_{j=1}^\ell\bmu(n+h_j)}{n}=0
	\end{equation}
for all $\ell\in\N$ and  $h_1,\dots,h_\ell \in\Z$. Moreover, a similar statement holds with the Liouville function $\blambda$ in place of $\bmu$.
\end{theorem}
\begin{remarks}	
		$\bullet$ Suppose  that the system $(Y,\nu,R)$ is ergodic. Then  for $\ell=2$ and all odd values of $\ell$  the conclusion  holds even if we omit the hypothesis $\int g\, d\nu$=0  assuming that  $h_1\neq h_2$ when $\ell=2$.    Indeed, if $g$  is constant, then \eqref{eq:main-2} holds for $\ell=2$ by \cite{Tao15} and for odd values of $\ell$ by \cite{TT17}. By adding and subtracting a constant we can thus reduce to the zero integral case.
	

$\bullet$  A variant similar to Theorem~\ref{th:main-new} can be proved in the same way:
the conclusion of Theorem~\ref{th:main-2} holds for every $y\in Y$ if $(Y,R)$ has zero topological entropy and at most countably many ergodic invariant measures assuming in addition that  they  are all totally ergodic and the function $g$ is orthogonal in $L^2(\nu)$ to all $R$-invariant functions.

$\bullet$ The   remark following Theorem~\ref{th:main-1} is also valid in this case if we assume in addition that the ergodic components of $(Y,\nu,R)$ are totally ergodic.
\end{remarks}
Theorem~\ref{th:main-2}  is new even in the case where $R$ is given by  an irrational rotation on $\T$ and $g(t):=\e^{2\pi i t}$, $t\in \T$. In this case we have $g(R^n0)=\e^{2\pi i n\alpha}$, $n\in\N$,  for some irrational $\alpha$, and we get the following result as a consequence:
\begin{corollary}
\label{C:main-2}
Let $\alpha\in \R$ be irrational. Then
	\begin{equation} \label{E:irrational}
	\lim_{N\to\infty}		\frac 1{\log N}\sum_{n=1}^N\frac{\e^{2\pi  i n \alpha}\, \prod_{j=1}^\ell\bmu(n+h_j)}{n}=0
\end{equation}
for all $\ell\in\N$ and  $h_1,\dots,h_\ell \in\Z$. Moreover, a similar statement holds with the Liouville function $\blambda$ in place of $\bmu$.	
	\end{corollary}
	\begin{remarks} $\bullet$ For $\ell=1$ the result is well known   and follows from  classical methods of Vinogradov. But even for $\ell=2$ the result is new.
		
$\bullet$	 More generally, if we  apply Theorem~\ref{th:main-2} for $R$  given by appropriate totally ergodic affine transformations on a torus with the Haar measure (as in \cite[Section~3.3]{Fu77}), we get that
	 \eqref{E:irrational} holds with $(\e^{2\pi i n\alpha})_{n\in\N}$ replaced by any sequence of the form  $(\e^{2\pi i P(n)})_{n\in \N}$, where $P\in \R[t]$ has an irrational non-constant coefficient.
	\end{remarks}
 It is straightforward to adapt our arguments in order to strengthen the conclusion in Theorems~\ref{th:main-new}, \ref{th:main-1}, and  \ref{th:main-2}  replacing  $\lim_{N\to\infty}		\frac 1{\log N}\sum_{n=1}^N$ by  $\lim_{N/M\to\infty}		\frac 1{\log (N/M)}\sum_{n=M}^N$.

\subsection{Proof strategy and a key structural result}
A brief description of the proof strategy of  Theorem~\ref{th:main-2}  is as follows (Theorems~\ref{th:main-new}  and~\ref{th:main-1} are proved similarly): In the case where the system $(Y,\nu, R)$ is totally ergodic (the more general case can be treated similarly), we first reinterpret the result as a
statement in ergodic theory about the disjointness of two measure preserving systems. The first is what we call  a
Furstenberg system of  the M\"obius (or the Liouville) function.  Roughly speaking, it  is defined on the  sequence space $X:=\{-1, 0, 1\}^\Z$ with the shift transformation,
 by a measure which assigns to each  cylinder set $\{x\in X\colon x(j)=\epsilon_j,j=-m,\ldots,m\}$  value equal to the logarithmic density of the set $\{n\in \N\colon \bmu(n+j)=\epsilon_j,j=-m\ldots,m\}$, where
$\epsilon_{-m},\ldots, \epsilon_m\in \{-1,0,1\}$ and  $m\in \N$ (we restrict to sequences of intervals along which all these densities exist). The precise definition is given in
Section~\ref{subsec:Furstenberg} and is motivated by analogous constructions made by Furstenberg  in \cite{Fu}. The second system is an  arbitrary  totally ergodic system
with zero entropy. In order to prove  that these two systems are disjoint, we have to understand in some fine detail the structure of all possible Furstenberg systems of the M\"obius and the Liouville   function. Our main structural result is the following (see Sections~\ref{S:back} and \ref{subsec:Furstenberg}  and Appendix~\ref{subsec:infinite-step}   for the definition of the  notions involved):

\begin{theorem}[Structural result]\label{th:main-structure}
A	Furstenberg system  
of  the M\"obius or the Liouville  function  is a factor of a system  that
	\begin{enumerate}
\item 
has no irrational spectrum;
	
	\item  has ergodic components
	 isomorphic to   direct products of infinite-step nilsystems and  Bernoulli systems.
	\end{enumerate}
\end{theorem}
 \begin{remarks}
$\bullet$ We allow the Bernoulli systems and the infinite-step nilsystems to be trivial, in other words,
a direct product of a Bernoulli system and an infinite-step nilsystem is   either a Bernoulli system,   an infinite-step  nilsystem, or  a direct product of both.

$\bullet$ The product decomposition depends on the ergodic component, in particular,  the infinite-step nilsystem depends on  the  ergodic component. On the other hand, our argument allows  us to take the   Bernoulli system to be the same on every ergodic component;   we are not going to   use this property though.

$\bullet$ A related result in a complementary direction was recently obtained in \cite{Fr16}; it states that  if
 a Furstenberg system of  the M\"obius or the Liouville  function is ergodic,  then it is isomorphic to a  Bernoulli system. The tools and the underlying ideas used in the proof of this result are very different and apply to a larger class of multiplicative functions.

$\bullet$ It is not clear to us how to adapt our argument in order to deal with more general bounded multiplicative functions. One would have to find a suitable variant of Proposition~\ref{prop:mu-fator-tilde-mu}  below and to also modify significantly the subsequent analysis.
 \end{remarks}


Using  ergodic theory machinery we prove  (see Part $\text{(ii)}$ of Proposition~\ref{P:disjoint}) that  any system satisfying properties  $\text{(i)}$ and $\text{(ii)}$  of Theorem~\ref{th:main-structure} is necessarily disjoint from every totally ergodic system with zero entropy, leading to a proof of Theorem~ \ref{th:main-2}. The argument used in the proof of Theorems~ \ref{th:main-new} and \ref{th:main-1}  depends on a different disjointness result (see Part  $\text{(i)}$ of Proposition~\ref{P:disjoint})
and this necessitates the use of some additional input from number theory that is contained in \cite{Tao15} in order to verify its hypothesis.

To prove  properties  $\text{(i)}$ and $\text{(ii)}$ of Theorem~\ref{th:main-structure} we combine tools from analytic number theory and ergodic theory.
Our starting point is an identity of Tao (Theorem~\ref{T:Tao1}) which  is implicit
in \cite{Tao15} and enables to express the self-correlations of the M\"obius and the Liouville function as an average of its
dilated self-correlations with prime dilates   (this step necessitates the use of logarithmic averages).  We use  this identity in order to
reduce our problem  to a result of purely  ergodic context. Roughly speaking, it asserts  that if we average the correlations
 of an  arbitrary measure preserving  system  $(X,\mu,T)$ over all prime dilates of its iterates, then the resulting system $(\wt X,\wt \mu, \wt T)$ (see Definition~\ref{D:tilde}), which we call the ``system of arithmetic progressions with prime steps'',   necessarily possesses  properties $\text{(i)}$ and$\text{(ii)}$ (see Theorem~\ref{th:main-mutilde}). Our motivation for establishing this property comes from the case where the ergodic components of the system    $(X,\mu,T)$  are totally ergodic.   It can then be shown that  the  resulting system $(\wt X,\wt \mu, \wt T)$ has additional structure, namely it is strongly stationary
(see Definition~\ref{D:sst}).  The structure of strongly stationary systems
was completely determined in \cite{J} and \cite{Fr04}, where it was shown  that they satisfy properties $\text{(i)}$ and $\text{(ii)}$  of Theorem~\ref{th:main-structure}. Unfortunately, we do not know how to establish  total ergodicity  of the ergodic components of Furstenberg systems of the  Liouville function (for the M\"obius function this property is not even true).  In order to overcome this obstacle we
 use  a more complicated line of arguing which we briefly describe next.

To prove that the system  $(\wt X,\wt \mu, \wt T)$ enjoys property $\text{(ii)}$  we initially use   a structural result of Host and Kra (see Theorem~\ref{T:HK} and Corollary~\ref{C:HK} in the Appendix)  and an ergodic  theorem (see Theorem~\ref{T:FHK}) in
order  to reduce  the problem to the case where the system
$(X,\mu,T)$ is an ergodic infinite-step nilsystem (see Lemma~\ref{L:mutilde-infty2}).  In this  case, we show (see Proposition~\ref{P:infnil})
 that  the ergodic components of the system  $(\wt X,\wt \mu, \wt T)$ are infinite-step nilsystems. Essential role in this part of the argument plays the theory of arithmetic progressions on nilmanifolds which we briefly review in Appendix~\ref{A:AP}. The details are given in Section~\ref{S:structure}.

The key ingredient in the proof of property $\text{(i)}$ is to  establish that the system $(\wt X,\wt \mu, \wt T)$
 satisfies a somewhat weaker property than strong stationarity, roughly speaking, it is  an inverse limit of  partially strongly stationary systems (a notion defined in Definition~\ref{D:sst}).
  We then adjust an argument of Jenvey \cite{J} in order to show that such systems do not have irrational spectrum. The details are given in Section~\ref{S:sst}.



 Finally, we  briefly record the input from analytic number theory needed to carry out our analysis:
The structural result of Theorem~\ref{th:main-structure}  uses some identities of Tao for the M\"obius and the  Liouville function which are implicit in \cite{Tao15}  and were obtained
from first principles using techniques  from probabilistic number theory. It also
 uses indirectly (via the use of Theorems~\ref{T:FHK} and \ref{T:ergid2} in various places) the Gowers uniformity of the $W$-tricked von Mangoldt function which was established in   \cite{GT09b, GT7, GTZ12}.  Theorem~\ref{th:main-2} does not use any other tools from  number theory. Theorems~\ref{th:main-new}, \ref{th:Liou-non-linear}, and \ref{th:main-1} use, in addition
to the previous number theoretic  tools,   a recent result of Tao~\cite{Tao15} on the two-point correlations of the Liouville function which in turn depends upon a recent result  of Matom\"aki and  Radziwi{\l}{\l}~\cite{MR15} on averages of the M\"obius and the Liouville function on short intervals. This additional input from number theory is used in order to verify that  on any   Furstenberg system of the   M\"obius (resp. Liouville) function, a function naturally associated to $\bmu$ (or $\blambda$) is orthogonal to the rational Kronecker factor of the system; this is needed in order to verify  the hypothesis  of the disjointness result stated in  Part~$\text{(i)}$ of Proposition~\ref{P:disjoint}.

\subsection{Comparison with existing results}
\label{subsec:particular}
We say that a topological dynamical system $(Y,R)$ \emph{satisfies the Sarnak conjecture} if for every continuous function $g$ on $Y$ and every $y\in Y$, the Ces\`aro averages
$$
\frac 1N\sum_{n=1}^N g(R^ny)\, \bmu(n)
$$
 tend to $0$ as $N\to\infty$. We say that $(Y,R)$ \emph{satisfies the logarithmic Sarnak conjecture} if the same property  holds with the logarithmic averages
$$ \frac 1{\log N}\sum_{n=1}^N\frac{g(R^ny)\, \bmu(n)}{n}
$$
in place of the Ces\`aro averages.  Note that the Sarnak conjecture for a system implies  the logarithmic Sarnak conjecture for the same  system.

The Sarnak conjecture has been proved for a variety of systems,
for example nilsystems~\cite{GT7},  some horocycle flows~\cite{BoSZ} and more general zero entropy systems arising from homogeneous dynamics~\cite{Pec}, certain  distal systems, in particular some  extensions of a rotation by a
torus~\cite{KL, LS, W17}, a large class of  rank one transformations~\cite{ELdlR, Bo, FM}, systems generated by  various substitutions  \cite{EKL,DDM,FKLM,MRi15}, all  automatic sequences~\cite{Mu16},
some interval exchange transformations~\cite{Bo,CE,FM}, some systems of number theoretic origin~\cite{Bo2,Gr}, and more... The survey article \cite{FKL17} contains an up to date list of relevant bibliography.
In most cases  the systems under consideration  are uniquely ergodic. The proof techniques  vary a lot since they make essential use of  special properties  of the system at hand.
However, in many cases, the proof is based upon a Lemma of K\'atai~\cite{Ka}, in a way introduced in~\cite{BoSZ}, and our method is completely different.

 Theorems~\ref{th:main-new}  and \ref{th:main-1}  in this article allow  one to deal with the  vastly  more general  class of zero entropy topological dynamical systems which are uniquely ergodic or have at  most countably many  ergodic invariant measures.
 The price to pay is  that  we  cover only the logarithmic variant of Sarnak's conjecture. Modulo this shortcoming,   Theorems~ \ref{th:main-new} and \ref{th:main-1}    cover most of the   systems cited above and can be used to handle   a wide variety of new systems. We briefly give a non-exhaustive list of examples covered by our main results:

\subsubsection*{Systems with countable support.} If $Y$ is a countable set, then
  the system $(Y,R)$ has  at most countably many ergodic invariant measures all of them giving rise to periodic systems. Hence, Theorem~\ref{th:main-new} applies and shows that the system $(Y,R)$ satisfies the logarithmic Sarnak conjecture (the same conclusion can also be obtained using \cite[Theorem~1.4]{HWZ16} which deals with  Ces\`aro averages). In particular, this implies that the support of the subshift generated by the Liouville function is an uncountable set, and this is true even  if we change the values of the Liouville function on a set of logarithmic density $0$.

\subsubsection*{Homogeneous dynamics.}  Nilsystems and several horocycle flows  have zero entropy and    every point is generic for an ergodic measure, hence
Theorem~\ref{th:main-1} applies. The same holds for more general unipotent actions on homogeneous spaces of connected Lie groups.

\subsubsection*{Some distal systems.}  Our result  applies  for a wide family of topological distal systems. For example, suppose that  $(W,T)$ is a uniquely ergodic   system and  $(Y,R)$ is built from $(W,T)$ by a sequence of compact group extensions in the topological sense. Then the transformation  $R$ admits a ``natural'' invariant measure $\nu$ and if $(Y,\nu,R)$ is ergodic, then $(Y,R)$ is uniquely ergodic~\cite[Proposition~3.10]{Fu}, and Theorem~\ref{th:main-new} applies.

\subsubsection*{Rank one transformations.} Strictly speaking, rank one systems are defined in a pure measure theoretical setting, but they have a natural topological model. Most of these models (including those considered in the bibliography cited above) are uniquely ergodic and have zero topological entropy,
hence, Theorem~\ref{th:main-new} applies.

\subsubsection*{Subshifts with linear block growth.}
Let $(Y,R)$ be  a transitive subshift  with linear block growth (see Section~\ref{S:linear-growth}). Then $(Y,R)$ has zero topological entropy and  by Proposition~\ref{prop:linear-grow} it admits only finitely many ergodic invariant measures (for minimal subshifts this result was already known~\cite{Bos}). Hence, Theorem~\ref{th:main-new} applies  and shows that it satisfies the logarithmic Sarnak conjecture.
We use this fact in the proof of Theorem~\ref{th:Liou-non-linear}.

\subsubsection*{Substitution dynamical systems.}
Theorem~\ref{th:main-new}  applies to  all systems of  \emph{primitive substitutions}~\cite{Q} with not necessarily  constant length,
 because they have zero topological entropy and are uniquely ergodic.

\subsubsection*{Interval exchange transformations.} All
 interval exchange transformations have zero entropy and  minimality of  the interval exchange (which is equivalent to the non-existence of a point with a finite orbit)   implies that it has a finite number of  ergodic invariant measures \cite{Kat,Ve78}.
Hence, Theorem~\ref{th:main-new} applies and shows that all  minimal interval exchange transformations
satisfy the logarithmic Sarnak conjecture.

 \subsubsection*{Finite rank  Bratteli-Vershik dynamical systems.} More generally, Theorem~\ref{th:main-new} applies to all finite rank  Bratteli-Vershik dynamical systems~\cite{BDM2} (minimality is part of their defining properties) because they have zero  entropy and finitely many ergodic invariant measures. This class contains all the examples
 mentioned  in the previous two classes.

\medskip

Although the class of topological dynamical systems to which Theorem~\ref{th:main-2}  applies is more restrictive (due to our total ergodicity assumption) it is still large.
For instance, totally ergodic nilsystems, several  horocycle flows,   several distal systems as the ones described above,
some classical rank one transformations (for example the Chacon  system),
 and typical interval exchange transformations, have zero topological entropy and are uniquely ergodic and totally ergodic, hence
Theorem~\ref{th:main-2}
applies.

\subsection{Further comments and some conjectures}\label{SS:problems}
Theorems~\ref{th:main-new}, \ref{th:main-1}, and \ref{th:main-2} deal with logarithmic averages rather than the more standard Ces\`aro averages. This is necessary for our proof since on the first step of our argument we use the identities of Tao stated in  Theorem~\ref{T:Tao1}, and these are only known in a form useful to us for logarithmic averages.

If one  shows that  Furstenberg systems of the  Liouville function have no  rational spectrum except $1$,
 then
   Theorem~\ref{th:main-2} can be proved  for the Liouville function  for any $y\in Y$ that is generic  for a measure $\nu$ such that the system $(Y,\nu,R)$ has zero entropy and  at most countably many ergodic components and every $g\in C(Y)$ that is orthogonal in $L^2(\nu)$ to all $R$-invariant functions.

Theorem \ref{th:main-1}  handles the case where   a point $y\in Y$ is generic (or quasi-generic) for a measure $\nu$ such that the system $(Y,\nu,S)$ has zero entropy  and  at most countably many ergodic components. But if  $(Y,\nu,S)$ has
uncountably many ergodic components,   our argument falls apart.
A particular instance is the following one: Let $(\alpha_k)_{k\in\N}$ be a sequence that is equidistributed in $\T$
and suppose  that the finite sequences  $(n\alpha_k)_{n\in [k^2,(k+1)^2)}$, $k\in \N$,  are asymptotically equidistributed in $\T$ as $k\to\infty$, meaning, $\lim_{k\to\infty}\frac{1}{2k+1}\sum_{k^2\leq n< (k+1)^2}f(n\alpha_k)= \int f\, dm_\T$ for every $f\in C(\T)$.
We let
$$
y_0(n):=\sum_{k=1}^\infty \e^{2\pi i n \alpha_k}\,
\one_{[k^2,(k+1)^2)}(n),\quad n\in\N,
$$
and $y_0(n):=1$ for $n\leq 0$.
Let  $\mathbb{S}$ be the unit circle, $Y=\mathbb{S}^\Z$, $R\colon Y\to Y$ be the shift transformation,  and  let $g\in C(Y)$ be defined by $g(y):=y(0)$ for $y\in Y$. Note that  $y_0(n)=g(R^ny_0)$ for every  $n\in\Z$. We claim that the point $y_0\in Y$ is generic for some invariant measure $\nu$ on $Y$ and that the system $(Y,\nu,R)$ is measure-theoretically isomorphic to the system $(\T^2,m_{\T^2},T)$ where $m_{\T^2}$ is the Haar measure of $\T^2$ and $T\colon \T^2\to\T^2$ is defined by
$$
T(s,t):=(s,t+s),\quad s,t\in\T.
$$
Assuming the claim for the moment, we easily  conclude that the system $(Y,\nu,R)$ has zero entropy, no  eigenvalue other than $1$,   uncountably many ergodic components, and is disjoint from
every ergodic system.  Our methods do not allow   us to prove that this system is  disjoint from  Furstenberg systems of the M\"obius or the Liouville function or that the logarithmic averages of $y_0(n)\, \bmu(n)$ or  $y_0(n)\, \blambda(n)$ are $0$.

To prove the claim,  define the map $\phi\colon\T^2\to\mathbb{S}$ by $\phi(s,t):=\e^{2\pi i t}$, for  $s,t\in \T,$ and the map $\Phi\colon\T^2\to Y$ by $(\Phi(s,t))(n)=\phi(T^n(s,t)):=\e^{2\pi i(t+ns)}$ for $n\in\Z$, $s,t\in \T$.
We have $\Phi\circ T=R\circ\Phi$ and the image $\nu$ of the measure $m_{\T^2}$ under $\Phi$ is invariant under $R$.
Moreover, $\phi(T(s,t))\, \overline{\phi(s,t)}=\e^{2\pi is}$ and it follows that $\Phi$ is one to one and  thus is an isomorphism from $(\T^2,m_{\T^2},T)$ to $(Y,\nu,R)$. It remains to show that the point $y_0$ is generic for the measure $\nu$.
 For $m\in \N$ let  $\ell_{-m},\ldots, \ell_m\in\Z$ and define
$$
F(y):=\prod_{j=-m}^m y(j)^{\ell_j} \quad \text{ for } \ y=(y(n))_{n\in\Z}\in Y.
$$
Then by a direct computation it is not hard to verify that
\begin{multline*}
\lim_{N\to\infty}\frac{1}{N} \sum_{n=1}^N F(R^ny_0)
=
\lim_{N\to\infty}\frac{1}{N} \sum_{n=1}^N\prod_{j=-m}^m y_0(n+j)^{\ell_j}\\
=\int_{\T^2}\prod_{j=-m}^m\e^{2\pi i(t+js) \ell_j}\,ds\,dt
=\int_{\T^2}F\circ\Phi\,dm_{\T^2}
=\int_Y F\,d\nu.
\end{multline*}
By linearity and density, the same formula holds for every continuous function $F$ on $Y$ and the claim follows.

We would also like to remark that it is consistent  with existing knowledge (though highly unlikely) that some  Furstenberg system of the Liouville function is isomorphic to the low complexity system $(\T^2,m_{\T^2},T)$ described above. Here is a related problem:
\begin{problem}
	Let $\phi\colon \T\to \{-1,1\}$ be the function defined by $\phi(t):={\bf 1}_{[0,1/2)}(t)-{\bf 1}_{[1/2,1)}(t)$. Show that the following identity cannot hold:
	$$
	\lim_{N\to\infty}\frac{1}{N}\sum_{n=1}^N \prod_{j=1}^\ell\blambda(n+h_j)=
	\int_{\T^2}\prod_{j=1}^\ell\phi(t+h_js)\, dt \, ds
	$$
	for all $\ell\in \N$ and $h_1,\ldots, h_\ell\in \Z$.
\end{problem}
In the initial step of our argument (Proposition~\ref{prop:mu-fator-tilde-mu}) we make  essential use of the fact  that $\bmu$ and  $\blambda$ are equal to $-1$   on   the primes. But
we expect  the conclusion of   Theorem~\ref{th:main-structure} to remain
valid even when one uses an arbitrary  multiplicative function $f\colon \N\to [-1,1]$ in place of $\bmu$ and $\blambda$. In fact, we expect ergodicity in all cases and  we conjecture the following:
\begin{conjecture}\label{Conj1}
	Every multiplicative function  $f\colon \N\to [-1,1]$  has a unique Furstenberg system.\footnote{Equivalently, the point $(f(n))_{n\in\N}$ is generic for some measure on the sequence space $[-1,1]^\N$.} This
	system   is ergodic and  isomorphic  to the
	direct product of a Bernoulli system and an ergodic odometer.\footnote{An ergodic odometer is an ergodic inverse limit of periodic systems, or equivalently, an ergodic system $(X,\mu,T)$ for which   the rational eigenfunctions span a dense subspace of $L^2(\mu)$.}
\end{conjecture}
Note that all three possibilities can occur, for example it is known that the Furstenberg system of $\bmu^2$ (called the square-free system)
is  an ergodic odometer \cite{CS}, and conditional to the Chowla conjecture it is known that the Furstenberg system of the Liouville function $\blambda$ is isomorphic to a Bernoulli system and the Furstenberg system of  the M\"obius function $\bmu$ is a relatively Bernoulli extension over the procyclic factor induced by  $\bmu^2$ (see \cite[Lemma~4.6]{AKLR14}).

How do we then distinguish
 (at least conjecturally) between the possible structures of the
Furstenberg system of a  multiplicative function $f\colon \N\to [-1,1]$?  It seems easier to do this  when $f$ takes values  in $\{-1,1\}$ in which case we expect the following dichotomy:
\begin{conjecture}\label{Conj2}
The Furstenberg system of a multiplicative function  $f\colon \N\to \{-1,1\}$ is either a  Bernoulli system
or an ergodic odometer. Moreover, it is a Bernoulli system if and only if   $f$ is aperiodic.
\end{conjecture}
Aperiodicity, which is also often referred to as non-pretentiousness,  means that
the averages $\frac{1}{N}\sum_{n=1}^N f(an+b)$ converge to $0$ as $N\to \infty$
for all $a,b\in \N$.
 It can be shown that  the Furstenberg system of a zero mean multiplicative function $f\colon \N\to \{-1,1\}$ is Bernoulli if and only if  all   multiple correlations of  distinct shifts of $f$  vanish. When one works with logarithmic averages, Tao showed in \cite{T1} (when $f=\blambda$
but his argument applies with some modifications for general multiplicative $f\colon\N\to\{-1,1\}$, see~\cite[Theorem~1.8]{Fr16})
that this is equivalent to asserting that $f$ satisfies the Sarnak conjecture.
So for multiplicative functions $f\colon \N\to \{-1,1\}$, aperiodicity, Bernoullicity of the corresponding Furstenberg system, $f$ satisfies the logarithmic Chowla conjecture, and $f$ satisfies the logarithmic Sarnak
conjecture,
are expected to be equivalent properties. Of course, none of the last three properties is known unconditionally
 even for the Liouville function (only aperiodicity is known).

\subsection{Notation and conventions} For readers convenience, we gather here   some notation used throughout the article.

 We write $\T=\R/\Z$ and $\mathbb{S}$ for the unit circle. For $t\in\R$ or $\T$ we write $\e(t):=\e^{2\pi it}$.

We denote by $\N$ the set of positive integers and by $\P$ the set of prime numbers.
For $N\in \N$ we denote by $[N]$ the set $\{1,\ldots,
N\}$. Whenever we write $\bN$ we mean a sequence of intervals of integer $([N_k])_{k\in\N}$ with $N_k\to\infty$.

Unless otherwise specified,  with $\ell^\infty(\Z)$ we denote the space of all bounded,
real valued, doubly infinite
sequences.

If $A$ is a finite non-empty set we let  $\E_{n\in A}:=\frac{1}{|A|}\sum_{n	\in A}$.

With $(Y,R)$ we denote  the topological dynamical system used to define the weight in the
formulation of Theorems~\ref{th:main-new}, \ref{th:main-1}, and   \ref{th:main-2};  it sometimes comes equipped with an $R$-invariant measure $\nu$.

With $(X,\mu,T)$ we denote a  Furstenberg system
of the M\"obius or the Liouville function, and we also use the same notation when we study properties of abstract measure preserving systems.

With $(X^\Z,\wt \mu, S)$ we denote  the system of arithmetic progressions with prime steps associated with a system $(X,\mu,T)$.

\subsection{Acknowledgement} We would like to thank F.~Durand, B.~Kra,  M.~Lema\'nczyk, and P.~Sarnak  for useful remarks.
 We also thank M.~Lema\'nczyk and T. de la Rue for pointing out a correction in Theorem 1.4 and Corollary 3.13. The second author thanks the CMM -- Universitad de Chile for its hospitality and support.

\section{Background in ergodic theory}\label{S:back}
We gather here some basic background in ergodic theory and  related notation used throughout the article.

\subsubsection*{Topological dynamical systems}
A \emph{topological dynamical system} $(X,T)$ is a compact metric space endowed with a homeomorphism $T\colon X\to X$. It is \emph{topologically transitive} if it has at least one dense orbit under $T$, and it is \emph{minimal} if each orbit is dense.

If $(X,T)$ and $(Y,S)$ are two topological dynamical systems, then the second system is a {\em factor} of the first if
there exists a  map $\pi\colon X\to Y$, continuous  and onto,  such that $S\circ\pi(x) = \pi\circ T(x)$ for every $x\in
X$. If the factor map $\pi$ is injective, then the two systems are {\em isomorphic}.

\subsubsection*{Measure preserving systems}
Throughout the article, we make the standard assumption that all probability  spaces $(X,\CX,\mu)$ considered are Lebesgue, meaning, $X$  can be given the structure of a compact metric space
and $\CX$ is its Borel $\sigma$-algebra.
 A {\em measure preserving system}, or simply {\em a system}, is a quadruple $(X,\CX,\mu,T)$
where $(X,\CX,\mu)$ is a probability space and $T\colon X\to X$ is an invertible, measurable,  measure preserving transformation.
We often omit the $\sigma$-algebra $\CX$  and write $(X,\mu,T)$. Throughout,  for $n\in \N$ we denote  with $T^n$   the composition $T\circ  \cdots \circ T$ ($n$ times) and let $T^{-n}:=(T^n)^{-1}$ and $T^0:=\id_X$. Also, for $f\in L^1(\mu)$ and $n\in\Z$ we denote by  $T^nf$ the function $f\circ T^n$.

\subsubsection*{Factors and isomorphisms}
A {\it homomorphism},  also called a \emph{factor map}, from a system $(X,\CX,\mu, T)$ onto a
system $(Y, \CY, \nu, S)$ is a measurable map $\pi\colon X\to Y$,
such that
$\mu\circ\pi^{-1} = \nu$ and with $S\circ\pi = \pi\circ T$ valid  $\mu$-almost everywhere.
 When we have such a homomorphism we say that the system $(Y, \CY,
\nu, S)$ is a {\it factor} of the system $(X,\CX,\mu, T)$.  If
the factor map $\pi\colon X\to Y$ is invertible\footnote{Meaning that there exists a factor map $Y\to X$, written $\pi\inv$, with $\pi\inv\circ\pi=\id_X$ valid $\mu$-almost everywhere (this implies that $\pi\circ\pi\inv=\id_Y$ holds $\nu$-almost everywhere).}
we say that $\pi$ is an \emph{isomorphism} and that
 the systems $(X,\CX, \mu, T)$ and $(Y, \CY, \nu, S)$
are {\it isomorphic}.

If $\pi\colon(X,\CX,\mu, T)\to(Y, \CY, \nu, S)$ is a factor map and $\phi\in L^1(\mu)$, the  function
$\E_{\mu}(\phi\mid Y)$ in $L^1(\nu)$ is determined by the property
	$\int_{A}\E_{\mu}(\phi\mid Y)\, d\nu=\int_{\pi^{-1}(A)} \phi\, d\mu$ for every $A\in \CY$.

If $\pi\colon(X,\CX,\mu, T)\to(Y, \CY, \nu, S)$ is a factor map, then $\pi^{-1}(\CY)$  is a $T$-invariant sub-$\sigma$-algebra  of $\CX$. Conversely,
 for any $T$-invariant sub-$\sigma$-algebra $\CY'$ of
$\CX$ there exists a factor map $\pi\colon(X,\CX,\mu, T)\to(Y, \CY, \nu, S)$ with $\CY'=\pi\inv(\CY)$ up to $\mu$-null sets.
This factor is unique up to isomorphism  and we call it
{\em the factor associated with (or induced by) $\CY'$}.
 See \cite[Section~2.3]{Wa82} or \cite[Section~6.2]{EW11} for details. When there is no danger of confusion, we may abuse notation and denote the transformation  $S$ on $Y$ by $T$.
We  pass constantly from invariant sub-$\sigma$-algebras to factors, the convention being that the
factors associated to the $\sigma$-algebras $\CY,\CZ,\dots$, are written $Y,Z,\dots$.


We will sometimes abuse notation and use the sub-$\sigma$-algebra $\mathcal{Y}$ in place of
the subspace $L^2(X,\CY,\mu)$. For example, if we write that a function is orthogonal to $\mathcal{Y}$,
we mean  that  it  is orthogonal to the subspace $L^2(X,\CY,\mu)$.

\subsubsection*{Spectrum}
  Let $(X,\mu,T)$ be a  system.
 For $t\in \T$, we say that $\e(t)$ is an \emph{eigenvalue} of the system if there exists a non-identically zero function $f\in L^2(\mu)$ such that $Tf=\e(t)f$, in which case we  say that $f$ is an \emph{eigenfunction} associated to the eigenvalue $\e(t)$. We call the eigenvalue $\e(t)$ \emph{rational} if $t$ is rational and \emph{irrational} otherwise.
The \emph{spectrum} of the system   is the subset of $\T$ consisting  of all eigenvalues,  and we define the \emph{rational} and the \emph{irrational spectrum} to be the subset of the  spectrum consisting  of  all rational (resp. irrational) eigenvalues.
With $\CK(T)$ we denote the  {\em rational Kronecker factor}  of $(X,\CX,\mu,T)$, it is  the smallest $T$-invariant sub-$\sigma$-algebra of $\CX$ with respect to which   all  eigenfunctions
with rational eigenvalues are measurable.  The linear span of these eigenfunctions is
 dense in $L^2(X,\CK(T),\mu)$.

\subsubsection*{Ergodicity and ergodic decomposition}
A system $(X,\mu,T)$ is {\em ergodic} if  all functions $f\in L^1(\mu)$ which satisfy  $Tf=f$ are constant. It is
{\em totally ergodic} if $(X,\mu,T^d)$ is ergodic for every $d\in \N$, equivalently, if it is ergodic and has no rational spectrum except $1$.

Let $(X,\CX,\mu,T)$ be a system and let $\pi\colon(X,\CX,\mu,T)\to(\Omega,\CO,P,T)$ be the factor map associated to the  $\sigma$-algebra of $T$-invariant sets of $X$. Then the disintegration of $\mu$ over $P$
\begin{equation}\label{E:ErDec}
\mu=\int_\Omega\mu_\omega\,dP(\omega),
\end{equation}
is called the \emph{ergodic decomposition} of $\mu$ under $T$ (see~\cite[Theorem~3.22]{G03}). The following properties hold:
\begin{itemize}
	\item
	$T$ acts as the identity on $\Omega$;
	\item
	the map $\omega\mapsto\mu_\omega$ is a measurable map from $\Omega$ to the set of  ergodic $T$-invariant measures on $X$;
	
	\item
	the decomposition~\eqref{E:ErDec} is  unique in the following sense: If $(Y,\CY,\nu)$ is a probability space and $ y\mapsto\mu'_y$ is a measurable map from $Y$ into the  set of ergodic measures on $X$ such that $\mu=\int_Y\mu'_y\,d\nu(y)$, then there exists a measurable map $\phi\colon Y\to \Omega$, mapping the measure $\nu$ to the measure $P$, such that
	$\mu_{\phi(y)}=\mu'_y$ for $\nu$-almost every $y\in Y$.
\end{itemize}

We call the systems  $(X,\CX,\mu_\omega,T)$, $\omega \in \Omega$, the {\em ergodic components} of
$(X,\CX,\mu,T)$.

\subsubsection*{Unique ergodicity} A topological dynamical system $(X,T)$ is   {\em uniquely ergodic} if there is a unique $T$-invariant Borel probability measure on $X$.

\subsubsection*{Bernoulli systems}
 For the purposes of this article, a {\em Bernoulli system} has the form $(X^\Z, \CB_{X^\Z}, \nu,S)$,  where $(X,\CX,\rho)$ is a  probability space,
 $S$ is the shift transformation on $X^\Z$,
$\CB_{X^\Z}$ is  the  product $\sigma$-algebra of  $X^\Z$, and  $\nu$ is the product measure $\rho^\Z$.

\subsubsection*{Nilsystems}
Let $s\in\N$, $G$ be an $s$-step nilpotent Lie group, and $\Gamma$ be a discrete cocompact subgroup of $G$. Then the quotient space $X=G/\Gamma$ is called an \emph{$s$-step nilmanifold}.
We denote the  elements of $X$ as points $x,y,\dots$, not  as cosets. The point $e_X$ is the image in $X$ of the unit element of $G$. The natural action of $G$ on $X$ is written $(g,x)\mapsto g\cdot x$ and the unique Borel measure on $X$ that is invariant under this action is called the \emph{Haar measure} of $X$ and is denoted by  $\mu_X$.
If $a\in G$, then the transformation $T\colon X\to X$ defined by $Tx=ax$, $x\in X$, is  called a {\em nilrotation of $X$}, and the system  $(X,\CX,\mu_X,T)$, where  $\CX$ is the  Borel-$\sigma$-algebra of $X$,   is called an \emph{$s$-step nilsystem}. When we do not care about the degree of nilpotency $s$ we simply call it a {\em nilsystem}. It is well known that  if $T$ is a nilrotation on $X$, then the statements  $(X,T)$ is topologically transitive, $(X,T)$ is  minimal,
$(X,\mu_X,T)$ is ergodic, and  $(X,T)$ is uniquely ergodic, are equivalent. Moreover, an ergodic nilsystem $(X,\mu_X,T)$ is
totally ergodic if and only if the nilmanifold $X$ is connected.

\subsubsection*{Joinings and disjoint systems}
Given two systems $(X,\CX, \mu,T)$ and $(Y,\CY, \nu,S)$ we call a measure $\rho$ on $(X\times Y, \CX\times \CY)$ a {\em joining} of the two systems
if it is $T\times S$ invariant and its projection onto the  $X$ and $Y$ coordinates are the measures $\mu$ and $\nu$ respectively. We say that the systems on $X$ and on $Y$ are {\em disjoint} if the only joining of the systems
is the product measure $\mu\times \nu$. If two systems are disjoint, then they have no
 non-trivial common factor,  but the converse is not true.
It is well known  that every Bernoulli system is disjoint from every zero-entropy system;  we will use  the zero entropy assumption in the proofs of our main results only via this property.

\section{Overview of the proof and  reduction to an ergodic statement}
\label{sec:overview}
In this section we give an overview of the proof of our main results and eventually reduce   to some statements of purely ergodic context which we establish in Sections~\ref{S:structure}-\ref{S:Disjoint}.   In Section~\ref{subsec:Furstenberg} we define the notion of a Furstenberg system  of an arbitrary bounded sequence.
In Section~\ref{SS:Tao} we
 reproduce  some striking identities of Tao that are implicit in \cite{Tao15} and  we use them in Section~\ref{SS:AP}  in order to  show that a Furstenberg system of the Liouville function  is a factor of a measure preserving system of purely ergodic origin;  we call it the ``system of arithmetic progressions with prime steps''. In Section~\ref{SS:StructureAP}  we state our main structural results for such systems  and we use them in Section~\ref{SS:Proof1} in order to
 get  similar structural results for   Furstenberg systems of the  M\"obius and the Liouville   function, thus proving Theorem~\ref{th:main-structure}. In Section~\ref{SS:Disj}
 we state a disjointness result which we use in Section~\ref{SS:Proof2} in order
 to prove  Theorems~\ref{th:main-new}, \ref{th:main-1}, and  \ref{th:main-2}.

\subsection{Notation regarding averages}
For  $N\in\N$  we let  $[N]=\{1,\dots,N\}$. For an arbitrary bounded sequence $a=(a(n))_{n\in\N}$ we write
$$
\E_{n\in[N]}\,a(n):=\frac 1N\sum_{n=1}^Na(n)\ \text{ and }\
\E_{n\in\N}:=\lim_{N\to\infty} \E_{n\in [N]}\, a(n)
$$
if  this limit exists.
Let $\bN= ([N_k])_{k\in\N}$ be a sequence of intervals with $N_k\to \infty$.
For an arbitrary bounded sequence $a=(a(n))_{n\in\N}$ we write
$$
\E_{n\in\bN}\, a(n):=\lim_{k\to\infty}\E_{n\in[N_k]} \, a(n)
$$
if  this limit exists and
$$
\lE_{n\in[N_k]}:=\frac 1{\log N_k}\sum_{n=1}^{N_k} \frac{ a(n)}{n}, \qquad
\lE_{n\in\bN}\, a(n):=\lim_{k\to\infty}\lE_{n\in[N_k]}\,  a(n)
$$
if this limit exists.
If $(a(p))_{p\in\P}$ is a sequence indexed by the primes, we write
$$
\E_{p\in\P}\, a(p):=\lim_{N\to\infty}\frac{1}{\pi(N)}\sum_{p\leq N}a(p),
$$
where $\pi(N)$ denotes the number of prime numbers less than $N$, if this limit exists.

Using partial summation one easily verifies  that  for a bounded sequence $(a(n))_{n\in\N}$, convergence of the Ces\`aro averages $\E_{n\in [N]}\, a(n) $ implies
convergence of the logarithmic averages $\lE_{n\in[N]}\,a(n)$ as $N\to \infty$, but the converse does not hold.
Moreover, the direct implication does not hold if we average over subsequences of intervals.

\subsection{Furstenberg systems of bounded sequences}
\label{subsec:Furstenberg}
To each bounded sequence  that is distributed ``regularly'' along a  sequence of intervals with lengths increasing to infinity,
we  associate a measure preserving system. For the purposes of this  article
all averages in the definition of Furstenberg systems of bounded sequences are taken to be logarithmic and we restrict to real valued bounded sequences.
\begin{definition}\label{D:correlations}
	Let $ \bN:=([N_k])_{k\in\N}$ be a sequence of intervals with $N_k\to \infty$.
	We say that the real valued   sequence $a\in \ell^\infty(\Z)$  {\em admits log-correlations  on $\bN$}, if the following  limits exist
	$$
	\lim_{k\to\infty} \lE_{n\in [N_k]}\,  a(n+h_1)\cdots a(n+h_\ell)
	$$
	 for every $\ell \in \N$ and  $h_1,\ldots, h_\ell\in \Z$ (not necessarily distinct).
\end{definition}
\begin{remarks}
$\bullet$	If $a\in \ell^\infty(\Z)$, then using a diagonal argument we get that every sequence of intervals $\bN=([N_k])_{k\in \N}$
	has a subsequence $\bN'=([N_k'])_{k\in\N}$, such that the sequence  $a\in \ell^\infty(\Z)$ admits log-correlations on $\bN'$.

$\bullet$ If $a(n)$ is only defined for $n\in\N$ we extend it in an arbitrary way to $\Z$ and define the analogous notion. Then all the limits above do not depend on the choice of the extension.
\end{remarks}

 The correspondence principle of Furstenberg was originally used in \cite{Fu77} in order to restate  Szemer\'edi's theorem on arithmetic progressions in ergodic terms.
We will use the following  variant of this principle
which applies to general real valued bounded sequences:
\begin{proposition}\label{P:correspondence}
	Let $a\in\ell^\infty(\Z)$ be a real valued sequence that  admits
	log-correlations  on 
	$\bN:=([N_k])_{k\in\N}$.
	Then there exist a  topological  system $(X,T)$,  a    $T$-invariant Borel probability measure $\mu$, and a real valued $T$-generating function $F_0\in C(X)$,\footnote{A real valued function $F_0\in C(X)$ is $T$-generating if the functions $T^nF_0$, $n\in\Z$, separate points  of $X$.
		By  the Stone-Weierstrass theorem, this holds if and only if the $T$-invariant subalgebra  generated by $F_0$ 
		is dense in $C(X)$ (we restrict to real valued functions) with the uniform topology. }
	 such that
\begin{equation}\label{E:correspondence}
	\lE_{n\in {\bN} }\,  \prod_{j=1}^\ell a(n+h_j) =\int \prod_{j=1}^\ell
	T^{h_j}F_0 \, d\mu
	\end{equation}
	for every $\ell\in \N$ and  $h_1, \ldots, h_\ell\in \Z$. 
\end{proposition}

\begin{definition}
	Let $a\in\ell^\infty(\N)$ be a real valued sequence that  admits
	log-correlations on
	$\bN:=(N_k)_{k\in\N}$.
	We call the system (or the measure $\mu$) defined in Proposition~\ref{P:correspondence}
	the {\em Furstenberg system  (or measure) associated with $a$ and $\bN$}.
\end{definition}
\begin{remarks}
	$\bullet$  Given $a\in \ell^\infty(\Z)$ and $\bN$,  the measure  $\mu$ is uniquely determined by \eqref{E:correspondence} since this identity determines the values of $\int f\, d\mu$ for all real valued  $f\in C(X)$.
	
$\bullet$ 	A priori a sequence $a\in \ell^\infty(\Z)$ may have several, perhaps   uncountably many, non-isomorphic Furstenberg systems
	depending on which sequence of intervals $\bN$ we use in the evaluation of the log-correlations of the sequence $a\in \ell^\infty(\Z)$.
	When we write that  a Furstenberg measure or system of  a sequence has a certain property we mean that any of these measures or systems has the asserted property.
\end{remarks}

In the construction of the Furstenberg system $(X,\CX,\mu,T)$  we can take $X$ to be the compact metric space $I^\Z$ (with the product topology) where $I$ is any closed and  bounded  interval  containing the range of $(a(n))_{n\in\Z}$,
$\CX$ to be  the   Borel-$\sigma$-algebra of $I^\Z$, and  $T$ to be  the shift transformation on $I^\Z$. Points of $X$ are written as $x=(x(n) )_{n\in\Z}$  and we let $F_0(x):=x(0)$, $x\in X$. Then  $F_0\in C(X)$ and $F_0$  is $T$-generating. We consider the sequence $a=(a(n))_{n\in\Z}$ as a point of $X$. Our hypothesis implies that the measures
\begin{equation}
\label{E:nu}
\lE_{n\in[N_k]}\delta_{T^na},
\quad  k\in \N,
\end{equation}
converge weak-star as $k\to\infty$ to a measure $\mu$ on $X$, and this measure  is  clearly $T$-invariant  and satisfies \eqref{E:correspondence}. Indeed, if   $F=\prod_{j=1}^\ell T^{h_j}F_0$, then  $F\in C(X)$ and
$F(T^na)=\prod_{j=1}^\ell a(n+h_j)$, $n\in \N$, and the weak-star convergence of the measures in \eqref{E:nu} to $\mu$ gives identity \eqref{E:correspondence}.

 In this article  we are mostly interested in applying the previous result when $a=\bmu$  in which  case we take $X:=\{-1,0, 1\}^\Z$. For every $h\in\Z$  we write $F_h\colon X\to\{-1,0, 1\}$ for the function given by
$$
F_h(x):=x(h),   \quad x\in X.
$$
Then  for every $h\in\Z$ we have  $F_h=T^hF_0$.   If  $(X,\CX,\mu,T)$  is the Furstenberg system associated with the M\"obius function and  the sequence  $\bN$,  by Proposition~\ref{P:correspondence} we have
$$
\int\prod_{j=1}^\ell F_{h_j}(x)\,d\mu(x)=
\int\prod_{j=1}^\ell  T^{h_j}F_0\,d\mu=
\lE_{n\in\bN}\prod_{j=1}^\ell\bmu(n+h_j).
$$
for every $\ell \in\N$ and  $h_1,\dots,h_\ell\in\Z$.

\subsection{A convergence result for  multiple correlation sequences} We will make use of  the following consequence of Theorem~\ref{T:FHK} below:
\begin{proposition}	\label{P:conv}
	Suppose that the sequence $a\in \ell^\infty(\Z)$
	admits log-correlations on the sequence of intervals $\bN$. Then  the limit
$$
\E_{p\in\P}\Bigl(\lE_{n\in\bN}\prod_{j=1}^\ell a(n+ph_j)\Bigr)
$$
exists  for all  $\ell\in\N$ and $h_1,\ldots, h_\ell\in \Z$.
\end{proposition}
\begin{proof}
	Let $(X,\CX, \mu,T)$ be the Furstenberg system associated with  $a\in \ell^\infty(\Z)$ and  $\bN$, and let  also $F_0\in L^\infty(\mu)$ be as in Proposition~\ref{P:correspondence}.
Using Theorem~\ref{T:FHK} \ in Section~\ref{subsec:prime-steps} we get  that  for every  $\ell\in\N$ and $h_1,\ldots, h_\ell\in \Z$  the limit
$$
 \E_{p\in\P}\int\prod_{j=1}^\ell T^{ph_j}F_0\,d\mu
$$
exists. By \eqref{E:correspondence} we can replace
  $\int \prod_{j=1}^\ell T^{ph_j}F_0\,d\mu$ by  $\lE_{n\in\bN}\prod_{j=1}^\ell a(n+ph_j)$ and we arrive to the asserted
conclusion.
\end{proof}

\subsection{Tao's identities}\label{SS:Tao}
A key tool in our argument is the following rather amazing  identity which is implicit in  \cite{Tao15}:
\begin{theorem}[Tao's identity for general sequences]\label{T:Tao}
	Let $\bN=([N_k])_{k\in\N}$ be a sequence of intervals with $N_k\to \infty$,   $a\in \ell^\infty(\Z)$ be a sequence (perhaps complex valued),  and $\ell\in\N$,  $h_1,\ldots, h_\ell\in \Z$.
	If we assume
	 that on the left and  right hand side below the   limits $\lE_{n\in \bN}$ exist for every $p\in\P$ and the limit  $\E_{p\in\P}$  exists, then we have the identity
	$$
	\E_{p\in\P}\, 
		\Big(\lE_{n\in \bN}\,\prod_{j=1}^\ell a(pn+ph_j)\Big)=
		\E_{p\in \P}\,
\Big(\lE_{n\in \bN}\, \prod_{j=1}^\ell a(n+ph_j)\Big).
$$
\end{theorem}
We give a sketch of the proof of a more general identity  in Appendix~\ref{A:Tao}; the argument  is almost entirely based on
the argument given by  Tao in \cite{Tao15}.

Using the previous result we verify the following identities for the M\"obius and the  Liouville  function:
\begin{theorem}[Tao's identity for $\bmu$ and $\blambda$]\label{T:Tao1}
	Suppose that the M\"obius function $\bmu$  admits log-correlations on the sequence of intervals $\bN$. Then  we have
	$$
	\lE_{n\in \bN}\,  \prod_{j=1}^\ell \bmu(n+h_j)=
	(-1)^\ell \, \E_{p\in \P}\,  \,\Big( 	\lE_{n\in \bN}\,  \prod_{j=1}^\ell \bmu(n+ph_j)\Big)
	$$
	for all $\ell \in \N$ and $h_1,\ldots, h_\ell\in \Z$, in particular the limit $\E_{p\in\P}$ on the right hand side exists. A similar statement holds for the Liouville function $\blambda$.
\end{theorem}
\begin{proof}
 We first check the identity for the Liouville function. We verify that the hypothesis of   Theorem~\ref{T:Tao}  apply for $a:=\blambda$. The  limit
$\lE_{n\in \bN}$ on the left and right hand side   exists for every $p\in \P$
since $\blambda$  admits log-correlations on $\bN$ and
it is completely  multiplicative. Moreover,
using complete multiplicativity, the left hand side becomes
$(-1)^\ell\,\lE_{n\in \bN}\, \prod_{j=1}^\ell \blambda(n+h_j)$. The right hand side is $\E_{p\in \P}\,  \,\Big( 	\lE_{n\in \bN}\,  \prod_{j=1}^\ell \blambda(n+ph_j)\Big)$; note that   the existence of the  limit   $\E_{p\in\P}$  follows from Proposition~\ref{P:conv}.	So Theorem~\ref{T:Tao} applies for  $a:=\blambda$ and gives the asserted identity.

The argument is slightly more complicated for the M\"obius function because in this case we lose complete multiplicativity. Arguing by contradiction, suppose that  the asserted estimate fails. Then there exist a subsequence $\bN':=([N_k'])_{k\in\N}$ of $\bN:=([N_k])_{k\in\N}$ and  $\ell\in\N$, $h_1,\ldots, h_\ell\in\Z$,
such that
the limit $\lE_{n\in \bN'}\,\prod_{j=1}^\ell \bmu(pn+ph_j)$ exists for every $p\in\P$,  and we have
\begin{equation}\label{E:noteq}
\lE_{n\in \bN'}\,  \prod_{j=1}^\ell \bmu(n+h_j)\neq
(-1)^\ell \, \E_{p\in \P}\,  \,\Big( 	\lE_{n\in \bN'}\,  \prod_{j=1}^\ell \bmu(n+ph_j)\Big).
\end{equation}
Note that the existence of the limit $\E_{p\in\P}$  on the right hand side follows again  from Proposition~\ref{P:conv}.
For $j=1,\ldots, \ell$  and $p\in\P$ we have
$\bmu(pn+ph_j)=-\bmu(n+h_j)$   unless $n+h_j\equiv 0\pmod{p}$. For $p\in\P$ this leads to the identity
$$
\lE_{n\in \bN'}\,\prod_{j=1}^\ell \bmu(pn+ph_j)=(-1)^\ell\, \lE_{n\in \bN'}\,\prod_{j=1}^\ell \bmu(n+h_j)+O(1/p)
$$
where the implicit constant depends only on $\ell$. Averaging over $p\in \P$ we get
\begin{equation}\label{E:eq}
\E_{p\in\P}\Big(\lE_{n\in \bN'}\,\prod_{j=1}^\ell \bmu(pn+ph_j)\Big)=(-1)^\ell\, \lE_{n\in \bN'}\,\prod_{j=1}^\ell \bmu(n+h_j),
\end{equation}
in particular, the limit $\E_{p\in\P}$ on the left hand side exists. So Theorem~\ref{T:Tao} applies for  $a:=\bmu$ and the sequence of intervals $\bI'$, and gives that
$$
\E_{p\in\P}\Big(\lE_{n\in \bN'}\,\prod_{j=1}^\ell \bmu(pn+ph_j)\Big)=
 \E_{p\in \P}\,  \,\Big( 	\lE_{n\in \bN'}\,  \prod_{j=1}^\ell \bmu(n+ph_j)\Big).
$$
Combining this identity with \eqref{E:eq} we get an identity which contradicts \eqref{E:noteq}. This completes the proof.
	\end{proof}

Using Theorem~\ref{T:Tao1} we immediately deduce the following identities
for Furstenberg systems of the M\"obius and the Liouville function:
\begin{theorem}[Ergodic form of Tao's identities for $\bmu$ and  $\blambda$] \label{T:Tao2}
	Let $(X,\CX,\mu,T)$ be a Furstenberg system of the M\"obius or the Liouville function and let $F_0$ be as in Proposition~\ref{P:correspondence}. Then the limit in the right hand side below exists and  we have
	\begin{equation}
	\label{eq:Furstenberg-Tao}
	\int \prod_{j=1}^\ell T^{h_j} F_0 \, d\mu=
	(-1)^\ell\,  \E_{p\in \P} \,   \int \prod_{j=1}^\ell T^{ph_j}F_0\, d\mu
	\end{equation}
	for all $\ell\in \N$ and  $h_1,\ldots, h_\ell\in \Z$.
\end{theorem}
Henceforth, our goal is to describe the structure of measure preserving systems that
satisfy the identities in \eqref{eq:Furstenberg-Tao} for some $T$-generating function $F_0\in C(X)$.  For technical reasons it is essential for us to work with
suitable extensions of such systems which we describe in the next subsection.  Our main task will then be  to get structural results for these extended systems.

\subsection{The system of arithmetic progressions with prime steps}\label{SS:AP}
Motivated by Theorem~\ref{T:Tao2},
given a system  $(X, \mu,T)$, we are going to construct a new system on the space $X^\Z$
by averaging  the prime dilates of  correlations of the system on the space $X$.
Since in some cases $X$ is itself  a sequence space  with elements  denoted by $x=(x(n))_{n\in\Z}$,  we denote elements of  $X^\Z$ by   $\ux=(x_n)_{n\in\Z}$.
\begin{definition}\label{D:tilde}
	Let $(X,\CX, \mu,T)$ be a system  and let $X^\Z$ be endowed with the product $\sigma$-algebra.	We write $\wt\mu$ for the   measure on $X^\Z$ characterized as follows:
	For every $m\in\N$ and all $f_{-m},\ldots,f_m\in L^\infty(\mu)$, we define
	\begin{equation}
		\label{eq:def-mutilde}
\int_{X^\Z}\prod_{j=-m}^m f_j(x_j)\,d\wt\mu(\ux):=		
\E_{p\in\P}\int_X\prod_{j=-m}^m T^{pj}f_j\,d\mu.
	\end{equation}
Note that the limit above exists by Theorem~\ref{T:FHK} in Section~\ref{subsec:prime-steps}.
Using the identity  $\int_X\prod_{j=-m}^m T^{p(j+1)}f_j\,d\mu=\int_X\prod_{j=-m}^m T^{pj}f_j\,d\mu$, we get that the measure $\wt\mu$ is invariant under 
 the shift transformation $S$ on $X^\Z$.  We say that $(X^\Z, \wt\mu,S)$ is the \emph{system of arithmetic progressions with prime steps} associated with the system  $(X,\mu,T)$.
\end{definition}

We return now to the case where  $(X,\mu,T)$ is a Furstenberg system of the Liouville function and make the  following  key observation:

\begin{proposition}
	\label{prop:mu-fator-tilde-mu}
	A Furstenberg system $(X,\mu,T)$ of the M\"obius or the Liouville function is a  factor of the associated system $(X^\Z,\wt\mu,S)$ of arithmetic progressions with prime steps.
\end{proposition}
\begin{remark}
	The fact that the M\"obius and the Liouville function  is $-1$ on primes is crucial for the proof of this result and  is used via the identity~\eqref{eq:Furstenberg-Tao}. In fact, our argument also works for all bounded multiplicative functions which take the value  $-1$ on a subset of the primes with relative density $1$.
	\end{remark}
\begin{proof}
	We  can take $X=\{-1,0,1\}^\Z$. We define the map $\pi\colon X^\Z\to X$  as follows:
	For  $\ux=(x_n)_{n\in\Z}\in X^\Z$ let
	$$
	(\pi(\ux))(n):=-x_n(0)=-F_0(x_n), \quad n\in\Z,
	$$	
	where, as usual, $F_h(x)=x(h)$, $x\in X$, $h\in \Z$.
	For $n\in\Z$ we then have
	$$
	(\pi(S\ux))(n)=-F_0((S\ux)_n)=-F_0(x_{n+1})=(\pi(\ux))(n+1)=(T\pi(\ux))(n).
	$$
	Thus
	$$
	\pi\circ S=T\circ\pi.
	$$

	Next, we claim that  $\wt\mu\circ \pi^{-1}=\mu$.
	Indeed, for every $\ell \in\N$ and  $h_1,\dots,h_\ell\in\Z$, by  identity~\eqref{eq:Furstenberg-Tao} in Theorem~\ref{T:Tao} and the definition~\eqref{eq:def-mutilde} of $\wt\mu$, we have
	\begin{multline*}
		\int_X\prod_{j=1}^\ell F_{h_j}(x)\, d\mu(x)=
		 \int_X\prod_{j=1}^\ell F_{0}(T^{h_j}x)\,d\mu(x)\\
		=
		(-1)^\ell \, \E_{p\in\P} \int_X\prod_{j=1}^\ell F_0(T^{ph_j}x)\,d\mu(x)
		=
		(-1)^\ell\, \int_{X^\Z}\prod_{j=1}^\ell F_0(x_{h_j})\,d\wt\mu(\ux)\\
		=
		\int_{X^\Z}\prod_{j=1}^\ell \bigl(-F_0(x_{h_j})\bigr)\,d\wt\mu(\ux)
		=\int_{X^\Z}\prod_{j=1}^\ell (F_{h_j}\circ\pi)(\ux)\,d\wt\mu(\ux).
	\end{multline*}
Since the algebra generated by  the functions $F_h$, $h\in \Z$,  is dense in $C(X)$ with the uniform topology, the claim follows.
	
	 Therefore, $\pi\colon (X^\Z,\wt\mu,S)\to (X,\mu,T)$ is a factor map and the proof is complete.
	\end{proof}

From this point on we   work with abstract systems of arithmetic progressions with prime steps
and use Proposition~\ref{prop:mu-fator-tilde-mu} in order to transfer any structural result  we get to
 a structural result for  Furstenberg systems of the M\"obius and the Liouville function.

\subsection{Structure of systems of arithmetic progressions with prime steps}\label{SS:StructureAP}
We state
our main structural results for abstract systems of arithmetic progressions with prime steps.
In Section~\ref{S:structure} we  show:
\begin{theorem}
	\label{th:main-mutilde}
	Let $(X,\mu,T)$ be a system. Then  almost every ergodic component
	of the system $(X^\Z,\wt\mu,S)$, of arithmetic progressions with prime steps,   is isomorphic  to
a direct product of an infinite-step nilsystem and a Bernoulli system.
\end{theorem}
In Section~\ref{S:sst} we  show:
\begin{theorem}
	\label{th:no-irr-eigen}
	Let $(X,\mu,T)$ be a system. Then the system $(X^\Z,\wt\mu,S)$, of arithmetic progressions with prime steps, has no   irrational spectrum.
\end{theorem}
We also establish similar results for systems of arithmetic progressions with integer steps (see Definition~\ref{D:under}).

\subsection{Proof of   Theorem~\ref{th:main-structure}  assuming the preceding material}\label{SS:Proof1}
Combining Proposition~\ref{prop:mu-fator-tilde-mu}   and Theorem~\ref{th:no-irr-eigen},
we  get that any Furstenberg system of the M\"obius or the  Liouville function is a factor of a system with no irrational spectrum (and hence has no irrational spectrum)  thus establishing Property~$\text{(i)}$ of Theorem~\ref{th:main-structure}.
Combining Proposition~\ref{prop:mu-fator-tilde-mu} and Theorem~\ref{th:main-mutilde}, we get Property~$\text{(ii)}$ of Theorem~\ref{th:main-structure}.\qed

\subsection{Disjointness}\label{SS:Disj}
As we previously remarked, our proof strategy for Theorems~\ref{th:main-new}, \ref{th:main-1}, and  \ref{th:main-2}
is to study the structure of Furstenberg systems of the M\"obius and the Liouville  function in enough detail
that enables us to prove a useful  disjointness result. The relevant disjointness result
 is the following one and is proved  in Section~\ref{S:Disjoint}:
\begin{proposition}
	\label{P:disjoint}
	Let $(X,\mu,T)$ be  a system with ergodic components   isomorphic to
		 direct products of    infinite-step nilsystems and  Bernoulli systems. Let  $(Y,\nu,R)$ be an ergodic system of zero entropy.
\begin{enumerate}
\item
\label{it:disjoint-1}
If the two systems   have disjoint irrational spectrum, then for every  joining  $\sigma$  of the two systems  and   function $f\in L^\infty(\mu)$  orthogonal to $\CK(T)$, we have
$$
	\int f(x)\, g(y)\, d\sigma(x,y)=0
$$
for every $g\in L^\infty(\nu)$.

\item
\label{it:disjoint-2}
If the two systems  have no common eigenvalue except $1$,
 then they are disjoint.
\end{enumerate}
\end{proposition}
We will use  the following direct consequence:
\begin{corollary}
\label{C:disjoint}
	Part~\eqref{it:disjoint-1} of Proposition~\ref{P:disjoint} holds under the weaker assumption that $(Y,\nu,R)$ is a zero entropy system with   at most countably many  ergodic components.
Furthermore, if  the two systems  have no common eigenvalue except $1$,	then  for every joining $\sigma$ of these systems we have
$$\int f(x)\, g(y)\, d\sigma(x,y)=0
$$ for every $f\in L^\infty(\mu)$ and  every  $g\in L^\infty(\nu)$ that is orthogonal in $L^2(\nu)$ to all $R$-invariant functions.
\end{corollary}
\begin{proof}
Let $\nu=\sum_{j\in J}c_j\nu_j$ be the ergodic decomposition of $\nu$ under $R$, where $J$ is a finite or an infinite countable set, $c_j>0$, $\sum_{j\in J}c_j=1$, and $\nu_j$, $j\in J$,  are ergodic $R$-invariant measures.  Let $Y=\cup_{j\in J}Y_j$ be a partition of $Y$ into $R$-invariant subsets such that for every $j\in J$  we have  $\nu_j(Y_j)=1$.

Let  $\sigma$ be a joining of
the systems $(X,\mu,T)$ and $(Y,\nu,R)$.
For $j\in J$ we let $\sigma_j:=\frac 1{c_j} \one_{X\times Y_j}\cdot\sigma$ and  $\mu_j$ be the image of $\sigma_j$ under the projection of $X\times Y$ on $X$.
Then for $j\in J$ we have that  $\mu_j$ is a $T$-invariant probability measure on $X$, the image of $\sigma_j$ under the projection of $X\times Y$ onto $Y$ is $\nu_j$, and $\sigma_j$ is a joining of the systems $(X,\mu_j,T)$ and $(Y,\nu_j,R)$.

For $j\in J$ the measure  $\nu_j$ is absolutely continuous with respect to $\nu$ and thus the spectrum of $(Y,\nu_j,R)$ is contained in the spectrum of
$(Y,\nu,R)$. Similarly,  for $j\in J$ the measure   $\mu_j$ is absolutely continuous with respect to $\mu$ and thus the spectrum of $(X,\mu_j,T)$ is contained in the spectrum of
$(X,\mu,T)$. Moreover, every ergodic component of $\mu_j$ is an ergodic component of $\mu$ and thus is isomorphic to the   direct product of   an infinite-step nilsystem and a   Bernoulli system.

In case~\eqref{it:disjoint-1}, suppose that $f\in L^\infty(\mu)$ is orthogonal to $\CK(X,\mu,T)$. This means that $f$ is orthogonal in $L^2(\mu)$ to every eigenfunction of $(X,\mu,T)$ corresponding to a rational eigenvalue. It follows that  for every $j\in J$ the function $f$ is orthogonal in $L^2(\mu_j)$ to every eigenfunction of $(X,\mu_j,T)$ corresponding to a rational eigenvalue, and by Part~\eqref{it:disjoint-1} of Proposition~\ref{P:disjoint} we have
$\int f(x)\, g(y)\, d\sigma_j(x,y)=0$
for every $g\in L^\infty(\nu_j)$. Summing up, we obtain
$\int f(x)\, g(y)\, d\sigma(x,y)=0$
for every $g\in L^\infty(\nu)$.

 Furthermore, for every $j\in J$ the systems
$(X,\mu_j,T)$ and $(Y,\nu_j,R)$ have no common eigenvalue except $1$, and thus are disjoint by Part~\eqref{it:disjoint-2} of Proposition~\ref{P:disjoint}.
Therefore,  for every $j\in J$ the measure $\sigma_j$ defined above is equal to $\mu_j\times\nu_j$.  Summing up, and since by assumption $\int g\, d\nu_j=0$ for every $j\in J$, we get
that $\int f(x)\, g(y)\, d\sigma(x,y)=0$. This completes the proof.
\end{proof}

\subsection{Proof of   Theorem~\ref{th:main-1} assuming the preceding material}
\label{SS:Proof2}
 We consider only the case of the  M\"obius  function, the proof for the Liouville function is identical.

 Arguing  by contradiction, suppose that the conclusion of Theorem~\ref{th:main-1} fails. Then   there exist a  topological  dynamical system $(Y,R)$,   a point $y_0\in Y$ generic for a measure $\nu$ such that the system $(Y,\nu,R)$ has  zero entropy and  at most countably many  ergodic components, and a function $g_0\in C(Y)$ such that the averages
\begin{equation}
\label{eq:main-1-proof}
\lE_{n\in [N]}g_0(R^ny_0)\, \bmu(n)
\end{equation}
do not converge to $0$ as $N\to\infty$. Hence, there exists a sequence $\bN=(N_k)_{k\in\N}$ of intervals with $N_k\to\infty$  such that the  limit
\begin{equation}
	\label{E:000}
\lE_{n\in\bN}\,g_0(R^ny_0)\, \bmu(n)=
\lim_{k\to\infty} \lE_{n\in [N_k]}\,g_0(R^ny_0)\, \bmu(n)
\end{equation}
exists and is non-zero. After passing to a subsequence, which we also denote by  $\bN$, we can further assume that the limit
\begin{equation}
\label{eq:main-1-proof2}
\lE_{n\in\bN}\, g(R^ny_0)\prod_{j=1}^{\ell}\, \bmu(n+h_j)
\end{equation}
exists for every $\ell \in \N$, $h_1,\dots,h_\ell\in\Z$, and $g\in C(Y)$.

Let $X:=\{-1,0,1\}^\Z$, $T\colon X\to X$ be the shift transformation,  and $x_0\in X$  be defined by $x_0(n)=\bmu(n)$, $n\in\Z$. Then the convergence~\eqref{eq:main-1-proof2} implies  that for every  $\ell \in\N$,  $h_1,\ldots,h_\ell\in\Z$, and every $g\in C(Y)$ the limit
$$
\lE_{n\in\bN}\,g(R^ny_0)\, \bigl(\prod_{j=1}^\ell F_{h_j}\bigr)(T^nx_0)
$$
exists (recall that $F_h(x)=x(h)$, $x\in X$,  $h\in\Z$). Since the algebra generated by  the functions $F_h$, $h\in \Z$,  is dense in $C(X)$  with the uniform topology, we deduce  that the sequence of measures
$$
\lE_{n\in [N_k]}\delta_{(T^nx_0,R^ny_0)}, \quad k\in\N,
$$
converges weak-star to some probability measure $\sigma$ on $X\times Y$ that satisfies
\begin{equation}
	\label{eq:main-1-proof3}
	\lE_{n\in\bN}\,g(R^ny_0)\, \prod_{j=1}^\ell\bmu(n+h_j)=\int \prod_{j=1}^\ell F_{h_j}(x)\,g(y)\,d\sigma(x,y)
\end{equation}
 for every $\ell \in \N$, $h_1,\dots,h_\ell\in\Z$, and $g\in C(Y)$.
 By construction, $\sigma$ is invariant under $T\times R$.

  The projection of $\sigma$ on $Y$ is the weak-star limit of the sequence of measures $\lE_{n\in [N_k]}\delta_{R^ny_0}$, $k\in\N$, and since the point $y_0$ is generic for $\nu$, this measure is equal to $\nu$ and thus the corresponding  measure preserving system  has zero entropy and  at most countably many  ergodic components.

 The projection of $\sigma$ on $X$ is the weak-star limit of the sequence of measures $\lE_{n\in [N_k]}\delta_{T^nx_0}$, $k\in\N$.
 It is thus a $T$-invariant measure  $\mu$ which is the Furstenberg measure associated with $\bmu$ and $\bN$ by Proposition~\ref{P:correspondence} and  $\sigma$ is a joining of the systems  $(X,\mu,T)$ and $(Y,\nu,R)$.

 By Proposition~\ref{prop:mu-fator-tilde-mu} and its proof, $(X,\mu,T)$
 is a factor of the system $(X^\Z,\wt \mu, S)$,
 with factor map  $\pi \colon X^\Z\to X$ given by
 $$
 (\pi(\ux))(n)=-x_n(0), \quad \ux\in X^\Z, \ n\in \Z.
 $$
 We define the joining $\wt\sigma$  of the systems  $(X^\Z, \wt \mu, S)$ and $(Y,\nu,R)$ by
\begin{equation}\label{E:f|X}
\int_{X^\Z\times Y} f(\ux)\cdot g(y) \, d\wt\sigma(\ux,y)= \int_{X\times Y} \E_{\wt\mu}(f\mid X)(x) \cdot g(y) \, d\sigma(x,y)
\end{equation}
for every $f\in L^\infty(\wt \mu)$ and $g\in L^\infty(\nu)$.

  By Theorems~\ref{th:main-mutilde} and \ref{th:no-irr-eigen},  the system $(X^\Z,\wt\mu,S)$  has no irrational spectrum and its ergodic components are isomorphic to
direct products of  infinite-step nilsystems and  Bernoulli systems.

 We  verify now that the function $\wt F_0:=F_0\circ\pi$ is
 orthogonal to the rational Kronecker factor of the system $(X^\Z,\wt \mu, S)$.
In fact we will show that $\wt F_0$ is orthogonal to the Kronecker factor of this system. By a well known consequence of the spectral theorem for unitary operators,
this property is equivalent to establishing that
\begin{equation}\label{E:Kronecker}
	\E_{n\in\N} \Big|\int  \wt F_0\cdot S^n  \wt F_0\, d\wt \mu\Big|=0.
\end{equation}
By the definition of the measure $\wt\mu$ (see \eqref{eq:def-mutilde})  and since for $h\in \N$ we have $\wt F_0(\ux) \,\wt F_0(S^h\ux)=(-F_0(x_0)) \, (-F_0(x_h))$,  we get for every $n\in \N$ that
$$
\int  \wt F_0\cdot S^n  \wt F_0\, d\wt \mu=\E_{p\in\P}\int  F_0\cdot T^{pn}  F_0\, d \mu.
$$
By \eqref{E:correspondence}, for  every $h\in \N$ we have
$$
\int F_0\cdot T^hF_0\, d \mu=	\lE_{n\in {\bN} }\,  \bmu(n)\, \bmu(n+h)=0
$$
where  the vanishing of the average follows from the main result of Tao in \cite{Tao15}. Combining the above identities we get \eqref{E:Kronecker}.

By Corollary~\ref{C:disjoint}, we have
$$
0=\int \wt F_0(\ux) \cdot g_0(y)\,d\wt\sigma(\ux,y)= \int F_0(x)\cdot  g_0(y)\,d\sigma(x,y)=\lE_{n\in\bN}\,g_0(R^ny_0)\, \bmu(n)
$$
by~\eqref{eq:main-1-proof3}, contradicting our assumption that the limit in \eqref{E:000} is non-zero. This completes the proof.
\qed

\subsection{Proof of Theorem~\ref{th:main-new} assuming the preceding material}
We proceed exactly as in the proof of Theorem~\ref{th:main-1}  in Section~\ref{SS:Proof2}. Arguing by contradiction, we assume that there exist a  topological dynamical system $(Y,R)$, a point $y_0\in Y$, and a continuous function $g_0$ on $Y$ such that the logarithmic averages~\eqref{eq:main-1-proof} do not converge to $0$.
We construct a sequence of intervals $\bN=(N_k)_{k\in\N}$, a system $(X,T)$,  and a measure $\sigma$ on $X\times Y$, as in the proof of Theorem~\ref{th:main-1} in Section~\ref{SS:Proof2}. The projection $\nu$ of $\sigma$ on $Y$ is an $R$-invariant measure, and since $(Y, R)$ has at most  countably many    ergodic invariant measures, $\nu$ has at most  countably many ergodic components. Since the system $(Y,R)$ has zero topological entropy, all these components have zero entropy and the system $(Y,\nu,R)$ has zero entropy. We conclude as in the proof of Theorem~\ref{th:main-1}  in Section~\ref{SS:Proof2}.\qed

\subsection{Proof of Theorem~\ref{th:main-2} assuming the preceding material}
We consider only the case of the  M\"obius  function, the proof for the Liouville function is identical.

Arguing  by contradiction, suppose that the conclusion of Theorem~\ref{th:main-2} fails.  Then   there exist a  topological dynamical system $(Y,R)$, a point $y_0\in Y$ that is generic for a measure $\nu$ such that the system $(Y,\nu,R)$ has  zero entropy and at most  countably many ergodic components all of which are totally ergodic, and   a function  $g_0\in C(Y)$  such that  for some $\ell_0 \in \N$ and some $h_{0,1},\dots,h_{0,\ell_0}\in\Z$  the identity \eqref{eq:main-2} fails, namely, the averages
$$
 \lE_{n\in [N]} \,g_0(R^ny_0)\, \prod_{j=1}^{\ell_0}\bmu(n+h_{0,j})
$$
do not converge to $0$ as $N\to\infty$.

 As in the proof of Theorem~\ref{th:main-1}  in Section~\ref{SS:Proof2},
we define a sequence of intervals $\bN=(N_k)_{k\in\N}$ such that the above averages converge to some non-zero number, a system $(X,T)$, and  a measure $\sigma$ on $X\times Y$ such that \eqref{eq:main-1-proof3} holds.
By construction, $\sigma$ is invariant under $T\times R$.
 By  assumption and the definition of genericity, the projection of $\sigma$ on $Y$ is the measure $\nu$, and thus   the system  $(Y,\nu,R)$ has zero entropy,   at most countably many  ergodic components, and no rational eigenvalue except  $1$.

The projection of $\sigma$ on $X$ is a $T$-invariant measure  $\mu$ which by \eqref{eq:main-1-proof3} is the Furstenberg measure associated with $\bmu$ and $\bN$ by Proposition~\ref{P:correspondence}. Hence, by Proposition~\ref{prop:mu-fator-tilde-mu}, the system $(X,\mu,T)$
 is a factor of the system $(X^\Z,\wt \mu, S)$. By Theorems~\ref{th:main-mutilde} and \ref{th:no-irr-eigen},  the system $(X^\Z,\wt\mu,S)$  has no irrational spectrum and its ergodic components are isomorphic to
direct products of   infinite-step nilsystems and  Bernoulli systems.

From the previous discussion it follows that    the function $g_0$ and the systems $(X^\Z,\wt \mu, S)$ and  $(Y,\nu,R)$
 satisfy the hypothesis of the second part   of  Corollary~\ref{C:disjoint}. Hence, for every joining $\wt{\sigma}$ of these
 systems and $\tilde{f}\in  L^\infty(\wt \mu)$, we have $\int \tilde{f}(\ux)\, g_0(y)\, d\wt{\sigma}(\ux,y)=0$.
  Since $\sigma$ is a joining of the systems $(X,\mu,T)$ and $(Y,\nu,R)$, and
 the system $(X,\mu,T)$ is a factor of  $(X^\Z,\wt \mu, S)$,the measure $\sigma$ can be lifted to  a joining $\wt\sigma$ of $(X^\Z,\wt \mu,
 S)$ and  $(Y,\nu,R)$. It follows that for every $f\in  L^\infty(\mu)$ we
 have $\int f(x)\, g_0(y)\, d\sigma(x,y)=0$.
We deduce that
$$
 \lE_{n\in \bN} \,g_0(R^ny_0)\, \prod_{j=1}^{\ell_0}\bmu(n+h_{0,j})
 =\int_{X\times Y} \prod_{j=1}^{\ell_0}F_{h_{0,j}}(x)\cdot g_0(y)\,d\sigma(x,y)=0.
$$
This contradicts our assumption that $\lE_{n\in \bN} \,g_0(R^ny_0)\, \prod_{j=1}^{\ell_0}\bmu(n+h_{0,j})\neq 0$ and completes the proof of  Theorem~\ref{th:main-2}. \qed

 \section{The structure of systems of arithmetic progressions}\label{S:structure}

The goal of this section is to prove Theorem~\ref{th:main-mutilde}  which gives information about   the structure of systems of arithmetic progressions
 with prime steps associated with a system $(X,\mu,T)$.
We will   work   progressively with systems of increasing complexity  starting from the  case where  $(X,\mu,T)$ is a nilsystem. This important  case will be dealt using the theory of arithmetic progressions on nilmanifolds which is summarized  in Appendix~\ref{A:AP}.

\subsection{Systems of arithmetic progressions}
We start with the definition of systems of arithmetic progressions with integer steps which are
a stepping stone towards understanding the structure of the systems of arithmetic progressions with prime steps.
\subsubsection{The system of arithmetic progressions with integer steps}
We  will use the following  result from \cite{HK} (convergence was also established in \cite{Zi07}):
\begin{theorem}
	\label{T:HK}
	Let $(X,\mu,T)$ be a system. Then
	for every $\ell\in\N$ and  $f_{1},\ldots ,f_\ell\in L^\infty(\mu)$ the following limit exists in $L^2(\mu)$
	\begin{equation}
		\label{E:HK}
		\E_{n\in\N}\prod_{j=1}^\ell T^{nj}f_j.
	\end{equation}
	Furthermore, if the system is ergodic, $Z_\infty$ is the infinite-step nilfactor of the system (see Appendix~\ref{SS:infnilfac}),  and if
	$\E_\mu(f_j\mid Z_\infty)=0$ for some  $j\in\{1,\ldots, \ell\}$, then the limit  \eqref{E:HK} is $0$.
\end{theorem}
In accordance to the system of arithmetic progressions with prime steps (see Definition~\ref{D:tilde}) we define  systems of arithmetic progressions with integer steps as follows:
\begin{definition}
	\label{D:under}
	Let $(X,\mu,T)$ be a system. We write $\umu$ for the measure on $X^\Z$ characterized as follows:
	For every $m\in\N$ and all $f_{-m},\ldots,f_m\in L^\infty(\mu)$, we define
	\begin{equation}
		\label{eq:def-mubar}
\int_{X^\Z}\prod_{j=-m}^m f_j(x_j)\,d\umu(\ux):=		\E_{n\in\N}\int_X\prod_{j=-m}^m T^{nj}f_j\,d\mu.
	\end{equation}
Note that the limit above exists by Theorem~\ref{T:HK} and
the measure $\umu$ is invariant under 
the shift $S$ of $X^\Z$.  We say that $(X^\Z, \umu,S)$ is the \emph{system of arithmetic progressions with integer steps} associated with the system  $(X,\mu,T)$.		
\end{definition}

\subsubsection{The system of arithmetic progressions with prime steps}
\label{subsec:prime-steps}

The system of arithmetic progressions with prime steps  $(X^\Z,\wt\mu,S)$ was defined in Section~\ref{SS:AP}.
We  recall here the defining property of the measure $\wt\mu$: For every $m\in\N$ and $f_{-m},\ldots, f_m\in L^\infty(\mu)$, we have
$$
\int_{X^\Z}\prod_{j=-m}^m f_j(x_j)\,d\wt\mu(\ux)=		
\E_{p\in\P}\int_X\prod_{j=-m}^m T^{pj}f_j\,d\mu.
$$
Note that convergence of the averages on the right hand side follows from the next result
that  was proved in \cite{FHK}
conditional to some conjectures obtained later in \cite{GT7, GTZ12} and the convergence part was also proved in \cite{WZ11}:
\begin{theorem}
	\label{T:FHK}
	Let $(X,\mu,T)$ be a system. Then
	for every $\ell\in\N$ and  $f_{1},\ldots ,f_\ell\in L^\infty(\mu)$ the following limit exists in $L^2(\mu)$
	\begin{equation}
		\label{E:FHK}
		\E_{p\in\P}\prod_{j=1}^\ell T^{pj}f_j.
	\end{equation}
	Furthermore, if the system is ergodic, $Z_\infty$ is the infinite-step nilfactor of the system (see Appendix~\ref{SS:infnilfac}), and if
	$\E_\mu(f_j\mid Z_\infty)=0$ for some $j\in\{1,\ldots, \ell\}$, then the limit \eqref{E:FHK} is $0$.
\end{theorem}
\begin{remark}  This result is not stated explicitly in \cite{FHK}, but  follows from the argument in  \cite[Section~5]{FHK}, using Theorem~\ref{T:HK} and
$U_{\ell+1}$-uniformity of the $W$-tricked von Mangoldt function (established in  \cite{GT09b, GT7, GTZ12}) in place of $U_3$-uniformity.
\end{remark}
In order to determine the support of the measure $\wt\mu$
we will use the following  multiple ergodic theorem:
\begin{theorem}\label{T:ergid2}
	Let  $(X, \mu, T)$ be a system and suppose that for some $d\in \N$ the ergodic components of the system $(X,\mu,T^d)$ are totally ergodic. Then
	\begin{equation}\label{E:primeform}
		\E_{p\in \P} \prod_{j=1}^\ell T^{pj}f_j=\E_{(k,d)=1} \E_{n\in\N} \prod_{j=1}^\ell T^{(nd+k)j}f_j
	\end{equation}
	for all $\ell \in \N$ and $f_{1},\ldots, f_\ell\in L^\infty(\mu)$, where convergence takes place in $L^2(\mu)$ and  the average $\E_{(k,d)=1}$ is taken over those $k\in\{1,\ldots,d-1\}$ such that $(k,d)=1$.
\end{theorem}
\begin{remark}
	The existence of the limits on the left and right hand side follows from Theorems~\ref{T:FHK} and \ref{T:HK} respectively.
\end{remark}
\begin{proof}
	For $w\in \N$ let  $W$ denote the product of the first $w$ primes that are relatively prime to $d$.
	Following the proof of \cite[Theorem~1.3]{FHK2}  we get that
	the limit on the left hand side of \eqref{E:primeform} is equal to the following limit\footnote{This is established in \cite{FHK2} only for $d=1$ but the same argument works for every $d\in \N$ using the  Gowers uniformity (as $N\to\infty$ and then $W\to \infty$) of the $W$-tricked von Mangoldt function $(\frac{\phi(dW)}{dW}\Lambda(dWn+k)-1)_{n\in [N]}$ for $k\in \N$ relatively prime to $dW$.}
	$$
	\lim_{W\to \infty}\E_{(k,dW)=1}\, \E_{n\in\N} \,\prod_{j=1}^\ell T^{(ndW+k)j}f_j
	$$
	where the average
	$\E_{(k,dW)=1}$ is taken over those $k\in \{1,\ldots, dW-1\}$ such that $(k,dW)=1$.  Since the ergodic components of $T^d$ are totally ergodic, we get by \cite[Theorem~6.4]{Fr04} (see also  Theorem~\ref{T:ergid1} below) that
	$$
	\E_{n\in\N} \prod_{j=1}^\ell T^{(ndW+k)j}f_j=
	\E_{n\in\N} \prod_{j=1}^\ell T^{(nd+k)j}f_j
	$$
	holds for every $W\in \N$. Hence, the limit we want to compute is
	\begin{equation}\label{E:Wm}
		\lim_{W\to \infty}\E_{(k,dW)=1}\E_{n\in\N} \prod_{j=1}^\ell T^{(nd+k)j}f_j.
	\end{equation}
	
	We claim that for general $d$-periodic sequences $(a(k))_{k\in\N}$,  for every $W\in \N$ with $(d,W)=1$ we have
	\begin{equation}\label{E:dW}
	\E_{(k,dW)=1}a(k)=\E_{(k,d)=1}a(k).
	\end{equation}
	To see this, for  $j\in \{0,\ldots, d-1\}$ consider  the set
	$$
	A_j:= \{k\in \{1,\ldots dW\}\colon k\equiv j \! \! \! \pmod{d} \ \text{ and  }\  (k,Wd)=1\}.
	$$
	If $(j,d)>1$, then $A_j=\emptyset$.  If $(j,d)=1$, then $(k,d)=1$ and
	$$
	A_j= \{k\in \{1,\ldots dW\}\colon k\equiv j \! \! \! \pmod{d} \ \text{ and  }\  (k,W)=1\}.
	$$
	Since $(W,d)=1$, we have  $|A_j|=\phi(W)$ if $(j,d)=1$. It follows from these simple facts and
	our assumption of $d$-periodicity of $(a(k))_{k\in\N}$ that \eqref{E:dW} holds.
	
	Applying \eqref{E:dW}   for $a(k):=  \E_{n\in\N} \prod_{j=1}^\ell T^{(nd+k)j}f_j$, $k\in \N$, which is $d$-periodic, we see that the limit in \eqref{E:Wm} is equal to the expression on the right hand side of \eqref{E:primeform}.
	This   completes the proof.
\end{proof}
%
%
%

\subsection{The case of a nilsystem} \label{SS:finitestep} We start with the following intermediate result which establishes  Theorem~\ref{th:main-mutilde} in the case where $(X,\mu,T)$ is a (finite-step) nilsystem:
\begin{proposition}\label{P:nil}
	If $(X,\mu,T)$ is an ergodic nilsystem, then the ergodic components of the systems $(X^\Z,\umu,S)$ and $(X^\Z,\wt\mu,S)$  are isomorphic to  nilsystems.
\end{proposition}

The proof is given in Section~\ref{SS:laststep}. We start with  some preliminaries.
\begin{notation}
	If $T$ is a transformation on $X$, we write $\overline T$ and $\overrightarrow T$
	for the transformations of $X^\Z$ given by
	$$
	(\overline T\ux)_j=Tx_j\ \text{ and } \ (\overrightarrow T\ux)_j=T^jx_j, \quad j\in\Z,
	$$
	where  $\ux=(x_k)_{k\in\Z}\in X^\Z$. We call $\overline T$ the \emph{diagonal transformation}.
	As usual, with  $S$ we denote the shift transformation  on $X^\Z$.
\end{notation}
We remark that $\overline T$ commutes with $\overrightarrow T$ and with $S$, and that $[S,\overrightarrow T]=\overline T$.
\subsubsection{Integer steps}
We use the same hypothesis and notation as in the preceding sections and now
we assume  in addition  that  $X=G/\Gamma$ is a nilmanifold,  $\mu=\mu_X$ is the Haar measure on $X$, and $T$ is an ergodic translation by some $\tau\in G$. Arguing as in \cite[Section~2.1]{leibman1} we can and will assume   that $G$  is spanned by the  connected component $G^0$ of $e_G$ and $\tau$. This condition implies that the groups $G_s$ are connected for every $s\geq 2$ (see \cite[Theorem~4.1]{BHK}).  The transformations $\overline T$ and $\overrightarrow T$ of $X^\Z$ are the translations by $\overline \tau=(\dots,\tau,\tau,\tau\dots)$ and $\overrightarrow\tau=(\dots,\tau^{-2},\tau^{-1},e_G,\tau,\tau^2,\dots)$, respectively.

The Hall-Petresco group $\uG$ and the nilmanifold of arithmetic progressions $\uX$ are defined in the Appendices~\ref{SS:HP} and \ref{subsec:AP-in-nilmanif}.
 It is immediate from the definition of $\uG$ that  $\overline \tau, \overrightarrow\tau\in \uG$. Therefore, $\overline T$ and $\overrightarrow T$ are nilrotations of   $\uX$.
The next result was established in \cite[Lemma~5.2]{BHK}:
\begin{lemma}\label{L:minimal}
 	If $(X,T)$ is a minimal nilsystem then
 	$$
 	\uX=\overline{\bigl\{ \overrightarrow T^n\overline T^me_{\uX}\colon m,n\in\Z\bigr\}}.
 	$$		
\end{lemma}

The next result was established in the form stated in \cite[Theorem~5.4]{BHK} and previously in a slightly different form in \cite{Zi05}:
\begin{proposition}\label{P:Z}
	Let $(X,T,\mu)$ be an ergodic nilsystem. Then
	for every $m\in\N$ and all $f_{-m},\ldots ,f_m\in L^\infty(\mu)$ we have
	$$
	\int_\uX \prod_{j=-m}^m f_j(x_j)\,d\mu_{\uX}(\ux)=\E_{n\in\N}\int_X\prod_{j=-m}^m T^{nj}f_j\,d\mu.
	$$
	In other words, the Haar measure $\mu_\uX$ of $\uX$ coincides with the measure $\umu$ on $\uX$ defined in Definition~\ref{D:under}.
\end{proposition}


\subsubsection{Prime steps}
\label{subsec:prime-AP-nil}
Let  $(X,\mu,T)$ be an ergodic nilsystem. It is a known and easy to prove   fact that this system is totally ergodic if and only if $X$ is connected. In general,  let $X_0$ be the connected component of $e_X$ and $\mu_0$ be its Haar measure. Then there exists $d\in \N$ such that the sets $T^lX_0$, $l \in \{0,\ldots, d-1\}$, form a partition of $X$ and we have
\begin{equation}\label{E:mu0}
\mu=\E_{0\leq l\leq d-1}T^l\mu_0.
\end{equation}
Moreover, the system  $(X_0,\mu_0,T^d)$ and the other ergodic components of the system $(X,\mu,T^d)$ are totally ergodic. We call $d$ the \emph{index} of $X_0$.

Let $\uX_0\subset X_0^\Z$ and the measure $\umu_0$ on $\uX_0$ be defined as $\uX$ and $\umu$ are  defined in Definition~\ref{D:under}, with the system $(X_0,\mu_0,T^d)$ in place of $(X,\mu,T)$.  Then  $\uX_0$ and $\umu_0$ are invariant under $\overline T^d$, $\overrightarrow T^d$, and $S$.
Applying Theorem~\ref{T:ergid2} for the  nilsystem $(X,\mu,T)$ which has index $d$,   we get
that for every  $m\in\N$ and
	$f_{-m},\ldots,f_m\in L^\infty(\mu)$ we have
	\begin{equation}\label{E:dprime}
	\E_{p\in\P}\int_X\prod_{j=-m}^m T^{pj}f_j\,d\mu=
	\E_{(k,d)=1}\E_{n\in\N}\int_X\prod_{j=-m}^m T^{(nd+k)j}f_j\,d\mu
	\end{equation}
where the average $\E_{(k,d)=1}$ is taken over those $k\in\{1,\ldots, d-1\}$ such that $(k,d)=1$.
Combining  \eqref{eq:def-mutilde}, \eqref{E:mu0}, and  \eqref{E:dprime},  we get    for every  $m\in\N$ and
$f_{-m},\ldots,f_m\in L^\infty(\mu)$ that
$$
\int_\uX \prod_{j=-m}^m f_j(x_j)\,d\wt\mu(\ux)=\E_{0\leq l \leq d-1}	\E_{(k,d)=1}\E_{n\in\N}\int_X\prod_{j=-m}^m T^{(nd+k)j+l}f_j\,d\mu_0.
$$
Moreover, applying \eqref{eq:def-mubar} for the system $(X_0,\mu_0,T^d)$ we get
$$
\int_\uX \prod_{j=-m}^m f_j(x_j)\,d\umu_0(\ux)=\E_{n\in\N}\int_X\prod_{j=-m}^m T^{ndj}f_j\,d\mu_0.
$$
Combining the last two identities we deduce that
\begin{equation}
	\label{E:tildemu}
	\wt\mu=
\E_{0\leq l\leq d-1}	\E_{(k,d)=1}\overline T^l\overrightarrow T^k\umu_0.
\end{equation}
Since the support of $\umu_0$ is $\uX_0$, it follows that the measure $\wt\mu$ is  supported on the set
$$
\wt X:=\bigcup_{l=0}^{d-1}\, \bigcup_{k\colon \! (k,d)=1}
 \overline T^l\overrightarrow T^k\uX_0.
$$
The precise form of $\wt X$ is not important,  the crucial point is that
$\wt X\subset \uX$. To see this, note that Lemma~\ref{L:minimal} implies that
the set $\uX$ is $\overline T$ and
$\overrightarrow T$ invariant and
$$
\uX_0=
\overline{\bigl\{ \overrightarrow T^{dn}\overline T^{dm}e_{\uX_0}\colon m,n\in\Z\bigr\}}\subset \uX.
$$
\subsubsection{Proof of Proposition~\ref{P:nil}}\label{SS:laststep}
	Let  
	$\wt\mu=\int \wt\mu_\omega\, dP(\omega)$ be the ergodic decomposition of the measure
	 $\wt\mu$ with respect to the transformation $S$ acting on $X^\Z$. Since as established above
	 $\wt\mu$ is supported on the $S$-invariant set $\uX$, almost every ergodic component
	 $\wt\mu_\omega$ admits a generic point in $\uX$.  For these $\omega$, we  have that
	 $\wt\mu_\omega$  is supported on  a closed $S$-orbit  in $\uX$ which we denote by
	 ${\wt X}_\omega$.
	By Proposition~\ref{P:uX} in the Appendix, the system
	 $({\wt X}_\omega,S)$ is topologically isomorphic to a  uniquely ergodic nilsystem. Thus,  
	 $\wt\mu_\omega$ is the unique invariant measure for the action of $S$ on
	  ${\wt X}_\omega$  and the system
	  	$({\wt X}_\omega,\wt\mu_\omega,S)$  is (measure theoretically) isomorphic to
	  	an ergodic nilsystem.

	  	A similar argument applies to the system $(\uX,\umu,S)$.\qed
	

\subsection{The case of an infinite-step nilsystem}\label{SS:infstep}
Our next goal is to treat the case where $(X,\mu,T)$ is an ergodic  infinite-step nilsystem and prove the following intermediate result:
\begin{proposition}\label{P:infnil}
	If $(X,\mu,T)$ is an ergodic infinite-step nilsystem, then the ergodic components of the systems $(X^\Z,\umu,S)$ and $(X^\Z,\wt\mu,S)$  are  isomorphic to  infinite-step nilsystems.
\end{proposition}
The proof is given in Section~\ref{subsec:proof48}. We start with  some preliminaries.

Our setup is as follows (see Appendix~\ref{A:A} for definitions and properties of inverse limits):
We have $(X,\mu,T)=\varprojlim(X_j,\mu_j,T)$ where for $j\in\N$ the system  $(X_j,\mu_j,T)$ is an ergodic nilsystem with base point $e_{X_j}$. For $j\in\N$, the factor maps are written $\pi_{j,j+1}\colon X_{j+1}\to X_j$ and $\pi_j\colon X\to X_j$ and, as explained in Appendix~\ref{subsec:infinite-step}, $\pi_{j,j+1}$ and $\pi_j$ are also topological factor maps. Thus,  we  also have $(X,T)=\varprojlim(X_j,T)$ in the topological sense (see Appendix~\ref{subsec:infinite-step}).

 The sequence $(X_j^\Z, \overline T,\overrightarrow T)$, $j\in\N$, with factor maps  $\pi_{j,j+1}^\Z\colon X_{j+1}^\Z\to X_j^\Z$, $j\in\N$,  is an inverse system. By the characterization of inverse limits stated in~\eqref{it:inv-lim-8c} and~\eqref{it:inv-lim-10} of Appendix~\ref{subsec:inv-lim-top}, we get that $(X^\Z, \overline T,\overrightarrow T)$, endowed with the factor maps $\pi_j^\Z\colon X^\Z\to X_j^\Z$, $j\in \N$,  is the inverse limit of the sequence $(X_j^\Z, \overline T,\overrightarrow T)$, $j\in\N$.

\subsubsection{Integer steps}


Let $\uX$ be the  orbit closure in  $X^\Z$ of  $e_{\uX}:=(\dots,e_X,e_X,e_X,\dots)$ under the transformations $\overline T$ and $\overrightarrow T$.
 Since $\pi_j^\Z(e_{\uX})=e_{\uX_j}$ for every $j\in \N$, it follows from Lemma~\ref{L:minimal}   and Part~\eqref{it:inv-lim-12} of Lemma~\ref{lem:inv-lim-top} in the Appendix that $\pi_j^\Z(\uX)=\uX_j$, $j\in\N$, and
	$(\uX,\overline T,\overrightarrow T)$ is the inverse limit of the systems $(\uX_j,\overline T,\overrightarrow T)$, $j\in\N$.
	In particular, we have
\begin{equation}
	\label{eq:uX}
	\uX=\bigl\{\ux\in X^\Z\colon\pi_j^\Z(\ux)\in\uX_j\ \text{ for every }j\in \N\bigr\}.
\end{equation}

Note that for $j\in\N$ the maps  $\pi_{j,j+1}^\Z\colon \uX_{j+1}\to\uX_j$ and $\pi_j^\Z\colon \uX\to \uX_j$ commute with the shift transformation $S$, and thus are factor maps from $(\uX_{j+1},S)$ and $(\uX,S)$ to $(\uX_j,S)$, respectively. It follows from the
characterization of topological inverse limits stated in~\eqref{it:inv-lim-8c} and~\eqref{it:inv-lim-10} of   Appendix~\ref{subsec:inv-lim-top} that
$$
(\uX,S)=\varprojlim(\uX_j,S)
$$
with factor maps $\pi_{j,j+1}^\Z\colon \uX_{j+1}\to\uX_j$ and $\pi_j^\Z\colon\uX\to\uX_j$, $j\in \N$.
By Proposition~\ref{P:uX} in the Appendix, for every $j\in \N$ we have that $(\uX_j,S)$ is topologically isomorphic to a nilsystem, hence the action of $S$ on each closed orbit  under $S$ in $\uX_j$ induces a uniquely ergodic nilsystem. From Lemma~\ref{lem:inv-lim-top} in the Appendix we deduce the following:

\begin{proposition}
\label{P:uXinfinite}
Let  $\uX$ be as above and for  $\ux\in\uX$ let $\uX':= \overline{\{S^n\ux\colon n\in \Z\}}$ be the closed orbit of $\ux$ under $S$. Then the system $(\uX',S)$ is topologically isomorphic to  a uniquely ergodic infinite-step nilsystem.
\end{proposition}

\subsubsection{Prime steps}

From  Definition~\ref{D:tilde} it follows that for every $j\in \N$ the image of the measure $\wt\mu$ under the maps $\pi_j^\Z$ is equal to $\wt\mu_j$ and that the image of $\wt\mu_{j+1}$ under $\pi_{j,j+1}^\Z$ is equal to $\wt\mu_j$. These maps commute with $S$, hence
it follows from the characterization of inverse limits~\eqref{it:inv-lim-1} and~\eqref{it:inv-lim-7} given in Appendix~\ref{subsec:inv-lim-ergo} that
\begin{equation}
	\label{eq:tildemu-lim-proj}
	(X^\Z,\wt\mu,S)=\varprojlim(X_j^\Z,\wt\mu_j,S).
\end{equation}
Furthermore, we saw in Section~\ref{subsec:prime-AP-nil} that for every $j\in \N$ the measure $\wt\mu_j$ is supported inside $\uX_j$ and thus
$$
\wt\mu\big(\bigl\{\ux\in X^\Z\colon\pi_j^\Z(\ux)\notin\uX_j\bigr\}\big)=0.
$$
It follows from this and \eqref{eq:uX} that  $\wt\mu$ is supported inside the subset $\uX$ of $X^\Z$.

\subsubsection{Proof of Proposition~\ref{P:infnil}}
\label{subsec:proof48}
In the previous subsection we established that the measure  $\wt\mu$  is supported inside the $S$-invariant set $\uX$. Using this and Proposition~\ref{P:uXinfinite} we deduce  that   almost every ergodic component of the system 
 $(X^\Z,\wt\mu,S)$  is  isomorphic to an infinite-step nilsystem;
the argument is identical to the one used in the last step of the  proof of
 Proposition~\ref{P:nil} (see Section~\ref{SS:laststep}).

A similar argument applies to the system $(X^\Z,\umu,S)$.
\qed

\subsection{General ergodic systems}
Our next goal is to prove the following result which comes very close to establishing Theorem~\ref{th:main-mutilde}:
\begin{proposition}\label{P:ergodic}
	If $(X,\mu,T)$ is an ergodic system,  then almost every  ergodic component of the systems $(X^\Z,\umu,S)$ and $(X^\Z,\wt\mu,S)$  is isomorphic to  a direct product of an infinite-step nilsystem and a  Bernoulli system.
\end{proposition}

This result is proved in Section~\ref{subsec:proof410}. First  we make some preparatory work.

Let $(X,\mu,T)$ be an ergodic system.
The infinite-step nilfactor of the system is defined in Section~\ref{SS:infnilfac}  and is denoted by  $(Z_\infty,\mu_\infty,T)$; in Corollary~\ref{C:HK} we show that
 it is isomorphic to  an infinite-step nilsystem.  Let  $p_\infty\colon X\to Z_\infty$ be the corresponding factor map and
let the measures $\umu_\infty$ and   $\wt\mu_\infty$ on $Z_\infty^\Z$  be  associated with the system $(Z_\infty,\mu_\infty,T)$ as in Definitions~\ref{D:tilde} and \ref{D:under} respectively. Then  $\umu_\infty$ and $\wt\mu_\infty$ are respectively the images of $\umu$ and $\wt\mu$  under $p_\infty^\Z\colon X^\Z\to Z_\infty^\Z$. Combining the second part of Theorems~\ref{T:HK} and \ref{T:FHK} with the definitions of the measures $\umu$ and $\wt\mu$, we get  for every  $m\in\N$ and $f_{-m},\ldots,f_m\in L^\infty(\mu)$ that
$$
\int_{X^\Z}\prod_{j=-m}^mf_j(x_j)\,d\umu(\ux)=\int_{Z_\infty^\Z}\prod_{j=-m}^m\E_\mu(f_j\mid Z_\infty)(z_j)\,d\umu_\infty(\uz)
$$
and
\begin{equation}\label{E:mutil}
\int_{X^\Z}\prod_{j=-m}^mf_j(x_j)\,d\wt\mu(\ux)=\int_{Z_\infty^\Z}\prod_{j=-m}^m\E_\mu(f_j\mid Z_\infty)(z_j)\,d\wt\mu_\infty(\uz).
\end{equation}
\begin{lemma}
	\label{L:mutilde-infty2}
	Let $(X,\mu,T)$ be an ergodic system and
	$(Z_\infty,\mu_\infty,T)$ be
	its infinite-step nilfactor. Then the system    $(X^\Z,\wt\mu,S)$  is isomorphic to the direct product of the system $(Z_\infty^\Z,\wt\mu_\infty,S)$ and a   Bernoulli system (that can be trivial). A similar statement also holds for the system $(X^\Z,\umu,S)$.
\end{lemma}
\begin{proof}[Proof of Lemma~\ref{L:mutilde-infty2}]
We give the argument for the system	$(X^\Z,\wt\mu,S)$; an analogous argument works for the system $(X^\Z,\umu,S)$.
	
 	Since the system $(X,\mu,T)$ is ergodic (and it is our working assumption that it is Lebesgue), it is a classical result of Rohlin (see for example   \cite[Theorem~3.18]{G03}) that there exists a (Lebesgue) probability space $(U,\rho)$ such that the (Lebesgue) probability  spaces $(X,\mu)$ and $(Z_\infty,\mu_\infty)\times(U,\rho)$ are isomorphic,  the factor map $p_\infty\colon X\to Z_\infty$ corresponds to the first  coordinate projection $Z_\infty\times U\to Z_\infty$, and  the conditional expectation $f\mapsto\E(f\mid Z_\infty)$ corresponds to the map $f\mapsto\int f(\cdot,u)\,d\rho(u)$ from $L^1(\mu_\infty\times\rho)$ to $L^1(\mu_\infty)$.  We identify  
	 $x$ with $(z,u)$ and  $\ux$ with $(\uz,\underline{u})$;  then identity \eqref{E:mutil} becomes
	\begin{multline*}
	\int_{X^\Z}\prod_{j=-m}^mf_j(x_j)\,d\wt\mu(\ux)=
	\int_{Z_\infty^\Z}\prod_{j=-m}^m\Bigl(\int_U f_j(z_j,u_j)\,d\rho(u_j)\Bigr)d\wt\mu_\infty(\uz)\\
		=\int_{Z_\infty^\Z\times U^{\Z}}\prod_{j=-m}^m f_j(z_j,u_j)\,d(\wt\mu_\infty\times \rho^{\Z})(\uz,\underline{u})
	\end{multline*}
	where $\rho^\Z$ is the measure $\dots\times\rho\times\rho\times\rho\times\dots$ on $U^\Z$.
	
	Since the algebra generated by functions  of the form $\ux\mapsto f(x_j)$, $j\in \Z$,  $f\in C(X)$, is dense in $C(X^\Z)$  with the uniform topology, we deduce that  	
	$\wt\mu=\wt\mu_\infty\times \rho^\Z$.
	Let $S_1$, $S_2$ denote the shift transformations on the spaces $Z_\infty^\Z$ and  $U^Z$ respectively.
	Then the system  $(X^\Z,\wt\mu,S)$ is  the direct product of the system $(Z_\infty^\Z,\wt\mu_\infty,S_1)$  and the Bernoulli system  $(U^\Z,\rho^\Z,S_2)$. This  completes the proof.
\end{proof}
\subsubsection{Proof of Proposition~\ref{P:ergodic}}
\label{subsec:proof410}
 We give the argument for the system	$(X^\Z,\wt\mu,S)$; an analogous argument works for the system $(X^\Z,\umu,S)$.

  By Lemma~\ref{L:mutilde-infty2}, the  system    $(X^\Z,\wt\mu,S)$  is isomorphic to the direct product of the system $(Z_\infty^\Z,\wt\mu_\infty,S)$ and a Bernoulli system. Since Bernoulli systems are weakly mixing, almost every  ergodic component
 of $(X^\Z,\wt\mu,S)$  is a direct product of an ergodic component of the system  $(Z_\infty^\Z,\wt\mu_\infty,S)$ and the  Bernoulli system given by Lemma~\ref{L:mutilde-infty2} (we used the uniqueness property of the ergodic decomposition here). As explained in Section~\ref{SS:infnilfac} in the Appendix, the system $(Z_\infty,\mu_\infty,T)$ is isomorphic to  an ergodic infinite-step nilsystem, hence  Proposition~\ref{P:infnil} applies and gives that the ergodic components of the system  $(Z_\infty^\Z,\wt\mu_\infty,S)$
 are isomorphic to infinite-step nilsystems. This completes the proof of Proposition~\ref{P:ergodic}. \qed

\subsection{General systems - Proof of Theorem~\ref{th:main-mutilde}}
Let $(X,\mu,T)$ be a system and let  $\mu=\int\mu_\omega\,dP(\omega)$ be the ergodic decomposition of $\mu$ under $T$. It follows from Definition~\ref{D:tilde} that
$$
\wt\mu=\int\wt\mu_\omega\,dP(\omega).
$$
As a consequence, by the uniqueness property of the ergodic decomposition, almost every  ergodic component of the system $(X^\Z,\wt\mu,S)$ is an ergodic component of the system
$(X^\Z,\wt\mu_\omega,S)$ for some $\omega\in \Omega$. We can therefore restrict to the case where
the system  $(X,\mu,T)$ is ergodic. In this case the result follows from Proposition~\ref{P:ergodic}. This completes the proof of Theorem~\ref{th:main-mutilde}.

 A similar argument applies for the system $(X^\Z,\umu,S)$. \qed


\section{Strong stationarity and systems of arithmetic progressions} \label{S:sst}
The goal of this section is to  introduce the notion of  strong stationarity and variants of it that turn out to be
linked to structural properties of systems of arithmetic progressions. We  then use this connection in order
to prove that systems of arithmetic progressions have no irrational spectrum, thus establishing  Theorem~\ref{th:no-irr-eigen}, which in turn gives the first part of Theorem~\ref{th:main-structure} (via Proposition~\ref{prop:mu-fator-tilde-mu}).

\subsection{Strong stationarity}\label{SS:sst}
Throughout this section we continue to denote by $X$ a compact metric space and we equip the sequence space $X^\Z$ with the product topology and the Borel $\sigma$-algebra. With $S$ we denote the shift transformation on $X^\Z$.
With $\CB_0$ we denote all Borel subsets of $X^\Z$ that depend only on the $0$-th coordinate of elements of $X^\Z$. Equivalently, $\CB_0$ consists of sets of the form
$\{x\in X^\Z\colon x(0)\in A\}$ where $A$ is a Borel subset of $X$. We also denote by $\CF_0$ the  algebra of $\CB_0$-measurable functions.

For $r\in\N$ we define the map $\tau_r\colon X^\Z\to X^\Z$ by
$$
(\tau_r(\ux))(j):=x(rj)\ \text{ for }\ux\in X^\Z\text{ and }j\in\Z.
$$
We remark that the maps $S$ and $\tau_r$ satisfy the following    commutation relation
\begin{equation}\label{E:comm}
	S\circ\tau_r=\tau_r\circ S^r.
\end{equation}
The notion of strong stationarity was introduced in a rather abstract setting by Furstenberg and  Katznelson in \cite{FK91}, we use here a variant adapted to our purposes:
\begin{definition}\label{D:sst}
	If $X$ is as above, we say that an $S$-invariant Borel measure $\nu$ on $X^\Z$ is
	\emph{strongly stationary} if it is invariant under $\tau_r$ for every $r\in\N$, and 	\emph{partially strongly stationary} if for some  $d\in \N$ it is invariant  under $\tau_r$ for every $r\in d\N+1$.	 Respectively, we say that the system $(X^\Z,\nu,S)$ is \emph{strongly stationary} and \emph{partially strongly stationary}.
\end{definition}
\begin{remark}
	Equivalently, we have  strong stationarity if and only if
	$$
	\int \prod_{j=-m}^m S^{j}f_j\, d\nu =\int \prod_{j=-m}^m S^{rj}f_j\, d\nu
	$$
	for all $m, r\in \N$ and  
	$f_{-m},\ldots, f_m\in  C(X^\Z)\cap \CF_0$. A similar equivalent condition holds for partial strong stationarity.
\end{remark}
In the next subsection we explain why the notion of partial strong stationarity is  linked to structural properties of systems of arithmetic progressions.

\subsection{Systems of arithmetic progressions and partial strong stationarity}
If a system  
is totally ergodic, then it can be shown that the associated system
of arithmetic progressions with prime and integer steps
is strongly stationary. The notion of total ergodicity turns
out to be  too restrictive, so  we introduce a somewhat  weaker notion that is better adapted to our purposes.
\begin{definition}
	We say that a system $(X,\mu,T)$ has {\em finite rational spectrum} if the set of eigenvalues
	of the system of the form $\e(t)$ with $t\in \Q$ is finite.
\end{definition}
\begin{remark}
	Equivalently, $(X,\mu,T)$  has finite rational spectrum
	if there exists $d\in \N$ such that
	the ergodic components of the system $(X,\mu,T^d)$ are totally ergodic.
\end{remark}
The link between strong stationarity and systems of arithmetic progressions
is given by the next result which is proved in Section~\ref{SS:pfsst} and forms an essential part of
the proof of Theorem~\ref{th:no-irr-eigen}:
\begin{proposition}\label{P:partialsst}
	Let $(X,\mu,T)$ be a system with finite rational spectrum.
	Then the systems  $(X^\Z,\wt\mu,S)$ and $(X^\Z,\umu,S)$ 
	are partially strongly stationary.
\end{proposition}
\begin{remark}

Our argument shows that we get full strong stationarity if the ergodic components of the system $(X,\mu,T)$ are totally ergodic. We do not use this fact though because  we are not able to verify this hypothesis for Furstenberg systems of the   Liouville function.
\end{remark}

\subsubsection{Some multiple ergodic theorems} The proof of Proposition~\ref{P:partialsst} is rather simple but is  based  on some highly non-trivial   known identities involving multiple ergodic averages that we use as a black box. Note that we implicitly assume convergence in $L^2(\mu)$ for all the multiple ergodic averages in this subsection; this is guaranteed to be the case by    Theorems~\ref{T:HK} and \ref{T:FHK}.

The first identity we use   was proved in \cite[Theorem~6.4]{Fr04}:
\begin{theorem}\label{T:ergid1}
	Suppose that the ergodic components of the system  $(X, \mu, T)$ are totally ergodic. Then for every $r\in \N$ we have
	$$
	\E_{n\in \N} \prod_{j=1}^\ell T^{nj}f_j= \E_{n\in\N} \prod_{j=1}^\ell T^{rnj}f_j
	$$
	for all $\ell \in \N$ and $f_{1},\ldots, f_\ell\in L^\infty(\mu)$, where   convergence takes place in $L^2(\mu)$.
\end{theorem}
Combining this result with Theorem~\ref{T:ergid2}  we get the following ergodic theorem that is better adapted to our purposes:
\begin{corollary}\label{C:ergid}
	Let $d\in \N$ and $(X, \mu, T)$ be a system such that the ergodic components of the system  $(X,\mu,T^d)$ are totally ergodic. Then for every  $r\in \N$ with $(r,d)=1$ we have
	$$
	\E_{n\in \N}\prod_{j=1}^\ell T^{nj}f_j=\E_{n\in \N}  \prod_{j=1}^\ell T^{rnj}f_j \quad \text{and} \quad
	\E_{p\in \P}\prod_{j=1}^\ell T^{pj}f_j=\E_{p\in \P}  \prod_{j=1}^\ell T^{rpj}f_j
	$$
	for all $\ell\in \N$ and $f_{1},\ldots, f_\ell\in L^\infty(\mu)$, where convergence takes place in $L^2(\mu)$.
\end{corollary}
\begin{proof}
	We prove the second identity, the proof of the first is similar (simply replace below $p\in \P$ with $n\in\N$ and  $\E_{(k,d)=1}$ with $\E_{k\in [d]}$). Our assumption gives that the ergodic components of $(T^r)^d$ are also totally ergodic.
	By Theorem~\ref{T:ergid2} (applied for $T^r$ in place of $T$), 	we get the identity
	$$
	\E_{p\in \P} \prod_{j=1}^\ell T^{rpj}f_j=\E_{(k,d)=1} \E_{n\in\N}\prod_{j=1}^\ell T^{(dn+k)rj}f_j
	$$
	where  the average $\E_{(k,d)=1}$ is taken over those $k\in\{1,\ldots,d-1\}$ such that $(k,d)=1$.
	Using Theorem~\ref{T:ergid1},   we get that the average on the right hand side is equal to
		$$
	\E_{(k,d)=1} \E_{n\in\N}\prod_{j=1}^\ell T^{(dn+kr)j}f_j=
	\E_{(k,d)=1} \E_{n\in\N} \prod_{j=1}^\ell T^{(dn+k)j}f_j=
	\E_{p\in \P} \prod_{j=1}^\ell T^{pj}f_j,
	$$
	where the first identity follows since 	$(r,d)=1$  and the second from Theorem~\ref{T:ergid2}.
	Combining the above we get the asserted identity.
\end{proof}
\subsubsection{Proof of Proposition~\ref{P:partialsst}}\label{SS:pfsst}
Our assumption gives that there exists $d\in \N$ such that the ergodic components of the system  $(X,\mu,T^d)$ are totally ergodic.
Let   $m\in\N$ and $f_{-m},\ldots, f_m\in C(X^\Z)\cap \CF_0$.  We have
\begin{align*}
	\int_{X^\Z}\prod_{j=-m}^m S^{(dn+1)j}f_j\,d\wt\mu & =	
	\E_{p\in\P}\int_X\prod_{j=-m}^m T^{(dn+1)pj}f_j\,d\mu\\
	&=\E_{p\in\P}\int_X\prod_{j=-m}^m T^{pj}f_j\,d\mu
	=\int_{X^\Z}\prod_{j=-m}^m S^{j}f_j\,d\wt\mu,
\end{align*}
where we used the defining property  of the measure $\wt\mu$ (see Definition~\ref{D:tilde}) to get the first and third identity and the second identity of Corollary~\ref{C:ergid} (for $r:=dn+1$) to get the middle identity.  This proves that the system $(X^\Z,\wt\mu,S)$ is partially strongly stationary.

A similar argument shows that  the system $(X^\Z,\umu,S)$ is partially strongly stationary, the only difference is that one uses
 the  first identity of Corollary~\ref{C:ergid} instead of the second.

\subsection{Spectrum of partially strongly stationary systems}
The next result was obtained in \cite[Section~3]{J}  for ergodic strongly stationary systems, but  the same argument also works with minor modifications for  partially strongly stationary  systems that are not necessarily ergodic.  We will summarize its proof for completeness. Note also that a somewhat more complicated argument can be used to show that a strongly stationary system can only have $1$ in its spectrum (see \cite[Section~4]{J}); but unfortunately a similar result fails for partially strongly stationary systems which can have rational spectrum different than $1$.
\begin{proposition}\label{P:sst-spectrum}
	Let $(X^\Z,\nu,S)$ be 	a partially strongly stationary system. Then the system   has no irrational spectrum.
\end{proposition}

In the proof of Proposition~\ref{P:sst-spectrum}  we will use the following
key property of the maps $\tau_r$:
\begin{lemma}[Lemma~2.3 in \cite{J}]
\label{L:taueigen}
	Let $\chi$ be an eigenfunction of the system $(X^\Z,\nu,S)$ with eigenvalue $\e(t)$ and suppose that for some $r\in \N$ the measure $\nu$ is invariant under $\tau_r$. Then $\chi\circ \tau_r$ is
	a finite linear combination of eigenfunctions for  eigenvalues of the form $\e((j+t)/r)$ for $j=0,\dots,r-1$.
\end{lemma}
\begin{proof}
	For $j=0,\ldots, r-1$ let
	$
	g_{j}:=\sum_{k=0}^{r-1}\, \e(- k(j+t)/r) \, \,
	\chi\circ\tau_r\circ S^k.
	$
	Then direct computation shows that
	$g_{j}\circ S= \e((j+t)/r) \,  g_{j}$, $j=0,\ldots, r-1$
	and that
	$
	\chi=\sum_{j=0}^{r-1}g_{j}.
	$
\end{proof}

We will also use the following  classical variant of van der Corput's fundamental Lemma (the stated version  is from  \cite{Be87a}):
\begin{lemma}[\bf{Van der Corput}]
\label{L:VDC}
	Let  $(v_n)_{n\in\N}$ be a  bounded sequence of vectors in a Hilbert space.
	Suppose that  for each $h\in \N$ we have
	$$
	\E_{n\in\N}\,
	\langle v_{n+h},v_n \rangle=0.
	$$
	Then
	$$
	\E_{n\in\N} \, v_n=0
	$$
	where convergence takes place in  norm.
	\end{lemma}

We are now ready to prove Proposition~\ref{P:sst-spectrum}.
\begin{proof}[Proof of Proposition~\ref{P:sst-spectrum}]
	By our assumption, there exists $d\in \N$ such that the measure $\nu$ is $\tau_r$-invariant for every $r\in d\N+1$.
	
	 Let $\chi\in L^\infty(\mu)$ be such that
	$
	S\chi=\lambda\cdot \chi
	$
	where $\lambda=\e(\alpha)$ with $\alpha$ irrational. We will show that $\chi=0$. To do this we follow closely the argument of Jenvey in \cite[Section~3]{J}.
	
	 Since for $r\in d\N+1$  the maps $\tau_r$ leave the $0$-th coordinate of $x\in X^\Z$  unchanged, we  have
	$f=f\circ \tau_r$ for every $f\in \CF_0$.
	Since linear combinations of functions of the form
	$\prod_{j=-m}^m S^j f_j$ with $f_{-m},\ldots, f_m\in C(X^\Z)\cap \CF_0$, $m\in\N$,  are dense in the space $C(X^\Z)$ with the uniform topology, it suffices to show that
	$$
	\int \chi \cdot \prod_{j=-m}^m S^j f_j\, d\nu=0
	$$
	for all $m\in \N$ and $f_{-m},\ldots, f_m\in C(X^\Z)\cap \CF_0$.
	Composing with the $\nu$-preserving maps  $S^m$ for $m\in \N,$ we see that it suffices to show that
	\begin{equation}\label{E:int0}
		\int \chi \cdot \prod_{j=0}^m S^j f_j\, d\nu=0
	\end{equation}
	for all $m\in \N$ and $f_{0},\ldots, f_m\in C(X^\Z)\cap \CF_0$.

	For $r\in d\N+1$,  we  compose the integrand with the $\nu$-preserving maps $\tau_r$  and then  use the commutation relations \eqref{E:comm} and the fact that $f\circ \tau_r=f$ for $f\in \CF_0$.  We deduce that the integral in  \eqref{E:int0} is equal to
	$$
	\int \chi\circ\tau_r \cdot \prod_{j=0}^m S^{rj} f_j\, d\nu
	$$
	for every $r\in d\N+1$. Averaging over $r\in d\N+1$ gives the identity
	$$
	\int \chi \cdot \prod_{j=0}^m S^j f_j\, d\nu=\E_{n\in\N}\int \chi\circ \tau_{dn+1} \cdot \prod_{j=0}^m S^{(dn+1)j} f_j\, d\nu.
	$$
	Hence, it suffices to show that for every $m\in \N$ and $f_1, \ldots, f_m\in L^\infty(\nu)$ we have
	\begin{equation}\label{E:wanted'}
		\E_{n\in\N}\, \chi\circ \tau_{dn+1} \cdot \prod_{j=1}^m S^{dnj} f_j=0
	\end{equation}
	where the limit is taken in $L^2(\nu)$. Note that from this point on we work with general functions $f_j\in L^\infty(\nu)$, $j=1,\ldots, m$,  not just those in $ C(X^\Z)\cap \CF_0$.
	
	Our first goal   is to successively apply van der Corput's lemma   and the  Cauchy-Schwarz inequality
	in order to reduce our problem  to establishing convergence to zero
	for  an expression that does not depend on the functions $f_1,\ldots, f_m$.
	In our first iteration, we apply  Lemma~\ref{L:VDC},  compose the integrand with $S^{-dn}$,  and use the Cauchy-Schwarz inequality;
	we see that in order to establish \eqref{E:wanted'} it suffices to show that for every $h_1\in \N$ we have
	$$
	\E_{n\in\N}\, S^{-dn}(\chi\circ \tau_{d(n+h_1)+1} \cdot  \overline\chi\circ\tau_{dn+1})
	\prod_{j=1}^{m-1} S^{dnj} f_j=0
	$$
	for all $f_1, \ldots, f_{m-1}\in L^\infty(\nu)$. Note  that the number of functions
	$f_j$ has decreased by one. Note also that by Lemma~\ref{L:taueigen}
	the function
	\begin{equation}\label{E:h_1}
		F_{h_1,n}:= S^{-dn}(\chi\circ \tau_{d(n+h_1)+1} \cdot \overline\chi\circ \tau_{dn+1})
	\end{equation}
	is a finite linear combination of eigenfunctions for $S$ with eigenvalue some root of unity times
	$$
	\e\bigl(\alpha\cdot (\phi(n+h_1)-\phi(n)\bigr)
	$$
	where
	$$
	\phi(n):=\frac{1}{dn+1}, \quad n\in \N.
	$$

	We define inductively the  functions $F_{h_1,\ldots, h_k,n}$, $h_1,\ldots,h_k,n\in \N$ as follows:   For $k=1$ and $h_1,n\in \N$ we let  $F_{h_1,n}$ be as in \eqref{E:h_1} and for $k\geq 2$ and  $h_1,\ldots,h_k,n\in \N$ we let
	$$
	F_{h_1,\ldots, h_k,n}:= S^{-dn}\big(F_{h_1,\ldots, h_{k-1},n+h_k}\cdot \overline{F_{h_1,\ldots, h_{k-1},n}}\big).
	$$
	After successively applying Lemma~\ref{L:VDC}  ($m+1$ times) and using the Cauchy-Schwarz inequality ($m$ times)
	we are left with showing that for every $h_1,\ldots, h_{m+1}\in \N$ we have
	\begin{equation}\label{E:is0}
		\E_{n\in\N} \int F_{h_1,\ldots, h_{m+1},n}\, d\nu=0.
	\end{equation}
	Using  Lemma~\ref{L:taueigen} and the inductive definition of the functions $F_{h_1,\ldots, h_{m+1},n}$, we get that for every $h_1,\ldots, h_{m+1},n\in\N,$ the function
	$F_{h_1,\ldots, h_{m+1},n}$ is a finite linear combination of eigenfunctions with eigenvalue equal to some root of unity times the number
	$$
	\e\big(\alpha \cdot \sum_{\epsilon\in\{0,1\}^{m+1}} (-1)^{|\epsilon|}\phi(n+\epsilon \cdot h)\big)
	$$
	where $h:=(h_1,\ldots, h_{m+1})$,
	$|\epsilon|:=\epsilon_1+\cdots +\epsilon_{m+1}$, and $\epsilon\cdot h:=\epsilon_1h_1+\cdots+ \epsilon_{m+1}h_{m+1}$. Hence,
	\begin{equation}\label{E:is0'}
		\int F_{h_1,\ldots, h_{m+1},n}\, d\nu=0
	\end{equation}
	unless some of the eigenvalues of the eigenfunctions composing the function  $F_{h_1,\ldots, h_{m+1},n}$ is $1$. Since  $\alpha$ is irrational and $\phi$ takes rational values, this
	can only happen if
	\begin{equation}\label{E:is0''}
		\sum_{\epsilon\in\{0,1\}^{m+1}} (-1)^{|\epsilon|}\phi(n+\epsilon \cdot h)=0.
	\end{equation}
Note that for fixed $h=(h_1,\dots,h_{m+1})\in\N^{m+1}$, the left hand side in \eqref{E:is0''} is a rational function in the variable $n$ and  has a pole at $n=0$, hence it  is not identically zero. 	
 After clearing denominators, \eqref{E:is0''} becomes a non-trivial polynomial identity in $n$, hence it can only have finitely many solutions in $n$. We deduce that~\eqref{E:is0'} holds for all large enough $n\in \N$. As a consequence,
	\eqref{E:is0} holds for all  $h_1,\ldots, h_{m+1}\in \N$. As remarked above, this proves that $\chi=0$ and completes the proof.
\end{proof}

\subsection{Proof of Theorem~\ref{th:no-irr-eigen}}
	Let $(X,\mu,T)$ be a system with ergodic decomposition  $\mu=\int\mu_\omega\,dP(\omega)$. It follows from \eqref{eq:def-mutilde} that
	$$
	\wt\mu=\int\wt\mu_\omega\,dP(\omega).
	$$ If $\alpha\in\T$ is irrational and  $\e(\alpha)$ is an eigenvalue of $(X^\Z,\wt\mu,S)$, then for $\omega$ in a set of positive $P$-measure the number
	$\e(\alpha)$ is an eigenvalue of $(X^\Z,\wt\mu_\omega,S)$. It thus suffices  to prove the theorem in the case where $(X,\mu,T)$ is ergodic and we restrict to this case.
	
	Let $(Z_\infty,\mu_\infty,T)$ be
	the infinite-step nilfactor of $(X,\mu,T)$. By Lemma~\ref{L:mutilde-infty2}, the system $(X^\Z,\wt\mu,S)$ is isomorphic to the direct product of the system $(Z_\infty^\Z,\wt\mu_\infty,S)$ and  a Bernoulli system. Since Bernoulli systems are weakly mixing, the system $(X^\Z,\wt\mu,S)$  has the same eigenvalues as the system  $(Z_\infty^\Z,\wt\mu_\infty,S)$. We can therefore restrict to the case where
	$(X,\mu,T)$ is an ergodic infinite-step  nilsystem.

	If  $(X,\mu,T)=\varprojlim (X_j,\mu_j,T)$ where for $j\in \N$ each system $(X_j,\mu_j,T)$ is an ergodic nilsystem, then we get   by~\eqref{eq:tildemu-lim-proj} that
	$$
	(X^\Z,\wt\mu,S)=\varprojlim (X_j^\Z,\wt\mu_j,S).
	$$
	Suppose that $\alpha$ is irrational and $\e(\alpha)$ is an eigenvalue of $(X^\Z,\wt\mu,S)$ with eigenfunction $f$. Then for every large enough $j\in \N$ the conditional expectation of $f$ with respect to $X_j^\Z$ is non-zero, and this function is an eigenfunction of $(X_j^\Z,\wt\mu_j,S)$ with eigenvalue $\e(\alpha)$ as well. Therefore, we can and will restrict to the case where   $(X,\mu,T)$ is an ergodic nilsystem.
	
	If $(X,\mu,T)$ is an ergodic nilsystem, then it has finite rational spectrum. Hence,  Proposition~\ref{P:partialsst}
	applies and gives that the system $(X^\Z, \wt\mu, S)$ is partially strongly stationary.  Proposition~\ref{P:sst-spectrum} then  shows that the system $(X^\Z, \wt\mu, S)$
	has no  irrational spectrum.	This finishes the proof of the absence of irrational spectrum for the system  $(X^\Z,\wt\mu,S)$. \qed
	
	We remark that a similar argument also shows that the system  $(X^\Z,\umu,S)$ has no irrational spectrum.

\subsection{An alternate approach to  Theorem~\ref{th:main-mutilde}}
In \cite{Fr04} it is shown that almost every  ergodic component of a strongly stationary system
is isomorphic to a direct product of an infinite-step nilsystem and a Bernoulli system.
A similar statement with exactly the same proof is valid under the weaker assumption of partial strong stationarity.
If
$(X,\mu,T)$ is an ergodic nilsystem, then it has finite rational spectrum and Proposition~\ref{P:partialsst}
shows that the system  $(X^\Z,\wt\mu,S)$ is partially strongly stationary.
By combining these  results we get   a different proof for a weaker version   of  Proposition~\ref{P:nil}, which  states  that in the case where $(X,\mu,T)$ is an ergodic nilsystem,  the ergodic components of the  system $(X^\Z,\wt\mu,S)$ are direct products of infinite-step nilsystems and Bernoulli systems (note that Proposition~\ref{P:nil} shows that the Bernoulli systems are superfluous). One could use this result as a starting point for an alternate proof of Theorems~\ref{th:main-mutilde} and \ref{th:no-irr-eigen}.
 The  disadvantage of this approach is that we get an unwanted Bernoulli component at a very early stage in the argument which causes some delicate technical problems in the subsequent analysis.

\section{Disjointness result}\label{S:Disjoint}
 The goal of this section is to prove the disjointness result of  Proposition~\ref{P:disjoint}. 
We
start with the following  simpler  result:
\begin{lemma}
	\label{L:disjoint}
	Let $(X,\mu,T)$ be an ergodic infinite-step nilsystem and  $(Y,\nu,R)$ be an ergodic system.
		\begin{enumerate}
		\item If the two systems   have disjoint irrational spectrum, then for every  joining  $\sigma$  of the two systems  and  function $f\in L^\infty(\mu)$  orthogonal to $\CK(T)$, we have
		$$
		\int f(x)\, g(y)\, d\sigma(x,y)=0
		$$
		for every $g\in L^\infty(\nu)$.

		\item If the two systems     have disjoint spectrum different than $1$,  then they are disjoint.
	\end{enumerate}
\end{lemma}
\begin{proof}
We prove part $\text{(i)}$.	We write $(X,\mu,T)=\varprojlim(X_j,\mu_j,T)$, where $(X_j,\mu_j,T)$, $j\in\N$,  are ergodic (finite-step)  nilsystems,  and let $\pi_j\colon X\to X_j$, $j\in \N$,  be the factor maps. Then for every $j\in \N$ the image $\sigma_j$ of $\sigma$ under $\pi_j\times\id \colon X\times Y\to X_j\times Y$ is a joining of $X_j$ and $ Y$ and  for every $f\in L^\infty(\mu)$ and $g\in L^\infty(\nu)$ we have
$$
\int f(x)\, g(y)\, d\sigma(x,y)=\lim_{j\to \infty}	\int (f\circ\pi_j)(x)\, g(y)\, d\sigma_j(x,y).
$$
	Since the function $f$  is orthogonal to $\CK(X,T)$, the function $f\circ\pi_j$  is  orthogonal to
	$\CK(X_j, T)$ for every $j\in \N$. We can therefore restrict to the case where $(X,\mu,T)$ is  an ergodic nilsystem.

	Suppose that  $(X,\mu,T)$ is an ergodic
	$s$-step nilsystem for some $s\in \N$. The eigenfunctions of $X$ associated to rational eigenvalues are  constant on the connected components of $X$. Therefore,  we can approximate in  $L^2(\mu)$ the function $f$ which is orthogonal to $\CK(X,T)$  by a function in $C^\infty(X)$, still orthogonal to $\CK(X,T)$, thus  reducing to the case where
$f\in C^\infty(X)$. Let $g\in L^\infty(\nu)$.
	Since $\sigma$ is $(T\times R)$-invariant we have
		$$
	\int f(x)\, g(y)\, d\sigma(x,y)=\int f(T^nx)\, g(R^ny)\, d\sigma(x,y)
	$$
	for every $n\in \N$.
	We average over $n\in\N$ and reduce  to showing  that
	\begin{equation}\label{E:disjoint'}
		\lim_{N\to\infty}\E_{n\in[N]}\int f(T^nx)\cdot g(R^ny)\, d\sigma(x,y)=0.
	\end{equation}
	
	Since $(X,T)$ is an $s$-step nilsystem and $f\in C^\infty(X)$,
	  it follows from  \cite[Theorem~2.13]{HK09} and the property characterizing the factors $\CZ_s$ given in \eqref{E:Zk} of  Appendix~\ref{SS:infnilfac}, that if
	$g$ is orthogonal to the factor $\CZ_s(R)$,  then
	there exists a set $Y_0$ with $\nu(Y_0)=1$ such that
	for every $y\in Y_0$ we have
	$$
	\lim_{N\to\infty}\E_{n\in[N]} f(T^nx)\cdot g(R^ny)=0
	$$
	for every $x\in X$. This implies that
	the last identity holds for $\sigma$-a.e.  $(x,y)\in X\times Y$ and the bounded convergence theorem gives
	\eqref{E:disjoint'}.
	
	
	
	Hence, we have reduced the problem to verifying \eqref{E:disjoint'} when $g\in \CZ_s(R)$. By Theorem~\ref{T:HK'} in the Appendix, the factor $(Z_s,\CZ_s,\nu_s,R)$ associated with $\CZ_s$ is
	 an inverse limit of ergodic  $s$-step nilsystems. Thus, by $L^2(\nu)$-approximation,
	in order to verify \eqref{E:disjoint'}, we can assume that the system on $Y$ is an
	ergodic $s$-step nilsystem and  $g\in C(Y)$.

	Let $X_0$  be the connected components of $e_X$ in $X$ and   let $\mu_0$ be the Haar measure of this nilmanifold. Then $\mu_0$ is the normalized restriction of $\mu$ to $X_0$. It is a general fact about nilsystems that there exists $k\in\N$ such that the sets $T^jY_0$, $0\leq j<k$, form a partition of $X$ and that $(X_0,\mu_0,T^k)$ is totally ergodic. The rational eigenvalues of $(X,\mu,T)$ are $\e(i/k)$ for $i=0,\dots,k-1$. Let $Y_0$, $\nu_0$ and $\ell$ be defined in the same way as $X_0,\mu_0,k$ was defined with $Y$ substituted for $X$ and let $d$ be the 	least common multiple of $k$ and $\ell$. Then
 $(X_0,\mu_0,T^d)$ and $(Y_0,\nu_0,R^d)$ are totally ergodic and thus have no rational spectrum except $1$. Moreover, if for some irrational $t$ we have that $e(t)$    is a common eigenvalue  for $(X_0,\mu_0,T^d)$ and $(Y_0,\nu_0,R^d)$,
then  $e(t)$ is a common eigenvalue for the systems  $(X,\mu,T^d)$ and $(Y,\nu,R^d)$. It is then
an easy  consequence that the systems $(X,\mu,T)$ and $(Y,\nu,R)$ have a common eigenvalue
of the form $e(s)$ with $s$ irrational (which can be chosen to satisfy $ds=t \bmod{1}$),  contradicting our  assumption that these systems have disjoint irrational spectrum.

We conclude from the previous analysis that   the systems $(X_0,\mu_0,T^d)$ and  $(Y_0,\nu_0,R^d)$ have disjoint  spectrum different than $1$.
As a consequence,  the product system
	$(X_0\times Y_0, \mu_{0}\times\nu_{0}, T^d\times R^d)$  is  ergodic, and since it is a nilsystem, it is uniquely ergodic.
	Let $x\in X, y\in Y$. There exist  $i,j\in \{0,\ldots, d-1\}$ such that $x':=T^{-i}x\in X_0$ and $y':=R^{-j}y\in Y_0$.
	Since the action of $T^d\times R^d$ on $X_0\times Y_0$ is uniquely ergodic, we have
	$$
	\E_{n\in[N]} f(T^{dn}x)\cdot g(R^{dn}y)=\E_{n\in[N]} f(T^{dn+i}x')\cdot g(R^{dn+j}y')\to
	\int T^{i}f\, d\mu_{0}\cdot \int T^{j}g\, d\nu_{0}=0,
	$$
	where the last identity follows since our assumption that  $f$ is orthogonal to $\CK(T)$ implies that   $\int T^{i}f\, d\mu_{0}=0$ for every $i\in \N$. Applying the last identity for  $T^qx, R^r y$ where $q,r\in \{0,\ldots, d-1\}$, in place of $x,y$, we deduce that
	$$
	\lim_{N\to\infty}\E_{n\in[N]} f(T^nx)\cdot g(R^ny)=0
	$$
	holds for every $x\in X, y\in Y$,   and the bounded convergence theorem gives
	\eqref{E:disjoint'}. This completes the proof of part $\text{(i)}$.
	
	We prove part $\text{(ii)}$.  In order to show that the systems are disjoint, it suffices to show  that for all $f\in C^\infty(X)$ and $g\in L^\infty(\nu)$, with $\int g\, d\nu=0$, we have
	\begin{equation}\label{E:product}
		\int f(x) \cdot g(y)\, d\sigma(x,y)=0.
	\end{equation}
As in the proof of part $\text{(i)}$ we  reduce to the case where the system $(X,\mu,T)$ is a  nilsystem. Composing with $(T\times R)^n$ and averaging over $n\in\N$, it thus suffices to show that
\begin{equation}\label{E:disjoint}
	\lim_{N\to\infty}\E_{n\in[N]}\int f(T^nx)\cdot g(R^ny)\, d\sigma(x,y)=0.
\end{equation}

As in the proof of part $\text{(i)}$ we reduce to the case where
 the system $(Y,\nu,R)$ is also a nilsystem, so  now the systems on $X$ and on $Y$ are ergodic nilsystems  with  disjoint  spectrum other than $1$. Then  the product system
$(X\times Y, \mu\times\nu, T\times R)$  is  ergodic and since it is a nilsystem, it is uniquely ergodic.
Hence,   for every $x\in X$ and $y\in Y$ we have
\begin{equation}\label{E:00}
	\lim_{N\to\infty}\E_{n\in[N]} f(T^nx)\cdot g(R^ny)=\int f\, d\mu\cdot \int g\, d\nu=0
\end{equation}
where the second identity follows since by assumption $\int g\, d\nu=0$.
Finally, using \eqref{E:00} and   the bounded convergence theorem we get
\eqref{E:disjoint}. This completes the proof of part $\text{(ii)}$.
\end{proof}

 \begin{lemma}
 	\label{L:disjoint2}
 	Proposition~\ref{P:disjoint} holds under the additional assumption that the system $(X,\mu,T)$ is ergodic.
 \end{lemma}
 \begin{proof}
 		By assumption,  $(X,\mu,T)$ is  the  direct product of an ergodic infinite-step nilsystem $(X',\mu',T')$ and a Bernoulli system $(W,\lambda, S)$.
 	
We prove part $\text{(i)}$.  After identifying $X$ with $X'\times W$,  we have to show that
 	\begin{equation}\label{E:xwy}
 		\int f(x',w)\, g(y)\, d\sigma(x',w,y)=0
 	\end{equation}
 	for every $g\in L^\infty(\nu)$.

 	Using  $L^2(\mu'\times \lambda)$-approximation on the orthocomplement of $\CK(T'\times S)$, we get that it suffices to verify \eqref{E:xwy} when $f(x',w)=f_1(x')\, f_2(w)$ for some $f_1\in L^\infty(\mu')$ and $f_2\in L^\infty(\lambda)$. Since Bernoulli systems are weakly mixing, we get that  $\CK(T'\times S)=\CK(T')$. Hence, our assumption on $f$ translates to the fact that either $\int f_2\, d\lambda=0$ or $f_1$ is orthogonal to $\CK(T')$.

 	Suppose that  $\int f_2\, d\lambda=0$. Let $\tau$ be the image of $\sigma$ under the projection of $X'\times W\times Y$ onto $X'\times Y$. Then $\sigma$ defines a joining of the zero entropy system  $(X'\times Y, \tau, T'\times R)$ and
 	the Bernoulli system $(W,\lambda, S)$. Since these systems are disjoint, we have
 	$\sigma=\tau\times \lambda$. Hence,
 	$$
 	\int f_1(x')\, f_2(w) \, g(y)\, d\sigma(x',w,y)=\int f_1(x')\, g(y)\, d\tau(x',y)\, \int f_2(w)\, d\lambda(w)=0,
 	$$
 	establishing that   \eqref{E:xwy} holds in this case.

 	Suppose now that  $f_1$ is orthogonal to $\CK(T')$. Let $\rho$ be   the image of $\sigma$ under the projection of $X'\times W\times Y$ onto $W\times Y$. Then $\rho$  defines a joining   of the Bernoulli system $(W,\lambda,S)$ and the zero entropy system $(Y,\nu,R)$. Since the systems are disjoint, we have $\rho=\lambda\times\nu$.
 	Hence, we can consider $\sigma$ as a joining of the system $(X',\mu',T')$ and the system   $(W\times Y,\lambda\times \nu,S\times R)$.
 	Since Bernoulli systems are weakly mixing,  the  system on $W\times Y$  is ergodic  and has the same eigenvalues as the system $(Y,\nu,R)$; hence no common irrational eigenvalue with the system $(X',\mu',T')$. It follows that  the assumptions of Part $\text{(i)}$ of  Lemma~\ref{L:disjoint} are satisfied and we conclude that \eqref{E:xwy} holds in this case as well,
 	completing the proof. 	

We prove part $\text{(ii)}$. 	Let $\sigma$ be a joining of the systems on   $X'\times W$ and on $Y$. As in the proof of part $\text{(i)}$ we get that 
$\sigma$ is a joining of the ergodic infinite-step nilsystem  $(X',\mu',T')$ and the ergodic system   $(W\times Y,\lambda\times \nu,S\times R)$ and that these systems have disjoint spectrum other than $1$. It follows that  the assumptions of Part $\text{(ii)}$ of  Lemma~\ref{L:disjoint} are satisfied and we conclude that  $\sigma=\mu'\times\lambda\times \nu$.
Hence,  the systems on $X$ and on $Y$ are disjoint, completing the proof.
 \end{proof}

 We are now ready to complete the proof of  Proposition~\ref{P:disjoint}.

 \begin{proof}[Proof of Proposition~\ref{P:disjoint}]	
 	We write
 	\begin{equation}\label{E:siom}
 	\sigma=\int\sigma_\omega\,dP(\omega)
 	\end{equation}
 	for the ergodic decomposition of the joining $\sigma$ under $T\times R$.  Since the system on $Y$ is ergodic, for almost every  $\omega\in\Omega$ the projection of $\sigma_\omega$ onto $Y$ is equal to $\nu$. We write $\mu_\omega$  for the projection of $\sigma_\omega$ on $X$. Then
 	by the uniqueness property of the ergodic decomposition, we get that for almost every  $\omega\in\Omega$ the measure   $\mu_\omega$ is  $T$-invariant and ergodic,  the measure $\sigma_\omega$ is an ergodic joining of the systems $(X,\mu_\omega,T)$ and $(Y,\nu,R)$, and the following  identity holds
 	\begin{equation}\label{E:muom}
 		\mu=\int \mu_\omega\,dP(\omega).
 	\end{equation}

 We prove part $\text{(i)}$.	Let $\lambda$ be an irrational eigenvalue of $(Y,\nu,R)$. By assumption, $\lambda$ is not an  eigenvalue of $(X,\mu,T)$, hence
 	$$
 	P\big(\bigl\{\omega\colon \lambda\text{ is an eigenvalue of }(X,\mu_\omega,T)\bigr\}\big)=0.
 	$$
 	Since $(Y,\nu,R)$ has at most countably many eigenvalues, it follows that
 	there exists a subset $\Omega_1$ of $\Omega$ with $P(\Omega_1)=1$ and such that for every $\omega\in \Omega_1$  the systems $(Y,\nu,T)$ and $(X,\mu_\omega,T)$ do not have any irrational eigenvalue in  common.
 	Moreover, since $f$ is orthogonal to $\CK(\mu, T)$, there exists $X_1\subset X$ with $\mu(X_1)=1$ and such that
 	$$
 	\E_{n\in \N} \, \e(n \alpha)\, f(T^nx)\to 0\text{ for every } \alpha\in \Q \text{ and every }x\in X_1.
 	$$
 	By \eqref{E:muom},  there exists a subset $\Omega_2$ of $\Omega_1$ with $P(\Omega_2)=1$ and such that for every $\omega\in \Omega_2$ we have   $\mu_\omega(X_1)=1$ and  the convergence above holds for $\mu_\omega$ almost every $x\in X$.
 	We conclude that  for every $\omega\in \Omega_2$ the function $f$ is orthogonal to $\CK(\mu_\omega, T)$.

 	From the above discussion we have that  for every $\omega\in \Omega_2$ the hypothesis of Part $\text{(i)}$ of  Lemma~\ref{L:disjoint2} is  satisfied for the function $f$ and  the joining $\sigma_\omega$ of the systems $(X,\mu_\omega,T)$ and $(Y,\nu,R)$. We deduce that  for every $\omega\in \Omega_2$ we have
 	$$
 	\int f(x) \, g(y)\,d\sigma_\omega(x,y)=0
 	$$
 	for every $g\in L^\infty(\nu)$.
 	Since $P(\Omega_2)=1$,  it follows from \eqref{E:muom} that
 	$$
 	\int f(x)\, g(y)\,d\sigma(x,y)=0
 	$$
 	for every $g\in L^\infty(\nu)$. This completes the proof of part $\text{(i)}$.

 	We prove part $\text{(ii)}$. As in the first part we show that for $P$-almost every $\omega\in \Omega$ the systems $(Y,\nu,R)$ and $(X,\mu_\omega,T)$ have disjoint spectrum other than $1$. Hence, Part $\text{(ii)}$  of Lemma~\ref{L:disjoint2} applies and gives that these two systems are disjoint and thus $\sigma_\omega=\mu_\omega\times\nu$ for almost every $\omega\in \Omega$. Therefore, by \eqref{E:siom} and \eqref{E:muom} we get   $\sigma=\mu\times\nu$. This completes the proof of part $\text{(ii)}$.
 \end{proof}

\section{Subshifts with linear block growth and proof of Theorem~\ref{th:Liou-non-linear}}
\label{S:linear-growth}
 The goal of this section is to deduce Theorem~\ref{th:Liou-non-linear} from Theorem~\ref{th:main-new} and some
 facts about invariant measures of subshifts with linear block growth.

\subsection{Measures on a subshift with  linear block  growth}
We start with some definitions.
Let $A$ be  a non-empty finite set whose elements are called \emph{letters}. $A$ is endowed with the discrete topology and $A^\Z$ with the product topology and with the shift $T$. For $n\in\N$, a \emph{word of length $n$}  is a sequence $u=u_1\dots u_n$ of $n$ letters (we omit the commas), and we write $[u]=\{x\in A^\Z\colon x_1\dots x_n=u_1\dots u_n\}$.

A \emph{subshift}, also called a \emph{symbolic system}, is a closed non-empty $T$-invariant subset $X$ of $A^\Z$. Recall that $X$ is transitive if it has at least one dense orbit under $T$.

Let $(X,T)$ be a transitive subshift, equal to the closed orbit of some point $\omega\in A^\Z$.
For every $n\in \N$ we let $L_n(X)$ denote the set of words $u$ of length $n$ such that $[u]\cap X\neq\emptyset$.   Then $L_n(X)$ is also the set of words of length $n$ that occur (as consecutive values) in $\omega$. Note that  the set $L(X):=\bigcup_{n\in \N}L_n(X)$ determines $X$.    The \emph{block complexity} of $X$ or of $\omega$ is defined by $p_X(n)=|L_n(X)|$ for $n\in\N$.
We say that the subshift $(X,T)$ (or the sequence $\omega$) \emph{has linear block growth} if
$\liminf_{n\to\infty} p_X(n)/n<\infty$.

\begin{proposition}
	\label{prop:linear-grow}
	Let $(X,T)$ be a transitive subshift with linear block growth. Then $(X,T)$ admits only finitely many ergodic invariant measures.
\end{proposition}
This result was proved in~\cite{Bos} under the stronger hypothesis that $(X,T)$ is minimal. In order to
replace this hypothesis with transitivity we will use a result from \cite{CK} (alternatively we could use ~\cite[Theorem 7.3.7]{FMo}) which treats the case of non-atomic invariant measures.

\begin{proof}[Proof of Proposition~\ref{prop:linear-grow}]
	Let $X$ be the closed orbit under $T$ of some $\omega\in A^\Z$ and suppose  that  the subshift $(X,T)$ has linear block growth. If $\omega$ is periodic, then $X$ is  a finite orbit and the shift transformation on $X$  admits only one invariant measure; hence, we can restrict  to the case where $\omega$ is not periodic.
	Let $K$ be an integer such that $\liminf_{n\to\infty} p_X(n)/n\leq K$. Then for infinitely many $n\in \N$ we have $p_X(n+1)-p_X(n)\leq K$.
	
	We say that a word $u\in L_n(X)$ is right special if there exist two different letters $a,b\in A$ such that $ua$ and $ub$ belong to $L_{n+1}(X)$. The number of right  special words of length $n$ is clearly bounded by $p_X(n+1)-p_X(n)$. The left special words of length $n$ are defined in a similar way and their number is also bounded by $p_X(n+1)-p_X(n)$. By a \emph{special word} of length $n$  we mean a left or right special word. Then for  infinitely many values of $n\in \N$ there are at most $2K$ special words of length $n$.

	We claim that for every finite orbit $Y$ in $X$ and every $n\in\N$, the set  $L_n(Y)$ contains a special word. Suppose that this is not the case.
	Let $x\in Y$. Since the orbit of $\omega$ is dense in $X$, there exists $k\in\Z$ such that the words $\omega_{k+1}\dots\omega_{k+n}$ and $x_1\dots x_n$ are equal. We show that $T^k\omega=x$. We claim first that for $\ell\geq 1$ we have $\omega_{k+\ell}=x_\ell$.
	For $1\leq\ell\leq n$ there is nothing to prove. Suppose that this property holds until some $\ell\geq n$. Then the words $\omega_{k+\ell-n+1}\dots\omega_{k+\ell}$ and $x_{\ell-n+1}\dots x_\ell$ are equal, and since $x\in Y$, this word belongs to $L_n(Y)$ and thus is not right special. Since
	$\omega_{k+\ell-n+1}\dots\omega_{k+\ell}\omega_{k+\ell+1}$ and $x_{\ell-n+1}\dots x_\ell x_{\ell+1}$ belong to $L_{n+1}(X)$, we  have
	$\omega_{k+\ell+1}=x_{\ell+1}$, and the claim is proved. In the same way, using now the fact that $L_n(Y)$ does not contain any left special word, we obtain that   $\omega_{k+\ell}=x_\ell$ for $\ell\leq 0$ and we conclude that  $T^k\omega=x$. Since the orbit of $\omega$ is dense we deduce that
	$X=Y$ and thus $\omega$ is periodic. This contradicts our assumption and proves  the claim.

	We claim now that $X$ contains at most $2K$ distinct  finite orbits.  Suppose that this is not the case
	and that $Y_1,\dots,Y_{2K+1}$ are distinct finite orbits. Then the set $Y_j$, $j=1,\dots,2K+1$, are closed, invariant, pairwise disjoint, and it follows that for every sufficiently large $n\in \N$  the sets $L_n(Y_1),\dots,L_n(Y_{2K+1})$ are pairwise disjoint. Let $n\in \N$ be chosen so that there are at most $2K$ special words of length $n$. By the preceding step, each set $L_n(Y_j)$ contains a special word, and since these words are distinct,  we have a contradiction and the  claim is proved.

	By~\cite{CK},  the subshift $(X,T)$ has only finitely many non-atomic ergodic measures. Each atomic ergodic invariant measure is the uniform measure of a finite orbit, and we previously showed that there are at most $2K$ such orbits, hence there are at most $2K$ such measures. This completes the proof.
\end{proof}


\subsection{Proof of Theorem~\ref{th:Liou-non-linear}}
Suppose  that $\blambda$ has linear block growth. We extend $\blambda$ to a two sided sequence, written also $\blambda\in\{-1,1\}^\Z$, by letting $\blambda(n)=1$ for non-positive $n\in \Z$;  then the extended sequence  still has linear block growth.  Let $Y$ be the closed orbit of $\blambda$ in $\{-1,1\}^\Z$ and let $R$ be the shift on $Y$. Then $(Y,R)$ is a transitive subshift, and since it has linear block growth  it has zero topological entropy. Moreover,  by Proposition~\ref{prop:linear-grow} this system admits only finitely many ergodic invariant measures. Note that for every $n\in \N$ we have $\blambda(n)=F_0(R^n\blambda)$, where $F_0\colon \{-1,1\}^\Z\to\R$ is the map $x\mapsto x_0$. By Theorem~\ref{th:main-new} we get
$$
0= \lim_{N\to\infty} \frac 1{\log N}\sum_{n=1}^N\frac{F_0(R^n\blambda)\, \blambda(n)}{n} =\lim_{N\to\infty} \frac 1{\log N}\sum_{n=1}^N\frac {\blambda(n)^2}{n}=1,
$$
a contradiction. \qed

\appendix

\section{Inverse limits and infinite-step nilsystems}\label{A:A}
\subsection{Inverse limits in ergodic theory}
\label{subsec:inv-lim-ergo}
Let $(X_j,\CX_j,\mu_j,T_j)$, $j\in \N$,  be  measure preserving  systems and
let $\pi_{j,j+1}\colon X_{j+1}\to X_j$, $j\in \N$,  be  factor maps. We say that $(X_j,\pi_{j,j+1}\colon j\in\N)$ is an \emph{inverse sequence} of systems.
An  \emph{inverse limit} of this inverse sequence is defined  to be a system $(X,\CX,\mu, T)$ endowed with factor maps $\pi_j\colon X\to X_j$, $j\in \N$,  satisfying the following two properties:
\begin{enumerate}
	\item
	\label{it:inv-lim-1}
	$\pi_j=\pi_{j,j+1}\circ\pi_{j+1}$ for every $j\in\N$;
	\item
	\label{it:inv-lim-7}
	$\CX=\vee_{j\in\N}\pi_j\inv(\CX_j)$. 
\end{enumerate}
For a given inverse sequence of systems the existence of an inverse limit can be shown by an explicit construction.
Properties~\eqref{it:inv-lim-1} and~\eqref{it:inv-lim-7} characterize the system  $(X,\mu,T)$ up to isomorphism, thus we can say that $(X,\mu, T)$, endowed with the factor maps $\pi_j$, $j\in \N$,  is {\em the}  inverse limit instead of {\em an} inverse limit, and write
$$
(X,\mu,T)= \varprojlim(X_j,\mu_j,T_j)
$$
when the factor maps are clear from the context.

A typical example is when a system  $(X,\CX,\mu,T)$ is given and for $j\in \N$  the systems on
 $X_j$ are the ones associated to an increasing sequence $\CX_j$ of $T$-invariant sub-$\sigma$-algebras of $\CX$.
Then the inverse limit of this inverse sequence can be defined as the factor of $\CX$ associated with the $T$-invariant sub-$\sigma$-algebra $\CX':=\vee_{j\in\N}\pi_j^{-1}(\CX_j)$.

We record some easy but important properties of inverse limits:
\begin{lemma}
	\label{lem:inv-lim-erg}
Suppose  that $(X,\mu,T)=\varprojlim(X_j,\mu_j,T_j)$. Then
	\begin{enumerate}
		\item
		$(X,\mu,T)$ is ergodic if and only if $(X_j,\mu_j,T_j)$ is ergodic for every $j\in \N$.
		\item
		A complex number of modulus $1$ is an eigenvalue of $(X,\mu,T)$ if and only if it is an eigenvalue of  $(X_j,\mu_j,T_j)$ for every sufficiently large $j\in \N$.
	\end{enumerate}
\end{lemma}


\subsection{Inverse limits of topological dynamical systems}
\label{subsec:inv-lim-top}
Let $(X_j,T_j)$, $j\in\N$,  be  topological dynamical systems
and
$\pi_{j,j+1}\colon X_{j+1}\to X_j$, $j\in \N$,  be  factor maps. We say that $(X_j,\pi_{j,j+1}\colon j\in\N)$ is an \emph{inverse sequence} of topological dynamical systems.
An  \emph{inverse limit} of this inverse sequence is defined  to be a topological dynamical system $(X, T)$ endowed with factor maps $\pi_j\colon X\to X_j$, $j\in \N$,  satisfying
\begin{enumerate}
	\item
	\label{it:inv-lim-8c}
	$\pi_j=\pi_{j,j+1}\circ\pi_{j+1}$ for every $j\in\N$;
	\item
	\label{it:inv-lim-10}
	If $x,x'\in X$ are distinct, then  $\pi_j(x)\neq\pi_j(x')$ for some $j\in \N$.
\end{enumerate}
Again, for a given inverse sequence of topological systems the existence of an inverse limit can be established by an explicit construction.
Properties~\eqref{it:inv-lim-8c} and~\eqref{it:inv-lim-10} characterize the system $(X,T)$ up to isomorphism. We state the following  easy but important properties:
\begin{lemma}
	\label{lem:inv-lim-top}
	Suppose  that $(X,T)=\varprojlim(X_j,T_j)$ with factor maps $\pi_j\colon X\to X_j$, $j\in\N$.  Then
	\begin{enumerate}
		\item\label{it:inv-lim-12}
		Let $x\in X$ and $Y$ be the orbit closure of $x$ under $T$. Then for every $j\in\N$, $\pi_j(Y)$ is the orbit closure of  $\pi_j(x)$ under $T$ and $(Y,T)=\varprojlim(\pi_j(Y),T_j)$.
		\item\label{it:inv-lim-13}
		If $(X_j,T_j)$ is minimal for every $j\in \N$, then $(X,T)$ is minimal.
			\item\label{it:inv-lim-14}
		If $(X_j,T_j)$ is uniquely ergodic for every $j\in \N$, then $(X,T)$ is uniquely ergodic.
	\end{enumerate}
\end{lemma}

We verify the third property only. Let $\mu,\mu'$ be two $T$-invariant measures on $X$. For every $j\in \N$ the system  $(X_j,T_j)$ is uniquely ergodic with invariant measure $\mu_j$. Hence, for every $j\in \N$ the images of $\mu$ and $\mu'$ under $\pi_j$ are equal to $\mu_j$,  and $\int f\circ\pi_j\,d\mu=\int f\circ\pi_j\,d\mu'$ for every $f\in C(X_j)$. It follows from Property~\eqref{it:inv-lim-10} of topological inverse limits and the Stone-Weierstrass theorem that the collection of  functions $f\circ\pi_j$ where  $f\in C(X_j)$ and $j\in\N$ is dense in $C(X)$ with the uniform norm. We conclude that $\mu=\mu'$. Hence, the system $(X,T)$ is uniquely ergodic.

Up to notational changes, all definitions and results of Sections~\ref{subsec:inv-lim-ergo} and~\ref{subsec:inv-lim-top} remain valid for systems with several commuting transformations.

\subsection{Infinite step nilsystems}
\label{subsec:infinite-step}
Let  $(X_j,\mu_j,T_j)$, $j\in\N$,  be ergodic nilsystems  and $\pi_{j,j+1}\colon X_{j+1}\to X_j$, $j\in \N$,  be factor maps. By~\cite[Theorem~3.3]{parry3},\footnote{In \cite{parry3}  the result is given only when the groups defining the nilmanifolds are connected, but the proof extends to the general case. Another proof is implicit in~\cite[Section 6]{HKM}; see also~\cite[Chapter XII]{HK17}.} for every $j\in \N$ the measure theoretic factor map
$\pi_{j,j+1}\colon X_{j+1}\to X_j$ agrees almost everywhere with a topological factor map which we also denote by $\pi_{j,j+1}$.  Therefore, the topological dynamical systems $(X_j,T_j)$,  with factor maps $\pi_{j,j+1}$, $j\in\N$, form an inverse system. Let $(X,T)$ be the inverse limit of this sequence, and $\pi_j\colon X\to X_j$ be  the associated factor maps.
By Part~\eqref{it:inv-lim-14} of  Lemma~\eqref{lem:inv-lim-top}, the system  $(X,T)$ is uniquely ergodic.
Let $\mu$ be the unique  invariant measure of $(X,T)$. Then the Properties~\eqref{it:inv-lim-1} and~\eqref{it:inv-lim-7} of Section~\ref{subsec:inv-lim-ergo} are satisfied and
$(X,\mu,T)=\varprojlim(X_j,\mu_j,T_j)$.

We use the following terminology from~\cite{DDMSY}:
\begin{definition}
	We say that a measure preserving system $(X,\mu,T)$ is an \emph{ergodic infinite-step nilsystem} if it is the inverse limit of a sequence $(X_j,\mu_j,T_j)$, $j\in \N$,  of  ergodic nilsystems.
	By the preceding discussion, the topological dynamical system $(X,T)$   is then the inverse limit of the minimal nilsystems $(X_j,T_j)$, $j\in \N$,  and we say that $(X,T)$ is a \emph{minimal  infinite-step nilsystem}. We often abuse notation and denote the transformation $T_j$ on $X_j$ by $T$.
\end{definition}
We caution the reader that if $s_j$ is the degree of nilpotency of the nilmanifolds $X_j$, $j\in \N$,  then the sequence $(s_j)_{j\in\N}$ may be unbounded.

It follows from Property~\eqref{it:inv-lim-14} of Lemma~\ref{lem:inv-lim-top} and the well known fact that minimal (finite-step) nilsystems are uniquely ergodic, that minimal infinite-step nilsystems are uniquely ergodic.
\begin{lemma}
\label{lem:joining-infinite}
An ergodic joining of two ergodic finite or  infinite-step nilsystems is a finite or an infinite-step nilsystem respectively.
\end{lemma}
\begin{proof} We give the argument for infinite-step nilsystems only, the other case is similar.
Let $\sigma$ be an ergodic joining of the ergodic infinite-step nilsystems $(X,\mu,T)$ and $(X',\mu',T')$. We write $(X,\mu,T)=\varprojlim_j(X_j,\mu_j,T_j)$ and $(X',\mu',T')=\varprojlim_j(X'_j,\mu'_j,T'_j)$ where
the systems on $X_j$ and $X_j'$  are ergodic nilsystems for every $j\in \N$. For $j\in \N$ let $\sigma_j$ be the projection of $\sigma$ on $X_j\times X'_j$;  then $\sigma_j$ is an ergodic joining of the systems on  $X_j$ and $X'_j$. By \cite[Theorems~2.19 and 2.21]{leibman1}, for $j\in \N$,  the measure $\sigma_j$
 is the Haar measure on some sub-nilmanifold of the product nilmanifold $X_j\times X'_j$ , hence $(X_j\times X'_j,\sigma_j,T_j\times T_j')$ is an ergodic nilsystem. Since
$(X\times X',\sigma,T\times T')=\varprojlim_j(X_j\times X'_j,\sigma_j,T_j\times T_j')$, the result follows.
\end{proof}

\subsection{The infinite-step nilfactor of a system}\label{SS:infnilfac}
Let $(X,\mu,T)$ be an ergodic system and for  $k\in \N$ let
$(Z_k, \CZ_k, \mu_k,T)$ be the factor of order $k$ of $X$ as defined in~\cite{HK}.
In \cite{HK}  it is shown that $Z_k$ is characterized by the following property:
\begin{equation}\label{E:Zk}
	\text{\em for } \ f\in L^\infty(\mu),\ \ \E(f|Z_{k})=0\quad \text{ \em if and
		only if } \quad   \nnorm f_{k+1} = 0,
\end{equation}
where the seminorms $\nnorm{\cdot}_k$ are defined inductively
as follows: for  $f\in L^\infty(\mu)$ we let
	$\nnorm{f}_{1}\mathrel{\mathop:}=\big| \int f \ d\mu\big|$ and
	$\nnorm f_{k+1}^{2^{k+1}} \mathrel{\mathop:}=\E_{n\in\N}
	\nnorm{\bar{f}\cdot T^nf}_{k}^{2^{k}}$ for $k\in \N$, where all limits can be shown to exist.

The following  result was  proved in \cite{HK}:
\begin{theorem}\label{T:HK'}
If $(X,\mu,T)$ is an ergodic system, then the system $(Z_k,\CZ_k, \mu_k,T)$ is an inverse limit of ergodic $k$-step nilsystems.
\end{theorem}
The factors $\CZ_k,$ $k\in\N$,  form an increasing sequence of $T$-invariant  sub-$\sigma$-algebras of $\CX$ and we let
$\CZ_\infty:=\vee_{k\in \N}\CZ_k$ and  $(Z_\infty,\CZ_\infty,\mu_\infty,T)$ be the factor system associated with the $\CZ_\infty$. Then, this system is the inverse limit of the systems $(Z_k,\CZ_k, \mu_k,T)$, $k\in \N$.
\begin{corollary}\label{C:HK}
If $(X,\mu,T)$ is an ergodic system, then 	$(Z_\infty,\mu_\infty,T)$  is an ergodic infinite-step nilsystem.
\end{corollary}
\begin{proof}
	For $k\in \N$ we write $(Z_k,\mu_k,T)=\varprojlim_j (Z_{k,j},\mu_{k,j}, T_j)$ where the systems on $Z_{k,j}$ are ergodic  $k$-step ergodic nilsystems for every $j\in \N$.
	For  $\ell\in \N$ let $(Y_\ell,\nu_\ell,T)$ be the factor of $X$ associated with the $\sigma$-algebra
	$$
	\CY_\ell:=\bigvee_{k,j\in\N,\ k+j\leq\ell}\, \CZ_{k,j}.
	$$
 Then the system on $Y_\ell$ is an ergodic joining of the nilsystems on $Z_{k,j}$ with $k+j\leq\ell$. Hence,
 Lemma~\ref{lem:joining-infinite} gives that
	$(Y_\ell,\nu_\ell,T)$ is an ergodic nilsystem. Moreover, for every $\ell\in \N$ and for all $k,j\in \N$ with $k+j\leq\ell$ we have  $\CZ_{k,j}\subset \CZ_\ell$ and thus $\CY_\ell\subset\CZ_\ell$ and $\vee_\ell\CY_\ell\subset \CZ_\infty$. Conversely, for every $k\in \N$ we have $\CY_{k+j}\supset \CZ_{k,j}$ for every $j\in\N$, hence $\vee_\ell\CY_\ell=\vee_j\CY_{k+j}\supset \vee_j\CZ_{k,j}=\CZ_k$. Therefore, $\vee_\ell\CY_\ell\supset \CZ_\infty$ and we have equality $\vee_\ell\CY_\ell= \CZ_\infty$. By the characterization \eqref{it:inv-lim-1} and~\eqref{it:inv-lim-7} of inverse limits, we deduce that  $(Z_\infty,\mu_\infty,T)=\varprojlim_\ell(Y_\ell,\nu_\ell,T)$ and thus $(Z_\infty,\mu_\infty,T)$ is an infinite-step nilsystem.
 \end{proof}


%
\section{The nilmanifold and nilsystem of arithmetic progressions}\label{A:AP}
A key step in the proof of Theorem~\ref{th:main-structure}  is to determine  the structure of the system of arithmetic progressions with integer steps (see Definition~\ref{D:under})  in the case where the base system
is a nilsystem. We are thus naturally led to study configurations defined by  arithmetic progressions
on $G^\Z$, where $G$ is some nilpotent group,  of the form $(\ldots, h^{-2}g, h^{-1}g, g, hg, h^2g,\ldots)$,  where $g,h\in G$. It turns out that such configurations are not closed under pointwise multiplication and the smallest closed  subgroup of $G^\Z$ that contains these ``arithmetic progressions'' is the Hall-Petresco group that we define next.  An extensive study of arithmetic progressions in a nilpotent group and in a nilmanifold can be found in~\cite[Chapter XIV]{HK17}
and in \cite{HK17}.
 \subsection{The  group of arithmetic progressions}\label{SS:HP}
Let $s\in\N$ and let $X=G/\Gamma$ be an $s$-step nilmanifold.
We write
$$G=G_0=G_1\supset G_2\supset \dots\supset G_s\supset G_{s+1}=\{e_G\}$$
for the lower central series of $G$. 	 We denote by $\mu_X$ the Haar measure of $X$ and by $e_X$ the image of $e_G$ in $X$. The action of $G$ on $X$ is written $(g,x)\mapsto g\cdot x$.

We use the following convention for binomial coefficients with  negative entries:
$$
\binom nm=\frac{n(n-1)\cdots(n-m+1)}{m!}, \quad n\in \Z, \, m\geq 0,
$$
where the empty product is equal to $1$ by convention.

We write $\uG$ for the set of sequences $\ug=(g_j)_{j\in\Z}$ given by
\begin{equation}
	\label{eq:HP}
	g_j= a_0a_1^{\binom j1}a_2^{\binom j2}\cdots 
	a_s^{\binom js},\quad j\in\Z,
\end{equation}
where $a_m\in G_m$  for $m=0,1,\dots,s$.

It is known since the work of Hall~\cite{Hall} and Petresco~\cite{Pe} that $\uG$ forms a group with respect to  pointwise multiplication. This group is called the \emph{Hall-Petresco group of $G$} and  was extensively studied  by Leibman~\cite{leibman} and later  by Green and Tao~\cite{GT1,GT3}.

Elements of $\uG$ have  the following useful equivalent  characterization:
For $\ug=(g_j)_{j\in\Z}$ in $G^\Z$, let $\partial\ug\in G^\Z$ be defined by
$$
(\partial \ug)_j:=g_{j+1}g_j\inv, \quad j\in\Z.
$$
In other words, $\partial \ug= \sigma\ug\cdot \ug\inv$ where $\sigma\colon G^\Z\to G^\Z$ is the shift   defined by
$$
(\sigma(\ug))_j:=g_{j+1},  \quad \ug\in G^\Z, \ j\in\Z.
$$
For $m\in \N$ we let $\partial ^{\circ m} :=\partial\circ\dots\circ\partial$ ($m$ times). The next result was proved in
\cite[Proposition~3.1]{leibman} and also in \cite{lazard}:
\begin{lemma}\label{L:leibman}
	An element $\ug\in G^\Z$ belongs to $\uG$ if and only if for every $m\in \N$  we have  $\partial^{\circ m}\ug\in G_m^\Z$.
\end{lemma}
We immediately deduce from Lemma~\ref{L:leibman} the following basic properties:
\begin{itemize}
	\item  $\uG$ is invariant under the shift $\sigma\colon G^\Z\to G^\Z$.
	
	\item $\partial^{\circ(s+1)} \ug=e_{\uG}$ for every $\ug\in \uG$, that is,  $\sigma$ is a unipotent automorphism of $\uG$.
	
	\item  $\uG$ is a closed subgroup of $G^\Z$.
\end{itemize}

\subsection{The nilmanifold of arithmetic progressions}
\label{subsec:AP-in-nilmanif}

Let $X^\Z$ be endowed with the action of $\uG$ given by $(\ug\cdot\ux)_j=g_j\cdot x_j$ for $\ug\in\uG$, $\ux\in X^\Z$, and $j\in\Z$. If $e_{\uX}=(\dots,e_X,e_X,e_X,\dots)$ we define
$$
\uX:= \uG\cdot e_\uX= \bigl\{ (g_j\cdot e_X)_{j\in\Z}\colon (g_j)_{j\in \Z}\in \uG\bigr\}.
$$
The  stabilizer  of $e_{\uX}$ in $\uG$ is the subgroup $\uGamma:=\uG\cap\Gamma^\Z$ and thus we have
$$
\uX=\uG\,/\,\uGamma.
$$

%
%
%
%
%
%
%

 A priori, $\uX$ is an infinite dimensional object, but it will be convenient for us to represent  it as a nilmanifold, in order to be able to apply   the machinery of nilmanifolds.
 To this end, we show   that $\uG$ can be represented as a subgroup of $G^{s+1}$ and $\uX$ as a sub-nilmanifold of $X^{s+1}$.
We make use of the next lemma that follows from Lemma~\ref{L:leibman} and
was established by Green and Tao in the course of proving Lemma~14.2 in \cite{GT1}.
\begin{lemma}
The projection homomorphism
$$
	p\colon \uG\to G^{s+1}\text{ given by }p(\ug):=(g_0,g_1,\dots,g_s)
$$
is one to one and satisfies $p\inv(\Gamma^{s+1})=\uGamma$.
Furthermore, the projection
$$
	q\colon\uX\to X^{s+1}
	\text{ given by }q(\ux):=(x_0,x_1,\dots,x_s)
$$
is one to one.
\end{lemma}
We let
$$
\uG':=p(\uG), \quad \uGamma':=p(\uGamma)=\uG\cap \Gamma^{s+1}, \quad  \uX':=q(\uX).
$$
Writing  $\underline e_X':=(e_X,e_X,\dots,e_X)\in X^{s+1}$, we have  $\uX'=\uG'\cdot\underline e'_X$ by construction and we can identify $\uX'$ with $\uG'/\uGamma'$.
	
 By \cite[Section 5]{BHK} (see also~\cite{Zi05}),
$\uG'$ is a closed subgroup of $G^{s+1}$, hence a nilpotent Lie group, and the discrete subgroup $\uGamma'$ of $\uG'$ is cocompact. Therefore,  $\uX'$ is compact and can be identified with the nilmanifold $\uG'/\uGamma'$.

Since $\uG$ and $\uG'$ are Polish groups and $p\colon\uG\to\uG'$ is a continuous bijective homomorphism,  the inverse homomorphism is also continuous. Since $\uGamma'$ is cocompact in $\uG'$, it follows that $\uGamma$ is cocompact in $\uG$, hence $\uX$ is  compact and thus $q\colon\uX\to\uX'$ is a homeomorphism.
\begin{convention}
	In the sequel, we use the isomorphism $p$ to identify $\uG$ with $\uG'$ and $\uGamma$ with $\uGamma'$.  We use the homeomorphism $q$ to identify $\uX=\uG/\uGamma$ with the nilmanifold $\uX'=\uG'/\uGamma'$. We write $\mu_{\uX}$ for the Haar measure of $\uX$.
\end{convention}
\begin{definition}
	$\uX=\uG/\uGamma$ is called the \emph{nilmanifold of arithmetic progressions} in $X$.
\end{definition}


\subsection{The nilsystem of arithmetic progressions}
Since $\uG$ is invariant under the shift $\sigma$ of $G^\Z$ we get that $\uX$ is invariant under   the shift $S$ of $X^\Z$. We  have
\begin{equation}
	\label{eq:Sgphi}
	S(\ug\cdot\ux)=\sigma(\ug)\cdot S\ux, \quad \ux\in \uX, \ \ug\in \uG.
\end{equation}
By~\eqref{eq:Sgphi} the image of the measure $\mu_{\uX}$ under $S$ is invariant under translation by  elements of $\uG$, hence it is equal to  $\mu_{\uX}$. We have thus established that  $(\uX,\mu_{\uX},S)$ is a measure preserving system and
our next goal is to give  $(\uX,S)$ the structure of a nilsystem, called the \emph{nilsystem of arithmetic progressions in $X$}.

We define the group $\hat{\uG}$ to be the semidirect product $\hat{\uG}= G\rtimes_\phi\Z$, where $\phi\colon \Z\to\Aut(G)$ is the homomorphism $n\mapsto \overrightarrow \sigma^{\circ n}$ where
$\sigma^{\circ n}=\sigma\circ\dots\circ\sigma$ ($n$ times).  More explicitly, as a set we have $\hat{\uG}=\uG\times\Z$ and  the multiplication is given by
$$
(\ug,m)\cdot(\uh,n)=(\ug\cdot \sigma^{\circ m}(\uh),m+n), \quad   \ug,\uh\in \uG, \ m,n\in\Z.
$$
Then $\uG\times\{0\}$ is a normal subgroup of $\hat{\uG}$ that we identify with $\uG$.
Since $\uG$ is nilpotent and the automorphism $\sigma$ of $\uG$ is unipotent, it follows that $\hat{\uG}$ is nilpotent~\cite[Proposition~3.9]{leibman1}.
We give  $\hat{\uG}$ the structure of a Lie group by letting $\uG$  be an open subgroup of $\hat{\uG}$.

The group $\hat{\uG}$ acts on $\uX$ by $(\ug,m)\cdot \ux= \ug\cdot S^m\ux$ and this action preserves the Haar measure of $\uX$. Moreover, the stabilizer of $e_{\uX}$ is the discrete cocompact subgroup $\hat{\uGamma}:= \uGamma\rtimes_\phi\Z$ of $\hat{\uG}$ and we can identify $\uX$ with the nilmanifold $\hat{\uG}/\hat{\uGamma}$.
Since the measure $\umu$ is invariant under $S$ and the action of $\uG$, it is invariant under the action of $\hat{\uG}$ and thus coincides with the Haar measure of $\uX$ when identified with $\hat{\uG}/\hat{\uGamma}$.
Finally, with the above identifications, the  transformation $S$ is  the translation by the element $(\ue_G,1)$ of $\hat{\uG}$ and  thus $(\uX,\mu_\uX,S)$ is a nilsystem. The previous discussion leads to the following basic result:
\begin{proposition}\label{P:uX}
	If $X$ is a nilmanifold,  then the system $(\uX,S)$ is topologically isomorphic to  a nilsystem. As a consequence,
	if  $Y=\overline{\{S^n\ux\colon n\in \Z\}}$ for some $\ux\in \uX$,
	then the system $(Y,S)$ is  topologically isomorphic to  a uniquely  ergodic nilsystem.
\end{proposition}
The first claim was established in the previous discussion. The consequence follows, for example,  from \cite[Theorems~2.19 and 2.21]{leibman1}.

\section{Sketch of proof of Tao's identity}\label{A:Tao}
We recall the statement of Theorem~\ref{T:Tao} and briefly sketch its proof following almost entirely \cite{Tao15}. The  only  difference in our presentation, is that our assumption of existence of certain  limits allows  us to perform a partial summation at the beginning of the argument in order to connect the averages we are interested in to the averages treated in \cite{Tao15}.
\begin{proposition}
	Let $\bN=([N_k])_{k\in\N}$ be a sequence of intervals,  $\ell\in \N$,  $a_1,\ldots, a_\ell$  be bounded sequences of complex numbers,   and   $h_1,\ldots, h_\ell\in \Z$.
	Let also $(c_p)_{p\in\P}$ be a bounded sequence of complex numbers.
	Then,
	assuming that on the left and  right hand side below the   limits $\lE_{n\in \bN}$ exist for every $p\in\P$ and the limit  $\E_{p\in\P}$  exists, we have the identity
	\begin{equation}\label{E:wanted}
		\E_{p\in\P}\, c_p\,
		\Big(\lE_{n\in \bN}\,\prod_{j=1}^\ell a_j(pn+ph_j)\Big)=
		\E_{p\in \P}\,
		c_p\,
		\Big(\lE_{n\in \bN}\, \prod_{j=1}^\ell a_j(n+ph_j)\Big).
	\end{equation}
\end{proposition}
\begin{proof}[Sketch of Proof]

For $H\in \N$ let\footnote{In \cite{Tao15} the respective set $\CP_H$ consists of primes on the
	interval  $[\delta H/2, \delta H)$ for a sufficiently small $\delta$, but for our purposes we can take $\delta=1$.}
$$
\CP_{H}:=\{p\in \P\colon  H/2\leq p<H\}, \qquad W_H:=\sum_{p\in \CP_H} \frac{1}{p}\sim  \frac{1}{\log{H}}
$$
where the last asymptotic means that the quotient of the two quantities involved converges to a non-zero constant as $H\to \infty$, and follows from the prime number theorem using partial summation.

We first claim that the limits on the left and right hand side of \eqref{E:wanted}
are equal to
\begin{equation}\label{E:pn}
	\lim_{H\to\infty}		\frac{1}{W_H}\sum_{p\in \CP_H}\,\frac{c_p}{p}\,  \lE_{n\in \bN}\,  \prod_{j=1}^\ell\,a_j(pn+ph_j)
\end{equation}
and
$$
\lim_{H\to\infty}		\frac{1}{W_H}\sum_{p\in \CP_{H}}\,\frac{c_p}{p}\,  \lE_{n\in \bN}\,  \prod_{j=1}^\ell\,a_j(n+ph_j)
$$
respectively.
To see this, let  $$
A(p):= c_p\, \lE_{n\in \bN}\, \prod_{j=1}^\ell\,a_j(pn+ph_j).
$$
Our assumptions give that the limit
	$L:=\E_{p\in \P}A(p)$ exists and we  want to show that $$
	B(H):=\frac{1}{W_H}\sum_{p\in\CP_H}\frac{A(p)}{p}\to L\ \text{ as } \ H\to\infty.
	$$
	 (In a similar manner we treat the second average.)
	 Let $\varepsilon>0$. If  $S(x):=\sum_{p\leq x} (A(p)-L)$, where $x\in \N$, our hypothesis gives that $|S(x)|\leq \varepsilon \frac{x}{\log x}$ for all sufficiently large $x$.  Since $S(x)-S(x-1)$ is equal to $A(x)-L$ if $x$ is a prime and is $0$ otherwise, we get that  for every $H\in \N$ we have
	$$
	B(H)-L=\frac{1}{W_H}\sum_{H/2\leq n< H}\frac{S(n)-S(n-1)}{n}.
	$$ Using partial summation
	we get that $|B(H)-L|$ is bounded by   a sum of terms of the form $S(H)/(HW_H)$ and $\frac{1}{W_H}\sum_{H/2\leq n< H}\frac{S(n)}{n^2}$. For sufficiently large $H\in \N$ the first term is bounded by $\varepsilon$ and the second by $\varepsilon H \sum_{H/2\leq n<H} \frac{1}{n^2}\leq 2 \varepsilon$. This  completes the proof of the claim.

	 Next note the simple but important fact that  if $b\in \ell^\infty(\Z)$, then  for every $r\in \N$ we have\footnote{This identity holds for logarithmic averages and fails in general for Ces\`aro averages, which is the main reason why we cannot treat Ces{\`a}ro averages in this article.}
	  $$
\lE_{n\in \bN} (b(r n)- b(n) \, r\, {\bf 1}_{r\Z}(n))=0.
$$
 Using this for  $r=p$ and
for the sequence $b_p$,  $p\in \P$,  defined by
$$
b_{p}(n):= c_p\, \prod_{j=1}^\ell a_j(n+ph_j), \quad n\in \N,
$$
we can rewrite the limit in \eqref{E:pn} as
$$
\lim_{H\to\infty}
\frac{1}{W_H}\sum_{p\in \CP_H}\,
 	c_p\, \lE_{n\in \bN}\, \prod_{j=1}^\ell a_j(n+ph_j)
\cdot {\bf 1}_{p\Z}(n).
$$

Hence, in order to establish \eqref{E:wanted} and because all relevant limits exist,  it suffices to show that
\begin{equation}\label{E:Goal1}
	 \liminf_{H\to\infty}	\Big| \lE_{n\in \bN}	\,\frac{1}{W_H}\sum_{p\in \CP_H}\,   c_p\, \prod_{j=1}^\ell a_j(n+ph_j)
	\cdot \big({\bf 1}_{p\Z}(n)-p^{-1}\big)\Big|=0.
\end{equation}
We argue by contradiction. Suppose \eqref{E:Goal1} fails for some $h_1,\ldots, h_\ell\in\Z$. Since $W_H\sim \frac{1}{\log{H}}$   there exists $\varepsilon>0$ such that for $\delta:=\varepsilon^2$ (we can choose it any function of $\varepsilon$ we like)  we have (the argument is similar if $\leq -\varepsilon\, \frac{1}{\log{H}}$)
\begin{equation}\label{E:Goal1'}
	\lE_{n\in \bN}\,\sum_{p\in \CP_H}\, c_p\, \prod_{j=1}^\ell a_j(n+ph_j)
	\cdot \big({\bf 1}_{p\Z}(n)-p^{-1}\big) \geq \varepsilon\, \frac{1}{\log{H}}
\end{equation}
for all large enough $H\in \N$.
Using the translation invariance of the average $\lE_{n\in \bN}$  we shift $n$ by $h$ and sum over $h\in [H]$. We get that

\begin{equation}\label{E:Goal2}
	\lE_{n\in [N_k]}\,\sum_{p\in \CP_H}\, \sum_{h\in[H]} c_p\, \prod_{j=1}^\ell a_j(n+h+ph_j)
	\cdot \big({\bf 1}_{p\Z}(n+h)-p^{-1}\big) \geq \varepsilon\, \frac{H}{\log{H}}
\end{equation}
for all large enough $H\in \N$ depending on $\varepsilon$ and all large enough $k$ depending on $\varepsilon$ and $H$.  Furthermore, after approximating the sequences $a_j$, $j=1,\ldots, \ell$,  to the nearest element of the lattice $\varepsilon^2\Z[i]$ we can assume that they take values on a finite set $A=A_\varepsilon$
and \eqref{E:Goal2} continues to hold  (with $\varepsilon/2$ in place of $\varepsilon$). For details
see \cite[Section~2]{Tao15}.


For $k\in \N$, on the space $\N$ we define the (non-shift invariant) probability
measure $\mathbb{P}_k$ on all subsets of $\N$ by letting
$$
\mathbb{P}_k(E):=\lE_{n\in [N_k]} {\bf 1}_E(n), \qquad E\subset \N.
$$
We also define the vector valued random variables ${\bf X}_H\colon \N\to \C^{\ell H}$ and ${\bf Y}_H\colon \N\to \prod_{p\leq H}\Z/ p\Z$ as follows:
$$
{\bf X}_H(n):=(a_{j,h}(n))_{j\in [\ell], h\in [ H]}, \ n\in \N,
\
\text{  where }  \ a_{j,h}(n):=a_j(n+h),
$$
$$
{\bf Y}_H(n):=\big(n \, \, (p) \big)_{p\leq H},\  n\in \N,
$$
where $\big(n \, \, (p) \big)_{p\leq H}$ denotes  the reductions of $n$ modulo the primes $p$ that are less than $H$.
Furthermore,  for $H\in \N$ we let $F_H\colon A^{\ell H}\times \prod_{p\leq H}\Z/ p\Z\to \R$
be defined by
\begin{equation}\label{E:FH}
F_H((x_{j,h})_{j\in [\ell], h\in [L H]}, (r_p)_{p\leq H} ):=
\sum_{p\in \CP_H} \sum_{h\in [H]} \, c_p\, \prod_{j=1}^\ell x_{j, h+ph_j}
\,  ( {\bf 1}_{p\Z}(r_p+h) -p^{-1})
\end{equation}
where $L:=\max_{j=1,\ldots, \ell}(h_j)+1$.
Let also $\E_{k}F$ denote the expectation of a function  $F\colon \N\to \C$ with respect to the probability measure $\P_k$.
Then  \eqref{E:Goal2}  gives that
\begin{equation}\label{E:Goal3}
	|\E_kF_H({\bf X}_H(n), {\bf Y}_H(n))| \geq \varepsilon\, \frac{H}{\log{H}}
\end{equation}
for all large enough $H$ depending on $\varepsilon$ and all large enough $k$
depending on  $\varepsilon$ and $H$.

Using the entropy decrement argument as in \cite[Lemma~3.2]{Tao15} we get that there exist a positive integer   $H_-=H_-(\varepsilon)$  (which can be chosen suitably large depending on $\varepsilon$), a larger positive integer    $H^+=H^+(\varepsilon)$,  and for $k\in \N$ there exist $H_k\in [H_-,H_+]$ such that
$$
	\mathbb{I}_k({\bf X}_{H_k},{\bf Y}_{H_k})\leq \frac{H_k}{\log{H_k}\log\log{H_k}}
$$
for every $k\in \N$ where $\mathbb{I}_k$ is the mutual information function (defined in \cite[Section~3]{Tao15}) with respect to the probability measure $\P_k$. Since the integers $H_k$ belong to the finite interval $[H_-,H_+]$ for every  $k\in\N$,   there exists a fixed  integer $H_0\in [H_-,H_+]$ such that
\begin{equation}\label{E:independence}
	\mathbb{I}_k({\bf X}_{H_0},{\bf Y}_{H_0})\leq \frac{H_0}{\log{H_0}\log\log{H_0}}
\end{equation}
for infinitely many $k\in \N$.  We deduce that  for $H:=H_0$,
\eqref{E:Goal3} and \eqref{E:independence} hold simultaneously for infinitely many $k\in \N$.



Using \eqref{E:independence} one gets as in \cite{Tao15} (using the Pinsker type inequality
\cite[Lemma~3.3]{Tao15} and then the Hoeffding inequality as in \cite[Lemma~3.5]{Tao15})
the following estimate (it corresponds to \cite[Equation~(3.16)]{Tao15})
\begin{equation}\label{E:decoupled}
	\E_{(r_{p})_{p\leq H}\in \prod_{p\leq {H_0}}\Z/ p\Z}\,  \E_kF_{H_0}({\bf X}_{H_0}(n), (r_p)_{p\leq H_0})\geq C \varepsilon \frac{H_0}{\log{H_0}}
\end{equation}
for some $C>0$ and for infinitely many $k\in \N$.
But by \eqref{E:FH} we have
$$
\E_{(r_p)_{p\leq H}\in \prod_{p\leq {H}}\Z/ p\Z}\,  F_H({\bf X}_{H}(n), (r_p)_{p\leq H})=0,
$$
for every $n,H\in \N$. This contradicts \eqref{E:decoupled} and completes the proof.
\end{proof}

\end{document}